\documentclass[11pt]{amsart}
\usepackage{float}
\usepackage{amsfonts}
\usepackage{amsmath}
\usepackage{amssymb}
\usepackage{graphicx}
\usepackage[left=2.5cm, right=2.5cm, top=3cm]{geometry}
\usepackage{mathtools}
\usepackage{amsthm}
\usepackage{latexsym}
\usepackage{fancyhdr}
\usepackage{array}
\usepackage{amscd}
\usepackage{lscape}
\usepackage{tikz}
\usepackage{tikz-cd}
\usepackage{float}
\usepackage{bm}
\usepackage{subfigure}
\usepackage{grffile}
\usepackage{array}
\usepackage[tableposition=above]{caption}
\usepackage{longtable}
\usepackage{booktabs}
\usepackage{adjustbox}
\usepackage{lipsum}
\usepackage{makecell}
\usepackage{multirow}
\newcolumntype{C}[1]{>{\centering\arraybackslash}p{#1}}
\usepackage[toc,page, title, titletoc]{appendix}
\definecolor{navy}{HTML}{2F729C}
\definecolor{red1}{HTML}{FF0000} 
\usepackage[hyperfootnotes=false, colorlinks, linkcolor={blue}, citecolor={magenta}, filecolor={blue}, urlcolor={blue}, plainpages=false, pdfpagelabels]{hyperref}
\newcommand{\Q}{{\mathbb Q}}

\newcommand{\typ}{{\rm typ}}

\newtheorem{theorem}{Theorem}[section]
\newtheorem{lemma}[theorem]{Lemma}
\newtheorem{proposition}[theorem]{Proposition}
\newtheorem{corollary}[theorem]{Corollary}
\newtheorem{mainthm}{Theorem}
\newtheorem{mainprop}[mainthm]{Proposition}
\theoremstyle{definition}
\newtheorem{definition}[theorem]{Definition}

\theoremstyle{remark}

\numberwithin{equation}{section}
\newcommand{\KNT}{strongly-minimal }

\title[Local data of elliptic curves under quadratic twist]{Local data of elliptic curves under quadratic twist}

\author[A. Barrios]{Alexander J. Barrios}
\address{Department of Mathematics, University of St. Thomas, St. Paul, Minnesota 55105, USA}
\email{abarrios@stthomas.edu}

\author[M. Roy]{Manami Roy}
\address{Department of Mathematics, Lafayette College, Easton, PA 18042, USA}
\email{royma@lafayette.edu}

\author[N. Sahajpal]{Nandita Sahajpal}
\address{Department of Data, Media and Design, Nevada State University, Henderson, NV, 89002, USA}
\email{nandita.sahajpal@nevadastate.edu}

\author[D. Tallana]{Darwin Tallana}
\address{Department of Mathematics, University of Colorado, Boulder, Boulder, CO 80309, USA}
\email{Darwin.Tallana@colorado.edu}

\author[B. Tobin]{Bella Tobin}
\address{Department of Mathematics, Agnes Scott College, Decatur, GA 30030 USA}
\email{btobin@agnesscott.edu}

\author[H. Wiersema]{Hanneke Wiersema}
\address{Department of Pure Mathematics and Mathematical Statistics, University of Cambridge, Cambridge CB3 0WB, UK}
\email{hw600@cam.ac.uk}


\begin{document}
\begin{abstract}
Let $K$ be the field of fractions of a complete discrete valuation ring with a perfect residue field. In this article, we investigate how the Tamagawa number of $E/K$ changes under quadratic twist. To accomplish this, we introduce the notion of a \KNT model for an elliptic curve $E/K$, which is a minimal Weierstrass model satisfying certain conditions that lead one to easily infer the local data of $E/K$. Our main results provide explicit conditions on the Weierstrass coefficients of a \KNT model of $E/K$ to determine the local data of a quadratic twist $E^{d}/K$. We note that when the residue field has characteristic $2$, we only consider the special case $K=\mathbb{Q}_{2}$. In this setting, we also determine the minimal discriminant valuation and conductor exponent of $E$ and $E^d$ from further conditions on the coefficients of a \KNT model for $E$.
\end{abstract}

\maketitle
\tableofcontents

\section{Introduction}

Let $E$ be an elliptic curve defined over a number field $F$. The Birch and Swinnerton-Dyer conjecture \cite{BSD} predicts that the rank of the elliptic curve is related to its $L$-function. Moreover, the conjecture includes a precise formula for the leading term of the $L$-function containing various arithmetic data. Next we let $E^{d}$ denote the quadratic twist of $E$ by $d$ for some $d\in F^{\times}/(F^{\times})^{2}$. A natural question is how the factors of $E$ and $E^{d}$ appearing in the leading term of the $L$-function of $E/F$ compare. For instance, work of Pal \cite{Pal} determines how the real periods of $E$ and $E^{d}$ are related in the case when $F=\mathbb{Q}$. In this article, we consider the analogous question of how the local Tamagawa numbers of $E$ and $E^{d}$ compare. 

Recall that for each finite prime $\mathfrak{p}$ of $F$, there is a unique minimal proper regular model of $E$ over the ring of integers $A_\mathfrak{p}$ of the completion $F_{\mathfrak{p}}$ \cite{Neron1964}. For each finite prime $\mathfrak{p}$ of $F$, Tate's Algorithm \cite{Tate1975} yields the following local data of $E/A_{\mathfrak{p}}$: the Kodaira-N\'{e}ron type of the special fiber of the minimal proper regular model, the local conductor exponent (via Ogg's formula \cite{Ogg}), and the local Tamagawa number. For an elliptic curve $E/F$, the global Tamagawa number is defined as $c=\prod_{\mathfrak{p}}c_{\mathfrak{p}}$, where $c_{\mathfrak{p}}$ is the local Tamagawa number of $E/F_{\mathfrak{p}}$. The global Tamagawa number is one of the factors appearing in the leading term of the $L$-function of $E/F$ in the Birch and Swinnerton-Dyer Conjecture. 

More precisely, let $A$ be a complete discrete valuation ring with perfect residue field $\kappa$ of characteristic $p$. Let $K$ denote the field of fractions of $A$, and $v$ be the normalized valuation on $K$. In order to investigate the local Tamagawa number $c_{v}$ of $E/K$, we apply Tate's Algorithm to~$E$. To streamline the process, we introduce the notion of a $v$-\KNT model for $E$. That is, a Weierstrass model for $E/K$ that is a $v$-minimal model such that its Weierstrass coefficients satisfy certain conditions (see Table \ref{pnormalmodels}). Specifically, the conditions on our $v$-\KNT models ensure that $E/K$ runs through Tate's Algorithm without having to do any additional admissible change of variables, and in addition, the models are chosen so that the computation of the local Tamagawa number $c_{v}$ is reduced to verifying additional conditions on the Weierstrass coefficients of a $v$-\KNT model. Our first result establishes that each elliptic curve admits a $v$-\KNT model:

\begin{mainprop}[Proposition~\ref{lem:pnormal}]
\label{mainprop1}Let $K$ be the field of fractions of a complete discrete valuation ring with a perfect residue field of prime characteristic. If an elliptic curve over $K$ is described by a Weierstrass model $E$, then there exists a $v$-\KNT model $E^{\prime}$ such that $E$ and $E^{\prime}$ are $K$-isomorphic. Moreover, there are 
explicit conditions (see Table~\ref{pnormalmodels}) on the Weierstrass coefficients of a $v$-\KNT model to conclude the Kodaira-N\'{e}ron type and local Tamagawa number.
\end{mainprop}

This result was inspired by work of Cremona and Sadek \cite[Proposition~2.2]{CromonaSadek2020}, who consider Weierstrass models of elliptic curves over $\mathbb{Q}_{p}$ such that Tate's Algorithm terminates without requiring any additional change of variables. For our purposes, our $v$-\KNT models are not restricted to $\mathbb{Q}_{p}$, and further conditions have been imposed to ease the computation of the local Tamagawa number. Having established the existence of $v$-\KNT models, we then turn our attention to determining the local Tamagawa number $c_{v}^{d}$ of $E^{d}$ from the Weierstrass coefficients of a $v$-\KNT model for $E$. We note that by Proposition~\ref{mainprop1}, it suffices to determine a $v$-\KNT model of $E^d$.

In Section \ref{Sec:oddp}, we find the local Tamagawa number of $E^{d}$ when the residue characteristic of $A$ is odd. Our main result in this setting is the following:
\begin{mainthm}[Theorem~\ref{tamanot2}]
\label{mainthm1}Let $K$ be the field of fractions of a complete discrete valuation ring with perfect residue field of odd characteristic. Suppose an elliptic curve over $K$ is described by a $v$-\KNT model~$E$. Then, there are 
explicit conditions on the Weierstrass coefficients of~$E$ (see Table~\ref{tab:twistspodd}) to determine the local Tamagawa numbers $c_{v}$ and $c_{v}^{d}$ of $E$ and $E^{d}$, respectively.
\end{mainthm}

The conditions for this theorem are given explicitly in the body of the text. Theorem \ref{mainthm1} also gives conditions to determine the Kodaira-N\'{e}ron type of $E^{d}$. Since $E$ is assumed to be given by a $v$-\KNT model, the conditions on the local Tamagawa number of $E$ are given by Lemma~\ref{lem:pnormal}. Consequently, to prove Theorem~\ref{tamanot2}, it suffices to determine the local Tamagawa number $c_{v}^{d}$ of $E^{d}$. Since $E/K$ is given by a $v$-\KNT model, it has a Weierstrass model of the form $E:y^{2}=x^{3}+a_{2}x^{2}+a_{4}x+a_{6}$. We may, therefore, assume $E^{d}$ to be given by the Weierstrass model $y^{2}=x^{3}+a_{2}dx^{2}+a_{4}d^{2}x+a_{6}d^{3}$. The proof of Theorem~\ref{tamanot2} is then reduced to determining a $v$-\KNT model for $E^d$ by Proposition~\ref{mainprop1}.

Finally, we consider the case when the residue characteristic is $2$. This case is more complicated, and as a result, we restrict ourselves to the special case when $K=\Q_2$. Specifically, we prove:

\begin{mainthm}[Theorem \ref{thmQ2combined}]
\label{mainthm2} Suppose an elliptic curve over $\Q_2$ is described by a $v$-\KNT model~$E$. Then, there are
explicit conditions on the Weierstrass coefficients of $E$ (see Tables~\ref{tab:localdata-dodd} and~\ref{tab:localdata-deven}) to determine the local Tamagawa numbers $c_{v}$ and $c_{v}^{d}$ of $E$ and $E^{d}$, respectively.
\end{mainthm}

Theorem \ref{mainthm2} further provides conditions to determine the Kodaira-N\'{e}ron type of $E^{d}$, as well as the conductor and minimal discriminant exponents of $E$ and $E^{d}$. Due to the length of the proof, we consider the cases corresponding to $v(d)=0$ and $v(d)=1$ separately. See Sections~\ref{sec5_1} and \ref{sec5_2}, respectively, for the proof of these cases.

In what follows, we will work over local fields where the normalized valuation $v$ will be clear from context. Consequently, we omit subscripts for quantities such as $c_v$ and simply write $c$. We note that the explicit determination of the Kodaira-N\'{e}ron types of $E$ and $E^{d}$ in the setting of a complete discrete valuation ring with perfect residue field of prime characteristic was considered by Comalada \cite{Comalada}. This work was motivated by earlier work of Silverman \cite{SilMin}, which asked how the minimal discriminant valuation of $E$ and $E^d$ are related in the case when $v(d)=1$ and either $E$ has bad reduction or the residue characteristic is $2$. With regards to this question, Silverman states that \textit{any proof which computes directly with Weierstrass equations will need to consider a large number of cases.} This was indeed the case in Comalada's work, where he approached the question via Weierstrass equations. Similarly, our treatment of the Tamagawa numbers of $E$ and $E^d$ results in many cases to consider as we approach the question by focusing on a \KNT model for $E$. In the process of proving Theorem~\ref{mainthm2}, we uncovered several errors in \cite{Comalada} in the case when the residue characteristic is $2$. Theorem~\ref{mainthm2} corrects the case when $K=\Q_2$. More specifically, the following possible Kodaira-N\'{e}ron types after quadratic twists are missing in \cite[Tables~I~and~II]{Comalada}: $\rm{I}_0 \rightarrow \rm{I}_0, \rm{I}_n \rightarrow \rm{I}_n, \rm{IV} \rightarrow \rm{IV}, \rm{IV}^* \rightarrow \rm{IV}^* $ and a few cases of $\rm{I}_n^*$ when $v(d)=0$, and $\rm{I}_0 \rightarrow \rm{I}_0, \rm{I}_n \rightarrow \rm{I}_n$ when $v(d)=1$. We further note that the authors have conducted a literature review of articles citing~\cite{Comalada}, and it appears that no articles have used the incorrect results from the aforementioned work.
\section{Preliminaries}

We commence with some basic facts about elliptic curves. See \cite{Silverman2009,Silverman1994} for further details. Let $K$ be the field of fractions of a complete discrete valuation ring $A$ with perfect residue field $\kappa$. In what follows, we let $\pi$ be a uniformizer for $A$, and denote by $v$ the normalized valuation on $K$. Let $E/K$ be the elliptic curve given by the (affine) Weierstrass model
\[
E:y^{2}+a_{1}xy+a_{3}y=x^{3}+a_{2}x^{2}+a_{4}x+a_{6},
\]
where each $a_{i}\in K$. We then define
\[%
\begin{tabular}
[c]{lllll}%
$b_{2}=a_{1}^{2}+4a_{2},$ & $\qquad$ & $b_{4}=2a_{4}+a_{1}a_{3},$ &  &
$b_{6}=a_{3}^{2}+4a_{6},$\\
\multicolumn{3}{l}{$b_{8}=a_{1}^{2}a_{6}+4a_{2}a_{6}-a_{1}a_{3}a_{4}%
+a_{2}a_{3}^{2}-a_{4}^{2},$} & $\qquad$ & $\Delta=9b_{2}b_{4}b_{6}-b_{2}%
^{2}b_{8}-8b_{4}^{3}-27b_{6}^{2}.$
\end{tabular}
\]
We say that $E$ is $K$-isomorphic to an elliptic curve $E^{\prime}$ if there exists a morphism $\psi:E\rightarrow E^{\prime}$ defined by $\psi(x,y)=\left(  u^{2}x+r,u^{3}y+u^{2}sx+w\right)  $, where $u,r,s,w\in K$ with $u\neq0$. In this case, we write $\psi=\left[  u,r,s,w\right]  $. We say that $E$ is given by an \textit{integral Weierstrass model} if $v(a_{i})\geq0$ for each~$i$. We say that $E$ is given by a \textit{minimal model} if $E$ is given by an integral Weierstrass model that has the property that $v(\Delta)$ is minimal over all elliptic curves given by integral Weierstrass models that are $K$-isomorphic to $E$.

In our investigation, we will consider elliptic curves $E/K$ whose Weierstrass coefficients $a_i$ satisfy certain conditions. To ease presentation, we utilize the notation used in \cite{CromonaSadek2020} and set
\[
\mathcal{V}(E)=\left(  n_{1},n_{2},n_{3},n_{4},n_{6}\right)
\]
if $v(a_{i})\geq n_{i}$ for each $i$. If equality holds, say at $n_{2}$ and
$n_{4}$, then we write%
\[
\mathcal{V}(E)=\left(  n_{1},=n_{2},n_{3},=n_{4},n_{6}\right)  .
\]
If an additional condition must hold, say $v(\Delta)=6$, then we write
\[
\mathcal{V}(E)=\left(  n_{1},n_{2},n_{3},n_{4},n_{6}|v(\Delta)=6\right)  .
\]

Tate's Algorithm \cite{Tate1975} provides instructions for computing the proper regular minimal model $\mathcal{C}$ of $E$ over $A$. In particular, we obtain the following \textit{local data} for $E$:

\begin{enumerate}
\item The Kodaira-N\'{e}ron type of $E$, denoted $\operatorname*{typ}(E)$. We use Kodaira symbols to describe $\operatorname*{typ}(E)$;

\item The number of components $m$ (counted without multiplicity) defined over $\overline{\kappa}$ of the special fiber of $\mathcal{C}$;

\item The local Tamagawa number $c$ of $E$, which is the number of components of the special fiber of $\mathcal{C}$ that have multiplicity $1$ and are defined over $\kappa$.
\end{enumerate}

From Ogg's formula \cite{Ogg}, we obtain the conductor exponent $f=v(\Delta
)-m+1$, where $\Delta$ denotes the discriminant of a minimal model for $E$. In this article, whenever we refer to step $n$ of Tate's Algorithm, we are referring to step $n$ as given in Silverman's exposition \cite{Silverman1994}. In addition, we use the shorthand notation:
$$a_{i,j}=\frac{a_i}{\pi^j}.$$

Now suppose that $E/K$ and $E^{\prime}/K$ are two elliptic curves such that their $j$-invariants are equal. Then $E$ and $E^{\prime}$ are $\overline{K}$-isomorphic, and we say that $E$ is a \textit{twist} of $E^{\prime}$. In particular, by \cite[Proposition 5.4, Section X.5]{Silverman2009}, there exists a $d\in K^{\times}/\left(  K^{\times}\right)  ^{m}$ where
\[
m=\left\{
\begin{array}
[c]{cl}%
2 & j(E)\neq0, 1728,\\
4 & j(E)=1728,\\
6 & j(E)=0,
\end{array}
\right.
\]
such that $E$ and $E^{\prime}$ are $K(\sqrt[m]{d})$-isomorphic. In this article, we will focus on the case when $m=2$, i.e., $E^{\prime}$ is a \textit{quadratic twist} of $E$ by $d$. In this case, we write $E^{\prime}=E^{d}$. Now suppose that $E/K$ has the Weierstrass model
\[
E:y^{2}+a_{1}xy+a_{3}y=x^{3}+a_{2}x^{2}+a_{4}x+a_{6}.
\]
Then, following \cite[pp. 48-49]{Comalada} with $\alpha=0$, we can take the following as a Weierstrass model for the quadratic twist $E^{d}$ of $E$ by $d$:
\[
E^{d}:y^{2}=x^{3}+d\left(  a_{1}^{2}+4a_{2}\right)  x^{2}+d^{2}\left(
8a_{1}a_{3}+16a_{4}\right)  x+d^{3}\left(  16a_{3}^{2}+64a_{6}\right)  .
\]

\section{\texorpdfstring{$v$}{v}-\KNT models}
In this section, we introduce the notion of a $v$-\KNT model. When the normalized valuation $v$ is clear from context, we simply refer to these as \KNT models. To begin, we recall the following facts regarding complete discrete valuation rings whose residue characteristic is~$2$. 

\begin{lemma}
\label{Lemmachar2squares}Let $A$ be a complete discrete valuation ring with
perfect residue field $\kappa$ of characteristic~$2$ and uniformizer $\pi$.
For each $x\in A^{\times}$, there exists a $z\in A^{\times}$ such that
$z^{2}\equiv x\ \operatorname{mod}\pi$.
\end{lemma}

\begin{proof}
Since $\kappa$ is a perfect field of characteristic $2$, the Frobenius endomorphism $f:\kappa\rightarrow\kappa$ defined by $f(y)=y^{2}$ is an automorphism of $\kappa$. Consequently, for $x\in A^{\times}$, there exists $z\in A$ such that $z^{2}\equiv x\ \operatorname{mod}\pi$.
\end{proof}

\begin{lemma}\label{char2splitpoly}
Let $\kappa$ be a field of characteristic $2$, and let $a,b,c\in\kappa$ with $ab\neq0$. Then the quadratic polynomial $ax^{2}+bx+c$ splits completely in $\kappa\lbrack x]$ if and only if $\frac{ac}{b^{2}}\in\operatorname{Im}T$, where $T:\kappa\rightarrow\kappa$ is the Artin-Schreier endomorphism defined by $T\!\left(  \alpha\right)  =\alpha^{2}+\alpha$.
\end{lemma}

\begin{proof}
Let $y=axb^{-1}$. The result now follows since $ax^{2}+bx+c=0$ is equivalent
to $y^{2}+y=acb^{-2}$.
\end{proof}

\begin{lemma}
\label{Lem:Inmodel}
Let $K$ be the field of fractions of a complete discrete valuation ring $A$ with perfect residue field of characteristic $2$ and normalized valuation $v$. Suppose that an elliptic curve over $K$ is given by a Weierstrass model $E$. If $\operatorname*{typ}(E)=\mathrm{{I}_{n}}$ for some $n>0$, then
\begin{itemize}
    \item[$(i)$] if $\mathcal{V}(E)=\left(  =0,0,\frac{n+1}{2},\frac{n+1}{2},0\right)  $, then $v(a_{6})=n$.
\item[$(ii)$]  if $\mathcal{V}(E)=\left(=0,0,\frac{n+2}{2},0,n+1\right)  $ with $n$ even, then $v(a_{4})=\frac{n}{2}$.
\end{itemize}

Conversely, suppose $\left(  i\right)  \ \mathcal{V}(E)=\left(  =0,0,\frac
{n+1}{2},\frac{n+1}{2},=n\right)  $ or $\left(  ii\right)  \ \mathcal{V}%
(E)=\left(  =0,0,\frac{n+2}{2},=\frac{n}{2},n+1\right)  $ with $n$ even. Then
$\operatorname*{typ}(E)=\mathrm{{I}_{n}}$.
\end{lemma}

\begin{proof}
It is verified in \cite[Lemma3\_3.ipynb]{gittwists} that the discriminant $\Delta$ of $E$ can be written as
\begin{equation}
\Delta=a_{3}^{2}k_{1}+a_{4}^{2}k_{2}+a_{3}a_{4}k_{3}+a_{6}k_{4}%
,\label{InDiscvaluation}
\end{equation}
where each $k_{i}\in A$ with $v(k_{2})=v(k_{3})=v(k_{4})=0$. Now suppose that $\operatorname*{typ}(E)=\mathrm{{I}_{n}}$, so that $v(\Delta)=n$. If $\mathcal{V}(E)=\left(  =0,0,\frac{n+1}{2},\frac{n+1}{2},0\right)  $, then by (\ref{InDiscvaluation}), $v(\Delta)=n$ if and only if $v(a_{6})=n$. Similarly, if $n$ is even and $\mathcal{V}(E)=\left(  =0,0,\frac{n+2}{2},0,n+1\right)  $, then by (\ref{InDiscvaluation}), $v(\Delta)=n$ if and only if $v(a_{4})=\frac{n}{2}$.

Conversely, suppose $\left(  i\right)  \ \mathcal{V}(E)=\left(  =0,0,\frac{n+1}{2},\frac{n+1}{2},=n\right)  $ or $\left(  ii\right)  \ \mathcal{V}(E)=\left(  =0,0,\frac{n+2}{2},=\frac{n}{2},n+1\right)  $ with $n$ even. Then, for a uniformizer $\pi$ of $A$, $b_{2}=a_{1}^{2}+4a_{2}\equiv a_{1}^{2}\ \operatorname{mod}\pi$. Since $v(a_{1})=0$, we conclude that $v(b_{2})=0$. By Tate's Algorithm, $\operatorname*{typ}(E)=\mathrm{{I}}_{v(\Delta)}$. From our assumptions on $\mathcal{V}(E)$ and~(\ref{InDiscvaluation}), we conclude that $v(\Delta)=n$.
\end{proof}

\begin{definition}
Let $K$ be the field of fractions of a complete discrete valuation ring $A$ with perfect residue field $\kappa$ of prime characteristic and normalized valuation $v$. An elliptic curve over $K$ described by a Weierstrass model $E$ is said to be given by a \textit{$v$-\KNT model} if $\mathcal{V}(E)$ occurs in the third column of Table~\ref{pnormalmodels}.
\end{definition}

\begin{proposition}
\label{lem:pnormal}
Let $K$ be the field of fractions of a complete discrete valuation ring $A$ with perfect residue field $\kappa$ of prime characteristic and normalized valuation $v$. Suppose an elliptic curve over $K$ is described by a Weierstrass model $E$. For each Kodaira-N\'{e}ron type $\operatorname*{typ}(E)$, there exists a $K$-isomorphic elliptic curve $E^{\text{sm}}$ that is a minimal model and whose Weierstrass coefficients satisfy the valuations $\mathcal{V}(E^{\text{sm}})$ given in Table~\ref{pnormalmodels}. The Tamagawa number $c$ is then as given in the fifth column. Moreover, by $T$ we mean the Artin-Schreier endomorphism $T:\kappa \rightarrow \kappa$ from Lemma~\ref{char2splitpoly}. The table also references the polynomial $P(X)=X^{3}+a_{2,1}X^{2}+a_{4,2}dX+a_{6,3}.$

Conversely, if an elliptic curve over $K$ is described by a Weierstrass model $E$ such that $\mathcal{V}(E)$ occurs in the third column of Table~\ref{pnormalmodels}, then $E$ is given by a minimal model and $\operatorname*{typ}(E)$ is as listed in the corresponding row.
\end{proposition}

{\begingroup \footnotesize
\renewcommand{\arraystretch}{1.5}
 \begin{longtable}{lllcc}
 	\caption{$v$-\KNT models and conditions to determine Tamagawa number~$c$}\\
	\hline
$\operatorname*{typ}(E)$ & $\operatorname*{char}\kappa$ & $\mathcal{V}(E)$ &
Conditions & $c$\\
	\hline
	\endfirsthead
	\hline
$\operatorname*{typ}(E)$ & $\operatorname*{char}\kappa$ & $\mathcal{V}(E)$ &
Conditions & $c$\\
	\hline
	\endhead
	\hline 
	\multicolumn{3}{r}{\emph{continued on next page}}
	\endfoot
	\hline 
	\endlastfoot
$\rm{I}_{0}$ & $2$ & $\left(  =0,0,2,1,=0\right)  $ &  & $1$\\\cmidrule{3-5}
&  & $\left(  =0,0,2,=0,1\right)  $ &  & $1$\\\cmidrule{3-5}
&  & $\left(  1,1,=0,0,0\right)  $ &  & $1$\\\cmidrule{2-5}
& $\neq2$ & $\left(  \infty,0,\infty,0,0|v(\Delta)=0\right)  $ &  & $1$\\\hline
$\rm{I}_{n>0\text{ odd}}$ & $2$ & $\left(  =0,0,=\frac{n+1}{2},\frac{n+1}{2},=n\right)  $ & $a_1^{-2}a_2 \in \operatorname{Im}T$ & $n$\\\cmidrule{4-5}
 &  & & $a_1^{-2}a_2 \not \in \operatorname{Im}T$ & $1$\\\hline

$\rm{I}_{n>0\text{ even}}$ & $2$ & $\left(  =0,0,\frac{n+2}{2},=\frac{n}{2},n+1\right)  $ & $a_1^{-2}a_2 \in \operatorname{Im}T$ & $n$\\\cmidrule{4-5}
 &  & & $a_1^{-2}a_2 \not \in \operatorname{Im}T$ & $2$\\\hline

$\rm{I}_{n>0}$ & $\neq2$ & $\left(  \infty,=0,\infty,\frac{n+3}{2},=n\right)  $ & $\left(
\frac{a_{2}}{\pi}\right)  =1$ & $n$\\\cmidrule{4-5}
&  &  & $\left(  \frac{a_{2}}{\pi}\right)  =-1$ & $2-(n\ \operatorname{mod}2)$\\\hline

$\rm{II}$ & $2$ & $\left(  1,1,1,1,=1\right)  $ &  & $1$\\\cmidrule{2-5}
& $\neq2$ & $\left(  \infty,1,\infty,1,=1\right)  $ &  & $1$\\\hline
$\rm{III}$ & $2$ & $\left(  1,1,1,=1,2\right)  $ &  & $2$\\\cmidrule{2-5}
& $\neq2$ & $\left(  \infty,1,\infty,=1,2\right)  $ &  & $2$\\\hline
$\rm{IV}$ & $2$ & $\left(  1,1,=1,2,2\right)  $ & $a_3^{-2}a_6 \not \in \operatorname{Im}T$ & $1$\\\cmidrule{4-5}
 & & & $a_3^{-2}a_6 \in \operatorname{Im}T$ & $3$ \\\cmidrule{2-5}

& $\neq2$ & $\left(  \infty,1,\infty,2,=2\right)  $ & $\left(  \frac{a_{6,2}}{\pi}\right)  =-1$ & 1\\\cmidrule{4-5}
&  &  & $\left(  \frac{a_{6,2}}{\pi}\right)  =1$ & 3\\\hline

$\rm{I}_{0}^{\ast}$ & $2$ & $\left(  1,1,2,3,=3\right)  $ &  &$1+\#\left\{  \alpha\in\kappa\mid P(\alpha)=0\right\}  $ \\\cmidrule{2-5}
& $\neq2$ & $\left(  \infty,1,\infty,2,3\mid v(\Delta)=6\right)  $ &  &
$1+\#\left\{  \alpha\in\kappa\mid P(\alpha)=0\right\}  $\\\hline
$\rm{I}_{n>0\text{ odd}}^{\ast}$ & $2$ & $\left(  1,=1,=\frac{n+3}{2},\frac{n+5}%
{2},n+3\right)  $ & $a_3^{-2}a_6 \not \in \operatorname{Im}T$  & $2$ \\\cmidrule{4-5}
&  &  & $a_3^{-2}a_6 \in \operatorname{Im}T$ & $4$ \\\cmidrule{2-5}
& $\neq2$ & $\left(  \infty,=1,\infty,\frac{n+5}{2},=n+3\right)  $ & $\left(
\frac{a_{6,n+3}}{\pi}\right)  =-1$ & $2$\\\cmidrule{4-5}
&  &  & $\left(  \frac{a_{6,n+3}}{\pi}\right)  =1$ & $4$\\\hline
$\rm{I}_{n>0\text{ even}}^{\ast}$ & $2$ & $\left(  1,=1,\frac{n+4}{2},=\frac
{n+4}{2},n+3\right)  $ & $\pi a_4^{-2}a_6 \not \in \operatorname{Im}T$ & $2$ \\\cmidrule{4-5}
&  &  & $\pi a_4^{-2}a_6 \in \operatorname{Im}T$ & $4$ \\\cmidrule{2-5}
& $\neq2$ & $\left(  \infty,=1,\infty,\frac{n+6}{2},=n+3\right)  $ & $\left(
\frac{-a_{2}^{-1}a_{6,n+2}}{\pi}\right)  =-1$ & $2$\\\cmidrule{3-5}
&  &  & $\left(  \frac{-a_{2}^{-1}a_{6,n+2}}{\pi}\right)  =1$ & $4$\\\hline
$\rm{IV}^{\ast}$ & $2$ & $\left(  1,2,=2,3,4\right)  $ & $a_3^{-2}a_6 \not \in \operatorname{Im}T$ & $1$\\\cmidrule{4-5}
&  &  & $a_3^{-2}a_6 \in \operatorname{Im}T$ & $3$ \\\cmidrule{2-5}
& $\neq2$ & $\left(  \infty,2,\infty,3,=4\right)  $ & $\left(  \frac{a_{6,4}}{\pi}\right)  =-1$ & $1$\\\cmidrule{4-5}

&&&$\left(  \frac{a_{6,4}}{\pi}\right)  =1$&$3$ \\\hline
$\rm{III}^{\ast}$ & $2$ & $\left(  1,2,3,=3,5\right)  $ &  &$2$ \\\cmidrule{2-5}
& $\neq2$ & $\left(  \infty,2,\infty,=3,5\right)  $ &  & $2$\\\hline
$\rm{II}^{\ast}$ & $2$ & $\left(  1,2,3,4,=5\right)  $ &  &$1$ \\\cmidrule{2-5}
& $\neq2$ & $\left(  \infty,2,\infty,4,=5\right)  $ &  &$1$
\label{pnormalmodels}
\end{longtable}\endgroup}

\begin{proof}
Suppose an elliptic curve over $K$ is described by a minimal Weierstrass model $E:y^{2}+a_{1}xy+a_{3}y=x^{3}+a_{2}x^{2}+a_{4}x+a_{6}$. Let $\Delta$ denote the minimal discriminant
of $E$ and let $\pi$ be a uniformizer for $K$. When $\operatorname*{char}%
\kappa\neq2$, we may further assume that $a_{1}=a_{3}=0$ by instead
considering the $K$-isomorphic elliptic curve obtained from $E$ via the
isomorphism $\left[  1,0,\frac{-a_{1}}{2},\frac{-a_{3}}{2}\right]  $.
Throughout the proof, we will consider various admissible change of variables
resulting in $K$-isomorphic elliptic curves $E^{\prime}$ and $E^{\prime\prime
}$. We denote the Weierstrass coefficients of these elliptic curves by
$a_{i}^{\prime}$ and $a_{i}^{\prime\prime}$, respectively. The various claims
on the Weierstrass coefficients of $E^{\prime}$ (or $E^{\prime\prime}$) were
verified in SageMath \cite{sagemath}, and the reader can find this
verification in the companion to this proof \cite[Proposition3\_4.ipynb]{gittwists}. Below, we
consider each of the possible Kodaira-N\'{e}ron types separately. For each
case, we consider subcases corresponding to whether $\operatorname*{char}%
\kappa$ is even or odd. When $\operatorname*{char}\kappa=2$, we set
$T:\kappa\rightarrow\kappa$ to be the Artin-Schreier endomorphism. We note
that the verification of the converse is omitted, as this is easily checked
via a case-by-case analysis using Tate's Algorithm.

\textbf{Case 1.} Suppose $\operatorname*{typ}(E)=\mathrm{I}_{0}$. If
$\operatorname*{char}\kappa\neq2$, then there is nothing to show as we may
assume that $\mathcal{V}(E)=\left(  \infty,0,\infty,0,0\right)  $ with
$v(\Delta)=0$. For $\operatorname*{char}\kappa=2$, we consider two subcases.

\qquad\textbf{Subcase 1a.} Suppose $\operatorname*{char}\kappa=2$ and
$v(a_{1})=0$. Then $a_{1}^{-1}\in A$, and the isomorphism $\left[
1,3a_{1}^{-1}a_{3},0,0\right]  $ applied to $E$ results in an elliptic curve
$E^{\prime}$ with $a_{1}^{\prime}=a_{1}$ and $a_{3}^{\prime}=4a_{3}$. So
without loss of generality, we may assume that $\mathcal{V}(E)=\left(
=0,0,2,0,0\right)  $. It is now verified that
\[
\Delta\equiv a_{1}^{4}\left(  a_{1}^{2}a_{6}+a_{4}^{2}\right)
\ \operatorname{mod}\pi.
\]
Thus, $v(\Delta)=0$ if and only if $v(a_{1}^{2}a_{6}+a_{4}^{2})=0$. In
particular, if $v(a_{4})\geq1$, then $v(a_{6})=0$, and if $v(a_{6})\geq1$,
then $v(a_{4})=0$.

Now suppose $v(a_{4}a_{6})=0$, and let $E^{\prime}$ be the elliptic curve
obtained from $E$ via the isomorphism $\left[  1,0,0,-a_{1}^{-1}a_{4}\right]
$. Then $a_{3}^{\prime}=a_{1}^{-1}\left(  a_{1}a_{3}-2a_{4}\right)  $,
$a_{4}^{\prime}=2a_{4}$, and $a_{6}^{\prime}=a_{1}^{2}\left(  a_{1}a_{3}%
a_{4}+a_{1}^{2}a_{6}-a_{4}^{2}\right)  $. By our assumptions, we have that
$v(a_{6}^{\prime})=0$ and thus, $\mathcal{V}(E^{\prime})=\left(
=0,0,1,1,=0\right)  $. If $v(a_{3}^{\prime})\geq2$, then we are done. So
suppose that $v(a_{3}^{\prime})=1$. Next, consider the elliptic curve
$E^{\prime\prime}$ obtained from $E^{\prime}$ via the isomorphism
$[1,a_{1}^{-1}a_{3}^{\prime},0,0]$. Then $a_{3}^{\prime\prime}=2a_{3}^{\prime
}$ and $a_{4}^{\prime\prime}=a_{1}^{-2}\left(  2a_{1}a_{2}a_{3}^{\prime}%
+a_{1}^{2}a_{4}^{\prime}+3a_{3}^{\prime2}\right)  $. In particular,
$\mathcal{V}(E^{\prime\prime})=\left(  =0,0,2,1,0\right)  $.

\qquad\textbf{Subcase 1b.} Suppose $\operatorname*{char}\kappa=2$ and
$v(a_{1})\geq1$. Then $\Delta\equiv a_{3}^{4}\ \operatorname{mod}\pi$, which
implies that $v(a_{3})=0$. Now let $E^{\prime}$ be the elliptic curve obtained
from $E$ via the isomorphism $\left[  1,a_{2},0,0\right]  $. Then
$\mathcal{V}(E^{\prime})=\left(  1,1,=0,0,0\right)  $.

\textbf{Case 2.} Suppose $\operatorname*{typ}(E)=\mathrm{I}_{n}$ for $n>0$.
By Tate's Algorithm, we have that $v(b_{2})=\left(  a_{1}^{2}+4a_{2}\right)
=0$ and $v(\Delta)=n$.

\qquad\textbf{Subcase 2a.} Suppose $\operatorname*{char}\kappa=2$. By Tate's Algorithm, we may assume that $\mathcal{V}(E)=\left(  0,0,1,1,1\right)  $. Since $v(b_{2})=0$, we deduce that $a_{1}\in A^{\times}$. We may then assume that $a_{3}=0$ after considering the $K$-isomorphic elliptic curve obtained from $E$ via the isomorphism $\left[  1,-a_{1}^{-1}a_{3},0,0\right]  $. Since $a_{1}^{\prime}=a_{1}$, we may assume that $\mathcal{V}(E)=\left(=0,0,\infty,1,1\right)  $. We first claim that there is an elliptic curve $E^{\prime}$ that is $K$-isomorphic to $E$ such that $\mathcal{V}(E^{\prime })=\left(  =0,0,\frac{n+1}{2},\frac{n+1}{2},=n\right)  $. 

If $v(a_{4})\geq\frac{n+1}{2}$, then Lemma \ref{Lem:Inmodel} implies that $v(a_{6})=n$. So suppose $v(a_{4})=k<\frac{n+1}{2}$, and let $E^{\prime}$ be the elliptic curve obtained from $E$ via the isomorphism $\left[1,0,0,a_{1}^{-1}a_{4}\right]  $. Then, $a_{1}^{\prime}=a_{1},$ $a_{3}^{\prime}=2a_{1}^{-1}a_{4}$, and $a_{4}^{\prime}=0$. In particular, $\mathcal{V}(E^{\prime})=\left(  =0,0,=k+v(2),\infty,1\right)  $. If $v(2)+k\geq\frac{n+1}{2}$, then Lemma \ref{Lem:Inmodel} implies that $v(a_{6}^{\prime})=n$, which shows that $E^{\prime}$ is our desired claimed model.

By the above, for the claim, it remains to consider the case when $\mathcal{V}(E)=\left(  =0,0,=k+v(2),\infty,1\right)  $ with $v(2)+k<\frac{n+1}{2}$. In this case, we consider the elliptic curve $E^{\prime}$ obtained from $E$ via the isomorphism $\left[  1,-a_{1}^{-1}a_{3},0,0\right]  $. Then $a_{1}^{\prime}=a_{1},$ $a_{3}^{\prime}=0,$ and $a_{4}^{\prime}=a_{1}^{-2}a_{3}\left(  3a_{3}-2a_{1}a_{2}\right)  $. Consequently, $\mathcal{V}(E^{\prime})=\left(  =0,0,\infty,k+2v(2),1\right)  $. If $v(a_{4}^{\prime})\geq\frac{n+1}{2}$, then $E^{\prime}$ is our claimed elliptic curve since $v(a_{6}^{\prime})=n$ by Lemma \ref{Lem:Inmodel}. If $v(a_{4}^{\prime})<\frac{n+1}{2}$, we proceed as above until we obtain a $K$-isomorphic elliptic curve $E^{\prime\prime}$ such that $\mathcal{V}(E^{\prime\prime})=\left(  =0,0,\frac{n+1}{2},\frac{n+1}{2},=n\right)$.

With the claim established, we can now assume that $\mathcal{V}(E)=\left(=0,0,\frac{n+1}{2},\frac{n+1}{2},=n\right)  $. We now consider the cases corresponding to $n$ being odd or even.

Suppose $n$ is odd. If $v(a_{3})=\frac{n+1}{2}$, then $E$ is given by a \KNT model. So suppose that $v(a_{3})>\frac{n+1}{2}$ and set $r=\pi^{\frac{n+1}{2}}$. Then, the $K$-isomorphism $\left[  1,r,0,0\right]  $ results in the elliptic curve $E^{\prime}$ with Weierstrass coefficients
\[
\left(  a_{1}^{\prime},a_{2}^{\prime},a_{3}^{\prime},a_{4}^{\prime}%
,a_{6}^{\prime}\right)  =\left(  a_{1},a_{2}+3r,a_{1}r+a_{3},a_{4}%
+2a_{2}r+3r^{2},a_{6}+a_{2}r^{2}+r^{3}+a_{4}r\right)  .
\]
By inspection, $\mathcal{V}(E^{\prime})=\left(  =0,0,=\frac{n+1}{2},\frac{n+1}{2},=n\right)  $. Thus, $E^{\prime}$ is our desired \KNT model.

Now suppose that $n$ is even, so that $\mathcal{V}(E)=\left(  =0,0,\frac{n+2}{2},\frac{n+2}{2},=n\right)  $. Since $v(a_{6})=n$, we can write $a_{6}=\pi^{n}\hat{a}_{6}$ for some $\hat{a}_{6}\in A^{\times}$. By Lemma \ref{Lemmachar2squares}, there exists $\hat{w}\in A^{\times}$ with $\hat{w}^{2}\equiv\hat{a}_{6}\ \operatorname{mod}\pi$. Now set $w=\pi^{\frac{n}{2}}\hat{w}$, and consider the elliptic curve $E^{\prime}$ obtained from $E$ via the $K$-isomorphism $\left[  1,0,0,w\right]  $. Then,
\[
E^{\prime}:y^{2}+a_{1}xy+\left(  a_{3}+2w\right)  y=x^{3}+a_{2}x^{2}+\left(
a_{4}-a_{1}w\right)  x+a_{6}-w^{2}+a_{3}w.
\]
Now observe that
\[
\frac{a_{6}-w^{2}}{\pi^{n}}=\hat{a}_{6}-\hat{w}^{2}\equiv0\ \operatorname{mod}%
\pi.
\]
It follows that $E^{\prime}$ is our desired \KNT model since $\mathcal{V}(E)=\left(  =0,0,\frac{n+2}{2},=\frac{n}{2},n+1\right)  $.

The above establishes that we may assume that $E$ is given by a \KNT model satisfying
\[
\mathcal{V}(E)=\left\{
\begin{array}
[c]{cl}
\left(  =0,0,=\frac{n+1}{2},\frac{n+1}{2},=n\right)   & \text{if }n\text{ is odd,}\\
\left(  =0,0,\frac{n+2}{2},=\frac{n}{2},n+1\right)   & \text{if }n\text{ is even.}
\end{array}
\right.
\]
By Tate's Algorithm, we have that $E$ has split multiplicative reduction if and only if $T^{2}+a_{1}T+a_{2}$ splits completely over $\kappa$. By Lemma~\ref{char2splitpoly}, this is equivalent to $a_{1}^{-2}a_{2}\in \operatorname{Im}T$. In this case, $c=n$. If $a_{1}^{-2}a_{2}\not \in \operatorname{Im}T$, then $E$ has non-split multiplicative reduction and $c=2-(n\ \operatorname{mod}2)$.

\qquad\textbf{Subcase 2b.} Suppose $\operatorname*{char}\kappa\neq2$. By
Tate's Algorithm, and the fact that the characteristic is odd, we may assume
that $\mathcal{V}(E)=\left(  \infty,0,\infty,1,1\right)  $. Since $v(b_{2}%
)=0$, we deduce that $v(a_{2})=0$. Next, set $v(a_{4})=k\geq1$. Since
$2a_{2}\in A^{\times}$, there exists $r\in A$ such that $2a_{2}r+a_{4}%
=\pi^{\frac{n+3}{2}}$. Note that $v(r)\geq\min\{k,\frac{n+3}{2}\}$. Now let
$E^{\prime}$ be the elliptic curve obtained from $E$ via the isomorphism
$\left[  1,r,0,0\right]  $. Then%
\[
E^{\prime}:y^{2}=x^{3}+\left(  a_{2}+3r\right)  x^{2}+\left(  \pi^{\frac
{n+3}{2}}+3r^{2}\right)  x+r^{3}+a_{4}r+a_{2}r^{2}+a_{6}.
\]
Hence, $\mathcal{V}(E^{\prime})=\left(  \infty,=0,\infty,\min\{2k,\frac
{n+3}{2}\},1\right)  $. If $v(a_{4}^{\prime})<\frac{n+3}{2}$, we can continue
as above until we obtain an elliptic curve $E^{\prime\prime}$ that is
isomorphic to $E$ such that $\mathcal{V}(E^{\prime\prime})=\left(
\infty,=0,\infty,\frac{n+3}{2},1\right)  $.

By the above, we may assume without loss of generality that $\mathcal{V}%
(E)=\left(  \infty,=0,\infty,\frac{n+3}{2},1\right)  $. Then $\Delta=a_{4}%
^{2}k_{1}+a_{6}k_{2}$ for some $k_{1},k_{2}\in A$ with $v(k_{2})=0$. It
follows that $v(\Delta)=n$ if and only if $v(a_{6})=n$. For the local Tamagawa
number, observe that $E$ has split multiplicative reduction if and only if the
polynomial $T^{2}-a_{2}$ splits in $\kappa$. Equivalently, $\left(
\frac{a_{2}}{\pi}\right)  =1$. In this case, $c=n$. If $E$ has non-split
multiplicative reduction, then $c=2-\left(  n\ \operatorname{mod}2\right)  $.

\textbf{Case 3.} Suppose $\operatorname*{typ}(E)=\mathrm{II}$.

\qquad\textbf{Subcase 3a.} Suppose $\operatorname*{char}\kappa=2$. By Tate's
Algorithm, we may assume that $\mathcal{V}(E)=\left(  0,0,1,1,=1\right)  $
with $v(b_{2})\geq1$. Since $v(b_{2})\geq1$, we must have $v(a_{1})\geq1$. If
$v(a_{2})\geq1$, then $E$ is given by a \KNT model. So, suppose that
$v(a_{2})=0$. By Lemma \ref{Lemmachar2squares}, there exists $s\in A$ such
that $s^{2}\equiv a_{2}\ \operatorname{mod}\pi$. Next, let $E^{\prime}$ be the
elliptic curve obtained from $E$ via the isomorphism $\left[  1,0,s,0\right]
$. Then%
\[
E^{\prime}:y^{2}+\left(  a_{1}+2s\right)  xy+a_{3}y=x^{3}+\left(  a_{2}%
-s^{2}+a_{1}s\right)  x^{2}+\left(  a_{4}-a_{3}s\right)  x+a_{6}.
\]
Since $a_{2}-s^{2}+a_{1}s\equiv0\ \operatorname{mod}\pi$, we have that
$\mathcal{V}(E^{\prime})=\left(  1,1,1,1,=1\right)  $.

\qquad\textbf{Subcase 3b.} Suppose $\operatorname*{char}\kappa\neq2$. By
Tate's Algorithm, and the fact that the characteristic is odd, we may assume
that $\mathcal{V}(E)=\left(  \infty,0,\infty,1,=1\right)  $ with $v(b_{2}%
)\geq1$. Since $v(b_{2})\geq1$, it follows that $v(a_{2})\geq1$. Thus, $E$ is
given by a \KNT model.

\textbf{Case 4.} Suppose $\operatorname*{typ}(E)=\mathrm{III}$. By Tate's Algorithm,
and Case $3$, we may assume that $\mathcal{V}(E)=\left(  1,1,1,1,2\right)  $
(resp. $\left(  \infty,1,\infty,1,2\right)  $) with $v(b_{8})=2$ if
$\operatorname*{char}\kappa=2$ (resp. $\neq2$). Now observe that $v(b_{8})=2$
if and only if $v(a_{4})=1$. Thus, $E$ is given by a \KNT model.

\textbf{Case 5.} Suppose $\operatorname*{typ}(E)=\mathrm{IV}$. By Tate's Algorithm and
Case $4$, we may assume that if $\operatorname*{char}\kappa=2$ (resp. $\neq
2$), then $\mathcal{V}(E)=\left(  1,1,1,2,2\right)  $ (resp. $\left(
\infty,1,\infty,2,2\right)  $) with $v(b_{6})=v(a_{3}^{2}+4a_{6})=2$. In
particular, $v(b_{6})=2$ if and only if $v(a_{3})=1$ (resp. $v(a_{6})=2$).
This shows that $E$ is given by a \KNT model. For the local Tamagawa number,
we observe that $c\in\left\{  1,3\right\}  $. Moreover, $c=3$ if and only if
$T^{2}+a_{3,1}T-a_{6,2}$ splits completely in $\kappa$.

\textbf{Case 6.} Suppose $\operatorname*{typ}(E)=\mathrm{I}_{0}^{\ast}$. From Tate's Algorithm, we may assume that $\mathcal{V}(E)=\left(1,1,2,2,3\right)  $ (resp. $\mathcal{V}(E)=\left(  \infty,1,\infty,2,3\right)$) if $\operatorname*{char}\kappa=2$ (resp. $\neq2$), and that the following polynomial has distinct roots in $\overline{\kappa}$:%
\[
P(X)=X^{3}+a_{2,1}X^{2}+a_{4,2}X+a_{6,3}.
\]
In particular, the discriminant of $P(X)$ satisfies%
\[
\operatorname*{Disc}(P)=\pi^{-6}\left(  a_{2}^{2}a_{4}^{2}-4a_{2}^{3}%
a_{6}-4a_{4}^{3}-27a_{6}^{2}+18a_{2}a_{4}a_{6}\right)  \not \equiv
0\ \operatorname{mod}\pi.
\]
Moreover, $c=1+\#\left\{  \alpha\in\kappa\mid P(\alpha)=0\right\}  $. Note that when $\operatorname*{char}\kappa\neq2$, $v(\Delta)=6$ since
$\Delta=16\operatorname*{Disc}(P)$. So it remains to consider the case when $\operatorname*{char}\kappa=2$.

Suppose $\operatorname*{char}\kappa=2$. Then $\operatorname*{Disc}(P)=\pi
^{-6}\left(  a_{2}^{2}a_{4}^{2}-27a_{6}^{2}\right)  \not \equiv
0\ \operatorname{mod}\pi$. Thus, if $v(a_{4})\geq3$, then $v(a_{6})=3$, which shows that $E$ is given by a \KNT model. So, suppose that $v(a_{4})=2$. By Lemma~\ref{Lemmachar2squares}, there exists a $r\in A$ such that $r^{2}\equiv\frac{a_{4}}{\pi^{2}}\ \operatorname{mod}\pi$. Next, let $E^{\prime}$ be the elliptic curve obtained from $E$ via the isomorphism $\left[  1,\pi
r,0,0\right]  $. Then%
\begin{align*}
E^{\prime}  &  :y^{2}+a_{1}xy+\pi^{2}\left(  a_{1,1}r+a_{3,2}\right)
=x^{3}+\pi\left(  3r+a_{2,1}\right)  x^{2}\\
&  +\pi^{2}\left(  2a_{2,1}r+3r^{2}+a_{4,2}\right)  x+\pi^{3}\left(
r^{3}+a_{2,1}r^{2}+a_{4,2}r+a_{6,3}\right)  .
\end{align*}
In particular, $\mathcal{V}(E^{\prime})=\left(  1,1,2,3,3\right)  $ since
$2a_{2,1}r+3r^{2}+a_{4,2}\equiv4a_{4,2}\ \operatorname{mod}\pi
=0\ \operatorname{mod}\pi$. Since $\operatorname*{typ}(E^{\prime}%
)=\mathrm{{I}_{0}^{\ast}}$, it follows that $v(a_{6}^{\prime})=3$.

\textbf{Case 7.} Suppose $\operatorname*{typ}(E)=\mathrm{I}_{n}^{\ast}$ for
$n>0$. Let
\begin{align*}
f(t)  &  =\left\{
\begin{array}
[c]{cl}%
t^{2}+a_{3,\frac{n+3}{2}}t-a_{6,n+3} & \text{if }n\text{ is odd,}\\
a_{2,1}t^{2}+a_{4,\frac{n+4}{2}}t+a_{6,n+3} & \text{if }n\text{ is even,}%
\end{array}
\right. \\
g(t)  &  =\left\{
\begin{array}
[c]{cl}%
a_{2,1}t^{2}+a_{4,\frac{n+3}{2}}t+a_{6,n+2} & \text{if }n\text{ is odd,}\\
t^{2}+a_{3,\frac{n+2}{2}}t-a_{6,n+2} & \text{if }n\text{ is even.}%
\end{array}
\right.
\end{align*}
By Tate's Algorithm, we may assume that $E$ is given by a model such that
$f(t)$ has distinct roots in $\kappa$ and $g(t)\equiv at^{2}%
\ \operatorname{mod}\pi$ where $a=a_{2,1}$ (resp. $1$) if $n$ is odd (resp.
even). Equivalently, we have that%
\[
\mathcal{V}(E)=\left\{
\begin{array}
[c]{cl}%
\left(  1,=1,\frac{n+3}{2},\frac{n+5}{2},n+3\right)  & \text{if }n\text{ is
odd,}\\
\left(  1,=1,\frac{n+4}{2},\frac{n+4}{2},n+3\right)  & \text{if }n\text{ is
even.}%
\end{array}
\right.
\]
We now proceed by cases.

\qquad\textbf{Subcase 7a}. Suppose $\operatorname*{char}\kappa=2$. Suppose
further that $v(a_{3})>\frac{n+3}{2}$ (resp. $v(a_{4})>\frac{n+4}{2}$) for $n$
odd (resp. even). Then%
\[
f(t)\equiv at^{2}+a_{6,n+3}\ \operatorname{mod}\pi,
\]
where $a=1$ (resp. $a_{2,1}$). By assumption, $f(t)$ has distinct roots in
$\overline{\kappa}$. Thus, $v(a_{6})=n+3$. By Lemma \ref{Lemmachar2squares},
there are $k,l\in A^{\times}$ such that $k^{2}\equiv a\ \operatorname{mod}\pi$
and $l^{2}\equiv a_{6,n+3}\ \operatorname{mod}\pi$. Consequently,%
\[
f(t)\equiv\left(  kt+l\right)  ^{2}\ \operatorname{mod}\pi.
\]
But this is a contradiction, as then $f(t)$ has a double root in $\kappa$.
Therefore, it must be the case that $v(a_{3})=\frac{n+3}{2}$ (resp.
$v(a_{4})=\frac{n+4}{2}$). For the local Tamagawa number, observe that
$c\in\left\{  2,4\right\}  $. Moreover, $c=4$ if and only if $f(t)$ splits
completely in $\kappa$. The result now follows from Lemma \ref{char2splitpoly}.

\qquad\textbf{Subcase 7b.} Suppose $\operatorname*{char}\kappa\neq2$. Let
$E^{\prime}$ be the elliptic curve obtained from $E$ via the isomorphism
$\left[  1,0,\frac{-a_{1}}{2},\frac{-a_{3}}{2}\right]  $. Then%
\[
E^{\prime}:y^{2}=x^{3}+\frac{1}{4}\left(  a_{1}^{2}+4a_{2}\right)  x^{2}%
+\frac{1}{2}\left(  a_{1}a_{3}+2a_{4}\right)  x+\frac{1}{4}\left(  a_{3}%
^{2}+4a_{6}\right)  .
\]
Thus,%
\begin{equation}
\mathcal{V}(E^{\prime})=\left\{
\begin{array}
[c]{cl}%
\left(  \infty,=1,\infty,\frac{n+5}{2},n+3\right)   & \text{if }n\text{ is
odd,}\\
\left(  \infty,=1,\infty,\frac{n+4}{2},n+3\right)   & \text{if }n\text{ is
even.}%
\end{array}
\right.  \label{GenLemIn*vals}%
\end{equation}
So we may suppose that $\mathcal{V}(E)$ satisfies (\ref{GenLemIn*vals}). If
$n$ is odd, then $f(t)\equiv t^{2}-a_{6,n+3}\ \operatorname{mod}\pi$. From
this, we conclude that $v(a_{6})=n+3$ since $f(t)$ has distinct roots in
$\overline{\kappa}$. Moreover, the Tamagawa number $c$ depends on whether
$a_{6,n+3}$ is a square modulo $\pi$.

It remains to consider the case when $n$ is even. If $v(a_{4})>\frac{n+4}{2}$,
then $f(t)$ has distinct roots if and only if $v(a_{6})=n+3$. This shows that
$E$ is given by a \KNT model. So, suppose instead that $v(a_{4})=\frac
{n+4}{2}$, and set%
\[
r=\frac{\pi^{\frac{n+6}{2}}-a_{4}}{2a_{2}}.
\]
Observe that $v(r)=\frac{n+2}{2}$. Next, let $E^{\prime}$ be the elliptic
curve obtained from $E$ via the isomorphism $\left[  1,r,0,0\right]  $. Then%
\[
E^{\prime}:y^{2}=x^{3}+\left(  a_{2}+3r\right)  x^{2}+\left(  \pi^{\frac
{n+6}{2}}+3r^{2}\right)  x+a_{2}r^{2}+r^{3}+a_{4}r+a_{6}.
\]
Thus, $\mathcal{V}(E^{\prime})=\left(  \infty,=1,\infty,\frac{n+6}%
{2},n+3\right)  $. So we may assume that $\mathcal{V}(E)$ satisfies these
valuations. Now observe that $f(t)\equiv a_{2,1}t^{2}+a_{6,n+3}%
\ \operatorname{mod}\pi$. Since $f(t)$ has distinct roots in $\overline
{\kappa}$, we conclude that $v(a_{6})=n+3$.\ Lastly, the Tamagawa number
depends on whether $\frac{-a_{6,n+2}}{a_{2}}$ is a square modulo $\pi$.

\textbf{Case 8.} Suppose $\operatorname*{typ}(E)=\mathrm{IV}^{\ast}$. By Tate's
Algorithm, we may assume that if $\operatorname*{char}\kappa=2$ (resp. $\neq
2$), then $\mathcal{V}(E)=\left(  1,2,2,3,4\right)  $ (resp. $\left(
\infty,2,\infty,3,4\right)  $) with $v(b_{6})=4$. It follows that $E$ is given
by a \KNT model since $v(b_{6})=4$ is equivalent to $v(a_{3})=2$ (resp.
$v(a_{6})=4$). For the Tamagawa number, we note that $c\in\left\{
1,3\right\}  $. Moreover, $c=3$ if and only if $Y^{2}+a_{3,2}Y-a_{6,4}$ splits
completely in $\kappa$. If $\operatorname*{char}\kappa=2$, then the polynomial splits completely if and only if $\frac{a_{6,4}}{a_{3,2}^{2}}=\frac{a_{6}}{a_{3}^{2}}\in\operatorname{Im}T$ by Lemma~\ref{char2splitpoly}. If
$\operatorname*{char}\kappa\neq3$, then the polynomial splits if and only if
$a_{6,4}$ is a square modulo $\pi$.

\textbf{Case 9.} Suppose $\operatorname*{typ}(E)=\mathrm{III}^{\ast}$ (resp. $\mathrm{II}^{\ast
}$). It is easily verified via Tate's Algorithm that $E$ admits a \KNT model.
\end{proof}

\begin{corollary}\label{TamagawaQ2}
Suppose an elliptic curve over $\Q_2$ is given by a \KNT model $E$ such that $\operatorname*{typ}(E)\in\left\{  \rm{I}_{n>0},\rm{IV},\rm{I}_{n\geq0}^{\ast},\rm{IV}^{\ast}\right\}  $. Then, the local Tamagawa number $c$ of $E$ can be determined from the conditions on the Weierstrass coefficients of $E$ as given in Table
\ref{TamaQ2}.
\end{corollary}

{\begingroup \small
\renewcommand{\arraystretch}{1.5}
 \begin{longtable}{lccc}
 	\caption{Local Tamagawa numbers for $E/\Q_2$, with $E$ given by a \KNT model}\\
	\hline
$\operatorname*{typ}(E)$ & $\mathcal{V}(E)$ &
Conditions & $c$\\
	\hline
	\endfirsthead
	\hline
$\operatorname*{typ}(E)$ & $\mathcal{V}(E)$ &
Conditions & $c$\\
	\hline
	\endhead
	\hline 
	\multicolumn{3}{r}{\emph{continued on next page}}
	\endfoot
	\hline 
	\endlastfoot
$\rm{I}_{n>0\text{ odd}}$ & $\left(  =0,0,=\frac{n+1}{2},\frac{n+1}{2},=n\right)  $ & $v(a_2) =0$ & $1$\\\cmidrule{3-4}
 &  & $v(a_2) \ge 1$ & $n$\\\hline

$\rm{I}_{n>0\text{ even}}$ & $\left(  =0,0,\frac{n+2}{2},=\frac{n}{2},n+1\right)  $ & $v(a_2) =0$ & $2$\\\cmidrule{3-4}
 &  & $v(a_2) \ge 1$ & $n$\\\hline

$\rm{IV}$  & $\left(  1,1,=1,2,2\right)  $ & $v(a_6) = 2$ & $1$\\\cmidrule{3-4}
 &  & $v(a_6) \ge 3$ & $3$ \\\hline

$\rm{I}_{0}^{\ast}$  & $\left(  1,1,2,3,=3\right)  $ & $v(a_2)=1 $ & $1$\\\cmidrule{3-4}

 &   & $v(a_2) \ge 2 $ & $2$\\\hline

$\rm{I}_{n>0\text{ odd}}^{\ast}$ & $\left(  1,=1,=\frac{n+3}{2},\frac{n+5}%
{2},n+3\right)  $ & $v(a_6) = n+3$  & $2$ \\\cmidrule{3-4}
&   & $v(a_6) \ge n+4$ & $4$ \\\hline

$\rm{I}_{n>0\text{ even}}^{\ast}$ & $\left(  1,=1,\frac{n+4}{2},=\frac
{n+4}{2},n+3\right)  $ & $v(a_6) = n+3$ & $2$ \\\cmidrule{3-4}
&   & $v(a_6) \ge n+4$ & $4$ \\\hline

$\rm{IV}^{\ast}$ & $\left(  1,2,=2,3,4\right)  $ & $v(a_6) = 4$ & $1$\\\cmidrule{3-4}
&   & $v(a_6) \ge 5$ & $3$
\label{TamaQ2}
\end{longtable}\endgroup}

\begin{proof}
Suppose $\operatorname*{typ}(E)\in\left\{  \rm{I}_{n>0},\rm{IV},\rm{I}_{n>0}^{\ast},\rm{IV}^{\ast}\right\}  $. Since $E$ is given by a \KNT model, we have that the local Tamagawa number is determined by whether the given ratio in Table \ref{pnormalmodels} is in $\operatorname{Im}T$. Since $K=\mathbb{Q}_{2}$, we have that $\kappa=\mathbb{F}_{2}$. Consequently, $\operatorname{Im}T=\left\{  0\right\}  $. The result now follows, since if the ratio is in $\operatorname{Im}T$, then the ratio is $0\ \operatorname{mod}2$. In particular, the valuation of the numerator must be greater than the valuation of the denominator. It remains to consider the case when $\operatorname*{typ}(E)=\rm{I}_{0}^{\ast}$. In this case, $\mathcal{V}(E)=\left(1,1,2,3,=3\right)  $ and $c=1+\#\left\{  \alpha\in\mathbb{F}_{2}\mid P(\alpha)=0\right\}  $, where
\begin{align*}
P(X)  & =X^{3}+a_{2,1}X^{2}+a_{4,2}X+a_{6,3}\\
& \equiv X^{3}+a_{2,1}X^{2}+1\ \operatorname{mod}2.
\end{align*}
Consequently,
\[
c=\left\{
\begin{array}
[c]{cl}%
1 & \text{if }v(a_{2})=1,\\
2 & \text{if }v(a_{2})\geq2. 
\end{array}
\right. \qedhere
\] 
\end{proof}

\begin{lemma}\label{Lemmadisccond}
Suppose an elliptic curve over $\Q_2$ is given by a \KNT model $E$. Then, there are explicit conditions on the Weierstrass coefficients of $E$ to determine the minimal discriminant valuation $\delta$ and conductor exponent $f$, as given in Table
\ref{tab:discriminantconductor}.
\end{lemma}

\begin{proof}
A case-by-case analysis of this proof can be found in \cite[Lemma3\_7.ipynb]{gittwists}. We note that the proof of \cite[Theorem~5.3]{CromonaSadek2020} contains an equivalent computation of this fact.
\end{proof}

{\begingroup \small
\renewcommand{\arraystretch}{1.5}
 \begin{longtable}{ccccccc}
 	\caption{Minimal discriminant valuation and conductor exponent of an elliptic curve $E/\Q_2$ given by a \KNT model}\\
	\hline
$\operatorname*{typ}(E)$ & \multicolumn{4}{c}{Conditions} & $\delta$ & $f$\\
	\hline
	\endfirsthead
	\hline
$\operatorname*{typ}(E)$ & \multicolumn{4}{c}{Conditions} & $\delta$ & $f$\\
	\hline
	\endhead
	\hline 
	\multicolumn{3}{r}{\emph{continued on next page}}
	\endfoot
	\hline 
	\endlastfoot

$\rm{I}_{0}$ &  &  &  &  & $0$ & $0$\\\hline

$\rm{I}_{n>0}$ &  &  &  &  & $n$ & $1$\\\hline

$\rm{II}$ & $v(a_{3})=1$ &  &  &  & $4$ & $4$\\\cmidrule{2-7}
& $v(a_{3})\geq2$ & $v(a_{1})\geq2$ &  &  & $6$ & $6$\\\cmidrule{3-7}
&  & $v(a_{1})=1$ & $v(a_{4})=1$ &  & $7$ & $7$\\\cmidrule{4-7}
&  &  & $v(a_{4})\geq2$ &  & $6$ & $6$\\\hline

$\rm{III}$ & $v(a_{3})=1$ &  &  &  & $4$ & $3$\\\cmidrule{2-7}
& $v(a_{3})\geq2$ & $v(a_{1})=1$ &  &  & $6$ & $5$\\\cmidrule{3-7}
&  & $v(a_{1})\geq2$ & $v(a_{2})=1$ & $v(a_{3}^{2}+4a_{6})=4$ & $9$ & $8$\\\cmidrule{5-7}
&  &  &  & $v(a_{3}^{2}+4a_{6})\geq5$ & $8$ & $7$\\\cmidrule{4-7}
&  &  & $v(a_{2})\geq2$ & $v(a_{3}^{2}+4a_{6})=4$ & $8$ & $7$\\\cmidrule{5-7}
&  &  &  & $v(a_{3}^{2}+4a_{6})\geq5$ & $9$ & $8$\\\hline

$\rm{IV}$ &  &  &  &  & $4$ & $2$\\\hline

$\rm{I}_{0}^{\ast}$ & $v(a_{3})=2$ &  &  &  & $8$ & $4$\\\cmidrule{2-7}
& $v(a_{3})\geq3$ & $v(a_{1})=1$ &  &  & $9$ & $5$\\\cmidrule{3-7}
&  & $v(a_{1})\geq2$ &  &  & $10$ & $6$\\\hline

$\rm{I}_{1}^{\ast}$ &  &  &  &  & $8$ & $3$\\\hline

$\rm{I}_{2}^{\ast}$ & $v(a_{1})=1$ &  &  &  & $10$ & $4$\\\cmidrule{2-7}
& $v(a_{1})\geq2$ & $v(a_{3})=3$ &  &  & $13$ & $7$\\\cmidrule{3-7}
&  & $v(a_{3})\geq4$ &  &  & $12$ & $6$\\\hline

$\rm{I}_{3}^{\ast}$ & $v(a_{1})=1$ &  &  &  & $11$ & $4$\\\cmidrule{2-7}
& $v(a_{1})\geq2$ &  &  &  & $12$ & $5$\\\hline

$\rm{I}_{n\geq4}^{\ast}$ & $v(a_{1})=1$ &  &  &  & $8+n$ & $4$\\\cmidrule{2-7}
& $v(a_{1})\geq2$ &  &  &  & $10+n$ & $6$\\\hline

$\rm{IV}^{\ast}$ &  &  &  &  & $8$ & $2$\\\hline

$\rm{III}^{\ast}$ & $v(a_{1})=1$ &  &  &  & $10$ & $3$\\\cmidrule{2-7}
& $v(a_{1})\geq2$ & $v(a_{3})=3$ &  &  & $12$ & $5$\\\cmidrule{3-7}
&  & $v(a_{3})\geq4$ & $v(a_{6})=5$ & $v(a_{1}^{2}+4a_{2})=4$ & $15$ & $8$\\\cmidrule{5-7}
&  &  &  & $v(a_{1}^{2}+4a_{2})\geq5$ & $14$ & $7$\\\cmidrule{4-7}
&  &  & $v(a_{6})\geq6$ & $v(a_{1}^{2}+4a_{2})=4$ & $14$ & $7$\\\cmidrule{5-7}
&  &  &  & $v(a_{1}^{2}+4a_{2})\geq5$ & $15$ & $8$\\\hline

$\rm{II}^{\ast}$ & $v(a_{1})=1$ &  &  &  & $11$ & $3$\\\cmidrule{2-7}
& $v(a_{1})\geq2$ & $v(a_{3})=3$ &  &  & $12$ & $4$\\\cmidrule{3-7}
&  & $v(a_{3})\geq4$ &  &  & $14$ & $6$

\label{tab:discriminantconductor}
\end{longtable}\endgroup}

\section{Local data of quadratic twists when \texorpdfstring{$\operatorname{char}\kappa \neq 2$}{chark =/=2}} \label{Sec:oddp}

\begin{theorem}\label{tamanot2}
Let $K$ be the field of fractions of a discrete valuation ring with uniformizer $\pi$ and perfect residue field $\kappa$ of odd characteristic. Let $v$ denote the normalized valuation on $K$, and suppose that an elliptic curve over $K$ is given by a \KNT model $E$. For $d\in K^{\times}/(K^{\times})^{2}$, Table~\ref{tab:twistspodd} gives explicit conditions on the Weierstrass coefficients of $E$ to determine the Kodaira-N\'{e}ron type of the quadratic twist $E^{d}$ of $E$ by $d$. Moreover,
conditions are listed to determine the local Tamagawa number $c^{d}$
of $E^{d}$. When applicable, Table~\ref{tab:twistspodd}
references the local Tamagawa number $c$, determined by the conditions listed in Table~\ref{pnormalmodels}. The table also references the polynomial $P_d(X)=X^{3}+a_{2,1}dX^{2}+a_{4,2}d^{2}X+d^{3}a_{6,3}.$
\end{theorem}

{ \begingroup \scriptsize
\renewcommand{\arraystretch}{1.36}
 \begin{longtable}{cccccc}
 	\caption{N\'eron-Kodaira types and local Tamagawa numbers of $E^d$ when $\operatorname{char}\kappa \neq 2$}\\
 	\hline
 $\operatorname{typ}(E)$ & $v(d)$ & $\operatorname{typ}(E^d)$ & Conditions & $(c,c^{d})$\\
	\hline
	\endfirsthead
	\hline
 $R$ & $v(d)$ & $R^{d}$ & Conditions & $(c,c^{d})$\\
	\hline
	\endhead
	\hline 
	\multicolumn{5}{r}{\emph{continued on next page}}
	\endfoot
	\hline 
	\endlastfoot 
$\rm{I}_{0}$ & $0$ & $\rm{I}_{0}$ &  & $\left(  1,1\right)  $\\\cmidrule{2-5}
& $1$ & $\rm{I}_{0}^{\ast}$ &  & $\left(  1,1+\#\left\{  \alpha\in\kappa\mid
P_{d}(\alpha)=0\right\}  \right)  $\\\hline
$\rm{I}_{n>0}$ & $0$ & $\rm{I}_{n}$ & $\left(  \frac{a_{2}d}{\pi}\right)  =-1$ &
$\left(  c,2-(n\ \operatorname{mod}2)\right)  $\\\cmidrule{4-5}
&  &  & $\left(  \frac{a_{2}d}{\pi}\right)  =1$ & $\left(  c,n\right)  $\\\cmidrule{2-5}
& $1$ & $ \rm{I}_{n}^{\ast}$ & $\left(  \frac{da_{6,n+1}}{\pi}\right)  =-1$ if
$n\not \in 2\mathbb{Z}
$ or $\left(  \frac{-a_{2}^{-1}a_{6,n}}{\pi}\right)  =-1$ if $n\in2\mathbb{Z}$ & $\left(  c,2\right)  $\\\cmidrule{4-5}
&  &  & $\left(  \frac{da_{6,n+1}}{\pi}\right)  =1$ if $n\not \in 2\mathbb{Z}$ or $\left(  \frac{-a_{2}^{-1}a_{6,n}}{\pi}\right)  =1$ if $n\in2\mathbb{Z}$ & $\left(  c,4\right)  $\\\hline
$\rm{II}$ & $0$ & $\rm{II}$ &  & $\left(  1,1\right)  $\\\cmidrule{2-5}
& $1$ & $\rm{IV}^{\ast}$ & $\left(  \frac{da_{6,2}}{\pi}\right)  =-1$ &
$\left(  1,1\right)  $\\\cmidrule{4-5}
&  &  & $\left(  \frac{da_{6,2}}{\pi}\right)  =1$ & $\left(  1,3\right)
$\\\hline
$\rm{III}$ & $0$ & $\rm{III}$ &  & $\left(  2,2\right)  $\\\cmidrule{2-5}
& $1$ & $\rm{III}^{\ast}$ &  & $\left(  2,2\right)  $\\\hline
$\rm{IV}$ & $0$ & $\rm{IV}$ & $\left(  \frac{da_{6,2}}{\pi}\right)  =-1$ & $\left(
c,1\right)  $\\\cmidrule{4-5}
&  &  & $\left(  \frac{da_{6,2}}{\pi}\right)  =1$ & $\left(  c,3\right)  $\\\cmidrule{2-5}
& $1$ & $\rm{II}^{\ast}$ &  & $\left(  c,1\right)  $\\\hline
$\rm{I}_{0}^{\ast}$ & $0$ & $\rm{I}_{0}^{\ast}$ &  & $\left(  c,c  \right)  $\\\cmidrule{2-5}
& $1$ & $\rm{I}_{0}$ &  & $\left(  c,1\right)  $\\\hline
$\rm{I}_{n>0}^{\ast}$ & $0$ & $\rm{I}_{n>0}^{\ast}$ & $\left(  \frac{da_{6,n+3}}{\pi
}\right)  =-1$ if $n\not \in 2\mathbb{Z}$ or $\left(  \frac{-a_{2}^{-1}a_{6,n+2}}{\pi}\right)  =-1$ if $n\in2\mathbb{Z}$ & $\left(  c,2\right)  $\\\cmidrule{3-5}
&  &  & $\left(  \frac{da_{6,n+3}}{\pi}\right)  =1$ if $n\not \in 2\mathbb{Z}$ or $\left(  \frac{-a_{2}^{-1}a_{6,n+2}}{\pi}\right)  =1$ if $n\in2\mathbb{Z}$ & $\left(  c,4\right)  $\\\cmidrule{2-5}
& $1$ & $\rm{I}_{n}$ & $\left(  \frac{da_{2,2}}{\pi}\right)  =-1$ & $\left(
c,2-(n\ \operatorname{mod}2)\right)  $\\\cmidrule{4-5}
&  &  & $\left(  \frac{da_{2,2}}{\pi}\right)  =1$ & $\left(  c,n\right)  $\\\hline
$\rm{IV}^{\ast}$ & $0$ & $\rm{IV}^{\ast}$ & $\left(  \frac{da_{6,4}}{\pi}\right)  =-1$ &
$\left(  c,1\right)  $\\\cmidrule{4-5}
&  &  & $\left(  \frac{da_{6,4}}{\pi}\right)  =1$ & $\left(  c,3\right)  $\\\cmidrule{2-5}
& $1$ & $\rm{II}$ &  & $\left(  c,1\right)  $\\\hline
$\rm{III}^{\ast}$ & $0$ & $\rm{III}^{\ast}$ &  & $\left(  2,2\right)  $\\\cmidrule{2-5}
& $1$ & $\rm{III}$ &  & $\left(  2,2\right)  $\\\hline
$\rm{II}^{\ast}$ & $0$ & $\rm{II}^{\ast}$ &  & $\left(  1,1\right)  $\\\cmidrule{2-5}
& $1$ & $\rm{IV}$ & $\left(  \frac{da_{6,6}}{\pi}\right)  =-1$ & $\left(
1,1\right)  $\\\cmidrule{3-5}
&  &  & $\left(  \frac{da_{6,6}}{\pi}\right)  =1$ & $\left(  1,3\right)  $
 \label{tab:twistspodd}
  \end{longtable}
\endgroup}

\begin{proof}
Since $E$ is given by a \KNT model, we have that $E$ is given by a model $E:y^{2}=x^{3}+a_{2}x^{2}+a_{4}x+a_{6}$ such that $v(a_{i})=n_{i}$ is a nonnegative integer satisfying the bounds given in Table \ref{pnormalmodels}. For the quadratic twist $E^{d}$ of $E$ by $d$, we consider the Weierstrass model
\[
E^{d}:y^{2}=x^{3}+a_{2}dx^{2}+a_{4}d^{2}x+a_{6}d^{3}.
\]
Thus, $\mathcal{V}(E^{d})=\left(  \infty,n_{2}+v(d),\infty,n_{4}+2v(d),n_{6}+3v(d)\right)  $. Thus, if $v(d)=0$, then $\operatorname*{typ}(E)=\operatorname*{typ}(E^{d})$ since $E^{d}$ is given by a \KNT model by Proposition \ref{lem:pnormal}. Now suppose $v(d)=1$. Then, using Table~\ref{pnormalmodels}, it is easily verified that $\operatorname*{typ}(E^{d})$ is as claimed whenever $\operatorname*{typ}(E)\not \in \left\{  \rm{I}_{n\geq 0}^{\ast},\rm{IV}^{\ast},\rm{III}^{\ast},\rm{II}^{\ast}\right\}  $. For these remaining cases, let $F^{d}$ be the elliptic curve obtained from $E^{d}$ via the $K$-isomorphism $\left[  \pi,0,0,0\right]  $. Then,
\[
\mathcal{V}(F^{d})=\left\{
\begin{array}
[c]{cl}
\left(  \infty,0,\infty,0,0\mid v(\Delta)=0\right)   & \text{if } \operatorname*{typ}(E)=\rm{I}_{0}^{\ast},\\
\left(  \infty,=0,\infty,\frac{n+1}{2},=n\right)   & \text{if }\operatorname*{typ}(E)=\rm{I}_{n>0}^{\ast}\text{ and }n\not \in 2
\mathbb{Z},\\
\left(  \infty,=0,\infty,\frac{n+2}{2},=n\right)   & \text{if }\operatorname*{typ}(E)=\rm{I}_{n>0}^{\ast}\text{ and }n\in2
\mathbb{Z},\\
\left(  \infty,1,\infty,1,=1\right)   & \text{if }\operatorname*{typ}(E)=\rm{IV}^{\ast},\\
\left(  \infty,1,\infty,=1,2\right)   & \text{if }\operatorname*{typ}(E)=\rm{III}^{\ast},\\
\left(  \infty,1,\infty,2,=2\right)   & \text{if }\operatorname*{typ}(E)=\rm{II}^{\ast}.
\end{array}
\right.
\]
We note that when $\operatorname*{typ}(E)=\rm{I}_{n>0}^{\ast}$, the argument in Subcase 2b of the proof of Proposition~\ref{pnormalmodels} demonstrates that after an admissible change of variables, we may take a $K$-isomorphic elliptic curve to $E^{d}$ whose Weierstrass coefficients satisfy the valuations $\left(  \infty,=0,\infty,\frac{n+3}{2},=n\right)  $. Consequently, we have that in each case under consideration, there is an elliptic curve that is $K$-isomorphic to $E^{d}$ that is given by a \KNT model. By Proposition~\ref{lem:pnormal}, we conclude that $\operatorname*{typ}(E^{d})$ is as claimed in Table~\ref{tab:twistspodd}. Since Table~\ref{pnormalmodels} gives the local Tamagawa number $c$ of $E$, it remains to show that the local Tamagawa number $c^{d}$ of $E^{d}$ is as claimed in Table~\ref{tab:twistspodd}.

When $\operatorname*{typ}(E^{d})\in\left\{  \rm{I}_{0},\rm{II},\rm{III},\rm{III}^{\ast},\rm{II}^{\ast}\right\},$ the local Tamagawa number is uniquely determined. We now consider the remaining cases separately.

\textbf{Case 1.} Suppose $\operatorname*{typ}(E^{d})=\rm{I}_{n>0}$. Then $\left(\operatorname*{typ}(E),v(d)\right)  \in\left\{  \left(  \rm{I}_{n>0},0\right),\left(  \rm{I}_{n>0}^{\ast},1\right)  \right\}  $.

\qquad\textbf{Subcase 1a.} Suppose $\operatorname*{typ}(E)=\rm{I}_{n>0}$ with $v(d)=0$. By Proposition \ref{lem:pnormal}, we conclude that
\[
c^{d}=\left\{
\begin{array}
[c]{cl}%
n & \text{if }\left(  \frac{a_{2}d}{\pi}\right)  =1,\\
n-\left(  n\ \operatorname{mod}2\right)   & \text{if }\left(  \frac{a_{2}%
d}{\pi}\right)  =-1.
\end{array}
\right.
\]

\qquad\textbf{Subcase 1b.} Suppose $\operatorname*{typ}(E)= \rm{I}_{n>0}^{\ast}$ with $v(d)=1$. Then,
\[
\mathcal{V}(E^{d})=\left\{
\begin{array}
[c]{cl}
\left(  \infty,=2,\infty,\frac{n+9}{2},=n\right)   & \text{if } \operatorname*{typ}(E)=\rm{I}_{n>0}^{\ast}\text{ and }n\not \in 2\mathbb{Z} ,\\
\left(  \infty,=2,\infty,\frac{n+10}{2},=n\right)   & \text{if } \operatorname*{typ}(E)=\rm{I}_{n>0}^{\ast}\text{ and }n\in2 \mathbb{Z}.
\end{array}
\right.
\]
As remarked above, there exists a $K$-isomorphism $\psi$ from $E^{d}$ to an elliptic curve $\widetilde{F}^{d}$ such that $\mathcal{V}(\widetilde{F}^{d})=\left(  \infty,=0,\infty,\frac{n+3}{2},=n\right)  $. Such a $K$-isomorphism must be of the form $\psi = \left[  \pi,r,0,0\right]  $. Thus,
\[
\widetilde{F}^{d}:y^{2}=x^{3}+\pi^{-2}\left(  a_{2}d+3r\right)  x^{2}+\pi^{-4}\left(  a_{4}d^{2}+2a_{2}dr+3r^{2}\right)  x+\pi^{-6}\left(  a_{6}d^{3}+a_{4}d^{2}r+a_{2}dr^{2}+r^{3}\right)  .
\]
By inspection of the Weierstrass coefficient $a_{4}^{\prime}$ of $\widetilde{F}^{d}$, we have that $v(a_{4})\geq\frac{n+9}{2}$. Therefore,
\[
T^{2}-\pi^{-2}\left(  a_{2}d+3r\right)  \equiv T^{2}-a_{2,2}d\ \operatorname{mod}\pi.
\]
It follows by Tate's Algorithm that $c^{d}$ is as claimed in Table~\ref{tab:twistspodd}.

\textbf{Case 2.} Suppose $\operatorname*{typ}(E^{d})=\rm{IV}$. Then $\left(\operatorname*{typ}(E),v(d)\right)  \in\left\{  \left(  \rm{IV},0\right)  ,\left(\rm{II}^{\ast},1\right)  \right\}  $.

\qquad\textbf{Subcase 2a.} Suppose $\operatorname*{typ}(E)=\rm{IV}$ with $v(d)=0$. By Proposition~\ref{lem:pnormal}, we conclude that
\[
c^{d}=\left\{
\begin{array}
[c]{cl}
1 & \text{if }\left(  \frac{d^{3}a_{6,2}}{\pi}\right)  =\left(  \frac{da_{6,2}}{\pi}\right)  =-1,\\
3 & \text{if }\left(  \frac{d^{3}a_{6,2}}{\pi}\right)  =\left(  \frac{da_{6,2}}{\pi}\right)  =1.
\end{array}
\right.
\]

\qquad\textbf{Subcase 2b.} Suppose $\operatorname*{typ}(E)=\rm{II}^{\ast}$ with $v(d)=1$. Then, we consider the elliptic curve $F^{d}$ discussed at the start of the proof, whose Weierstrass model is
\[
F^{d}:y^{2}=x^{3}+a_{2,2}dx^{2}+a_{4,4}d^{2}x+a_{6,6}d^{3}.
\]
Since $F^{d}$ is given by a \KNT model, we conclude by Proposition~\ref{lem:pnormal} that
\[
c^{d}=\left\{
\begin{array}
[c]{cl}
1 & \text{if }\left(  \frac{d^{3}a_{6,8}}{\pi}\right)  =\left(  \frac{da_{6,6}}{\pi}\right)  =-1,\\
3 & \text{if }\left(  \frac{d^{3}a_{6,8}}{\pi}\right)  =\left(  \frac{da_{6,6}}{\pi}\right)  =1.
\end{array}
\right.
\]

\textbf{Case 3.} Suppose $\operatorname*{typ}(E^{d})=\rm{I}_{0}^{\ast}$. Then $\left(  \operatorname*{typ}(E),v(d)\right)  \in\left\{  \left(  \rm{I}_{0}^{\ast},0\right)  ,\left(  \rm{I}_{0},1\right)  \right\}  $. Now set
\begin{align*}
P(X)  & =X^{3}+a_{2,1}X^{2}+a_{4,2}X+a_{6,3},\\
P_{d}(X)  & =X^{3}+a_{2,1}dX^{2}+a_{4,2}d^{2}X+d^{3}a_{6,3}.
\end{align*}
Since $E^{d}$ is given by a \KNT model, we have that $c^{d}=1+\#\left\{\alpha\in\kappa\mid P_{d}(\alpha)=0\right\}  $ by Proposition~\ref{lem:pnormal}. When $\operatorname*{typ}(E)=\rm{I}_{0}^{\ast}$ with $v(d)=0$, we can improve upon this via the following observation. The local Tamagawa number $c$ of $E$ is given by $c=1+\#\left\{  \alpha\in\kappa\mid P(\alpha)=0\right\}  $. Now let $\theta_{1},\theta_{2},$ and $\theta_{3}$ be the distinct roots of $P(X)$ over an algebraic closure $\overline{\kappa}$ of $\kappa$. Thus, $P(X)=\left(  X-\theta_{1}\right)  \left(  X-\theta_{2}\right)  \left(  X-\theta_{3}\right)  $. It is then checked that
\[
P_{d}(X)=\left(  X-d\theta_{1}\right)  \left(  X-d\theta_{2}\right)  \left(
X-d\theta_{3}\right).
\]
Since $d\theta_{i}\in\kappa$ if and only if $\theta_{i}\in\kappa$, we conclude that $c^{d}=c.$

\textbf{Case 4.} Suppose $\operatorname*{typ}(E^{d})=\rm{I}_{n>0}^{\ast}$. Then $\left(  \operatorname*{typ}(E),v(d)\right)  \in\left\{  \left(  \rm{I}_{n>0}^{\ast},0\right)  ,\left(  \rm{I}_{n>0},1\right)  \right\}  $. The local Tamagawa number $c^{d}$ of $E^{d}$ depends on whether
\[
f_{d}(t)=\left\{
\begin{array}
[c]{cl}
t^{2}-d^{3}a_{6,n+3} & \text{if }n\text{ is odd,}\\
da_{2,1}t^{2}+d^{2}a_{4,\frac{n+4}{2}}t+d^{3}a_{6,n+3} & \text{if }n\text{ is even,}
\end{array}
\right.
\]
splits in $\kappa$. Since $E$ is given by a \KNT model, we have that if $n$ is even, then $d^{2}a_{4,\frac{n+4}{2}}\equiv0\ \operatorname{mod}\pi$. Hence,
\[
c^{d}=\left\{
\begin{array}
[c]{clllll}
2 & \text{if }\left(  i\right)  \ \left(  \frac{d^{3}a_{6,n+3}}{\pi}\right)
=-1 & \text{ if }n\not \in 2\mathbb{Z} &\text{or} &\left(  ii\right)  \ \left(  \frac{-d^{2}a_{2,1}^{-1}a_{6,n+3}}{\pi}\right)  =-1 &\text{if }n\in2\mathbb{Z},\\
4 & \text{if }\left(  i\right)  \ \left(  \frac{d^{3}a_{6,n+3}}{\pi}\right)=1 & \text{ if }n\not \in 2\mathbb{Z} & \text{or} & \left(  ii\right)  \ \left(  \frac{-d^{2}a_{2,1}^{-1}a_{6,n+3}}{\pi}\right)  =1 &\text{if }n\in2\mathbb{Z}.
\end{array}
\right.
\]
The results now follows since
\begin{align*}
\left(  \frac{d^{3}a_{6,n+3}}{\pi}\right)    & =\left\{
\begin{array}
[c]{cl}
\left(  \frac{da_{6,n+3}}{\pi}\right)   & \text{if }v(d)=0,\\
\left(  \frac{da_{6,n+1}}{\pi}\right)   & \text{if }v(d)=1,
\end{array}
\right.  \\
\left(  \frac{-d^{2}a_{2,1}^{-1}a_{6,n+3}}{\pi}\right)    & =\left\{
\begin{array}
[c]{cl}
\left(  \frac{-a_{2}^{-1}a_{6,n+2}}{\pi}\right)   & \text{if }v(d)=0,\\
\left(  \frac{-a_{2}^{-1}a_{6,n}}{\pi}\right)   & \text{if }v(d)=1.
\end{array}
\right.
\end{align*}

\textbf{Case 5.} Suppose $\operatorname*{typ}(E^{d})=\rm{IV}^{\ast}$. Then $\left(\operatorname*{typ}(E),v(d)\right)  \in\left\{  \left(  \rm{IV}^{\ast},0\right),\left(  \rm{II},1\right)  \right\}  $.

\qquad\textbf{Subcase 5a.} Suppose $\operatorname*{typ}(E)=\rm{IV}^{\ast}$ with $v(d)=0$. Since $E^{d}$ is given by a \KNT model, Proposition~\ref{lem:pnormal} implies that
\[
c^{d}=\left\{
\begin{array}
[c]{cl}
1 & \text{if }\left(  \frac{d^{3}a_{6,4}}{\pi}\right)  =\left(  \frac{da_{6,4}}{\pi}\right)  =-1,\\
3 & \text{if }\left(  \frac{d^{3}a_{6,4}}{\pi}\right)  =\left(  \frac{da_{6,4}}{\pi}\right)  =1.
\end{array}
\right.
\]

\qquad\textbf{Subcase 5b.} Suppose $\operatorname*{typ}(E)=\rm{II}$ with $v(d)=1$. Since $E^{d}$ is given by a \KNT model, we conclude by Proposition~\ref{lem:pnormal} that
\[
c^{d}=\left\{
\begin{array}
[c]{cl}
1 & \text{if }\left(  \frac{d^{3}a_{6,4}}{\pi}\right)  =\left(  \frac{da_{6,2}}{\pi}\right)  =-1,\\
3 & \text{if }\left(  \frac{d^{3}a_{6,4}}{\pi}\right)  =\left(  \frac{da_{6,2}}{\pi}\right)  =1.
\end{array}
\right. \qedhere
\]

\end{proof}


\section{Local data of quadratic twists over \texorpdfstring{$\Q_2$}{Q\_2}}

Suppose an elliptic curve over $\Q_2$ is given by a \KNT model~$E$. For such a model, Lemma~\ref{Lemmadisccond} provides conditions on the Weierstrass coefficients of $E$ to determine the minimal discriminant valuation and conductor exponent. Similarly, for the Kodaira-N\'{e}ron types for which the local Tamagawa number is not uniquely determined, Corollary \ref{TamagawaQ2} gives conditions on the Weierstrass coefficients of $E$ to determine the local Tamagawa number of $E$. Next, let $E^{d}$ denote the quadratic twist of $E$ by $d\in \mathbb{Q}_2^{\times}/(\mathbb{Q}_2^{\times})^{2}$. Thus, we may assume that $v(d)\in\left\{  0,1\right\}  $. 

The goal for this section is to determine conditions on the Weierstrass coefficients of $E$ to deduce the Kodaira-N\'{e}ron type, conductor exponent, minimal discriminant valuation, and local Tamagawa number of $E^{d}$. The strategy is similar to that undertaken in Theorem \ref{tamanot2} for the case when $\operatorname*{char}\kappa\neq2$, but more involved due to the complexities of working over $\mathbb{Q}_{2}$. Namely, we will proceed by cases based on the Kodaira-N\'{e}ron type of $E$. Specifically, for a fixed Kodaira-type, we consider an elliptic curve given by a \KNT model
\[
E:y^{2}+a_{1}xy+a_{3}y=x^{3}+a_{2}x^{2}+a_{4}x+a_{6}.
\]
Consequently, $v(a_{i})\geq0$ for each $i$. We then fix the following model for the quadratic twist $E^{d}$:
\begin{equation}
E^{d}:y^{2}=x^{3}+d\left(  a_{2}^{2}+4a_{2}\right)  x^{2}+d^{2}\left(
8a_{1}a_{3}+16a_{4}\right)  x+d^{3}\left(  16a_{3}^{2}+64a_{6}\right)
.\label{Edmodel}
\end{equation}
Each case then consists of further subcases, whose ultimate goal is to find an elliptic curve $F^{d}$ that is $\Q_2$-isomorphic to $E^{d}$ such that $F^{d}$ is given by a \KNT model. Once we have established this, we can conclude the local data of $E^{d}$ from Proposition \ref{lem:pnormal}, Corollary \ref{TamagawaQ2}, and Lemma \ref{Lemmadisccond}. Due to the length of the argument, we consider the cases $v(d)=0$ and $v(d)=1$ separately. The proofs have accompanying notebooks written in SageMath \cite{sagemath}, and can be found in \cite{gittwists}.

\begin{theorem}\label{thmQ2combined}
Suppose an elliptic curve over $\Q_2$ is given by a \KNT model~$E$. For $d\in\mathbb{Q}_{2}$ with $v(d)=0$ (resp. $1$), let $E^{d}$ denote the quadratic twist of $E$ by $d$. Then, Table \ref{tab:localdata-dodd} (resp. \ref{tab:localdata-deven}) gives explicit conditions on the Weierstrass coefficients $a_{i}$ of $E$ to determine the Kodaira-N\'{e}ron types of $E$ and $E^{d}$. 

Additional conditions are also provided to determine the local Tamagawa numbers $c$ and $c^{d}$, the minimal discriminant valuations $\delta$ and $\delta^{d}$, and conductor exponents $f$ and $f^{d}$ of $E$ and $E^{d}$, respectively. When applicable, Table~\ref{tab:localdata-dodd} (resp. \ref{tab:localdata-deven}) references $n'=2-(n \mod 2)$, as well as additional conditions on a quantity $P_{R,j}^{v(d)}$, whose exact value is given in Table \ref{tab:PRj}. We note that in some cases, Tables~\ref{tab:localdata-dodd}~and~\ref{tab:localdata-deven} gives the possible pairs $\left(\delta_{2},\delta_{2}^{d}\right)  $ and $\left(  f_{2},f_{2}^{d}\right)  $. In these instances, the exact values of these pairs can then be determined from Lemma~\ref{Lemmadisccond}. Similarly, if $c$ appears in the table, then its exact value can be obtained from Corollary~\ref{TamagawaQ2}.
\end{theorem}

{ \begingroup \tiny
\renewcommand{\arraystretch}{1.3}
 \begin{longtable}{cccccclc}
 	\caption{Local data for $E/\Q_2$ and $E^d/\Q_2$ with $v(d)=0$ }
  \label{tab:localdata-dodd}
  \\\hline
$R$ & \multicolumn{2}{c}{Conditions} &
$R^{d}$ & $\left(  \delta,\delta^{d}\right)  $ &
$\left(  f,f^{d}\right)  $ & Additional conditions
& $\left(c,c^{d}\right)  $\\
	\hline
	\endhead
	\hline 
	\multicolumn{4}{r}{\emph{continued on next page}}
	\endfoot
	\hline 
	\endlastfoot

$\rm{I}_{0}$ & $d\equiv1\ \operatorname{mod}4$ &  & $\rm{I}_{0}$ & $\left(  0,0\right)
$ & $\left(  0,0\right)  $ &  & $\left(  1,1\right)  $\\\cmidrule{2-8}
& $d\equiv3\ \operatorname{mod}4$ & $v(a_{1})=0$ & $\rm{I}_{4}^{\ast}$ &
$\left(  0,12\right)  $ & $\left(  0,4\right)  $ & $a_{6}\equiv
1,2\ \operatorname{mod}4$ & $\left(  1,2\right)  $\\\cmidrule{7-8}
&  &  &  &  &  & $a_{6}\equiv0,3\ \operatorname{mod}4$ & $\left(
1,4\right)  $\\\cmidrule{3-8}
&  & $v(a_{1})\geq1$ & $\rm{II}^{\ast}$ & $\left(  0,12\right)  $ & $\left(
0,4\right)  $ &  & $\left(  1,1\right)  $\\\midrule

$\rm{I}_{n>0}$ & $d\equiv1\ \operatorname{mod}4$ &  & $\rm{I}_{n}$ & $\left(n,n\right)  $ & $\left(  1,1\right)  $ & $v(d-1+4a_2)=2$ & $\left(  c,n^{\prime}\right)  $\\\cmidrule{7-8}
&  &  &  &  &  & $v(d-1+4a_2)\ge 3$ & $\left(c,n\right)  $\\\cmidrule{2-8}

& $d\equiv3\ \operatorname{mod}4$ & & $\rm{I}_{n+4}^{\ast}$ & $\left(n,12+n\right)  $ & $\left(  1,4\right)  $ & $v(n)=0, v(P_{R,1})=n+1$ & $\left(  c,2\right)  $\\\cmidrule{7-8}
&  &  &  &  &  & $v(n)=0,v(P_{R,1})\geq n+2$ & $\left(  c,4\right)  $\\\cmidrule{7-8}
&  &  &  &  &  & $v(n)\ge 1, v(P_{R,2})=n+2$ & $\left(  c,2\right)  $\\\cmidrule{7-8}
&  &  &  &  &  & $v(n)\ge 1, v(P_{R,2})\geq n+3$ & $\left(  c,4\right)  $\\\midrule

$\rm{II}$ & $v(a_{3})\geq2$ &  & $\rm{II}$ & $\left(  k,k\right)  $ & $\left(k,k\right)  $ & for $k\in\left\{  6,7\right\}  $ & $\left(  1,1\right)  $\\\cmidrule{2-8}
& $v(a_{3})=1$ & $d\equiv1\ \operatorname{mod}4$ & $\rm{II}$ & $\left(4,4\right)  $ & $\left(  4,4\right)  $ &  & $\left(  1,1\right)  $\\\cmidrule{3-8}
&  & $d\equiv3\ \operatorname{mod}4,v(a_{4})=1$ & $\rm{III}$ & $\left(4,4\right)  $ & $\left(  4,3\right)  $ &  & $\left(  1,2\right)  $\\\cmidrule{3-8}
&  & $d\equiv3\ \operatorname{mod}4,v(a_{4})\geq2$ & $\rm{IV}$ & $\left(
4,4\right)  $ & $\left(  4,2\right)  $ & $v(da_{6}+d -1)=2$
& $\left(  1,1\right)  $\\\cmidrule{7-8}
&  &  &  &  &  & $v(da_{6}+d -1)\ge 3$ & $\left(  1,3\right)
$\\\midrule

$\rm{III}$ & $d\equiv1\ \operatorname{mod}4$ &  & $\rm{III}$ & $\left(  k,k\right)  $ &
$\left(  k-1,k-1\right)  $ & for $k\in\left\{  4,6,8,9\right\}  $ & $\left(
2,2\right)  $\\\cmidrule{2-8}
& $d\equiv3\ \operatorname{mod}4$ & $v(a_{3})=1$ & $\rm{II}$ & $\left(
4,4\right)  $ & $\left(  3,4\right)  $ &  & $\left(  2,1\right)  $\\\cmidrule{3-8}
&  & $v(a_{3})\geq2$ & $\rm{III}$ & $\left(  k,k\right)  $ & $\left(
k-1,k-1\right)  $ & for $k\in\left\{  6,8,9\right\}  $ & $\left(  2,2\right)
$\\\midrule

$\rm{IV}$ & $d\equiv3\ \operatorname{mod}4$ &  & $\rm{II}$ & $\left(  4,4\right)  $ & $\left(  2,4\right)  $ &  & $\left(  c,1\right)  $\\\cmidrule{2-8}

& $d\equiv1\ \operatorname{mod}4$ &  & $\rm{IV}$ & $\left(  4,4\right)  $ & $\left(  2,2\right)  $ & $v(da_{6}+d -1)=2$ & $\left(  c,1\right)  $\\\cmidrule{7-8}
&  &  &  &  &  & $v(da_{6}+d -1)\ge 3$ & $\left(c,3\right)  $\\\midrule

$\rm{I}_{0}^{\ast}$ & $d\equiv1\ \operatorname{mod}4$ &  & $\rm{I}_{0}^{\ast}$ &
$\left(  k,k\right)  $ & $\left(  k-4,k-4\right)  $ & $v(a_{2})=1$ &
$\left(  1,1\right)  $\\\cmidrule{7-8}
&  &  &  & \multicolumn{2}{c}{for $k\in\left\{  8,9,10\right\}  $} &$v(a_{2})\geq2$ & $\left(  2,2\right)  $\\\cmidrule{2-8}

& $d\equiv3\ \operatorname{mod}4$ & $v(a_{3})\geq3$ & $\rm{I}_{0}^{\ast}$ &$\left(  k,k\right)  $ & $\left(  k-4,k-4\right)  $ & $v(a_{1}+a_{2})=1$ & $\left(  c,1\right)  $\\\cmidrule{7-8}
&  &  &  & \multicolumn{2}{c}{for $k\in\left\{  9,10\right\}  $} & $v(a_{1}+a_{2})\ge 2$ & $\left(  c,2\right)  $\\\cmidrule{3-8}

&  & $v(a_{3})=2,v(a_{1}+a_{2})=1$ & $\rm{I}_{1}^{\ast}$ & $\left(8,8\right)  $ & $\left(  4,3\right)  $ & $v(P_{R,1}^{0})=4$ & $\left(  c,2\right)  $\\\cmidrule{7-8}
&  &  &  &  &  & $v(P_{R,1}^{0})\ge 5$ & $\left(  c,4\right)  $\\\cmidrule{3-8}

&  & $v(a_{3})=2,v(a_{1}+a_{2})\geq2$ & $\rm{IV}^{\ast}$ & $\left(
8,8\right)  $ & $\left(  4,2\right)  $ & $v(P_{R,1}^{0})= 4$ & $\left(  c,1\right)  $\\\cmidrule{7-8}
&  &  &  &  &  & $v(P_{R,1}^{0})\ge 5$ & $\left(  c,3\right)  $\\\midrule

$\rm{I}_{n>0}^{\ast}$ & $d\equiv1\ \operatorname{mod}4$ &  & $\rm{I}_{n}^{\ast}$ & \multicolumn{2}{l}{$n=1\ \left(  8,8\right)  \hspace{3em} \left(  4,4\right)  $}  &$v(n)=0,v(P_{R,1}^{0})=n+3$ & $\left(  c,2\right)  $\\\cmidrule{7-8}
&  &  &  &  \multicolumn{2}{l}{$n=2\ \underset{\text{for }k\in\left\{  10,12,13\right\}
}{\left(  k,k\right)  \hspace{1.5em} \left(  k-6,k-6\right)  }$}  & $v(n)=0,v(P_{R,1}^{0})\geq n+4$ & $\left(  c,4\right)  $\\\cmidrule{7-8}
&  &  &  &  \multicolumn{2}{l}{$n=3\ \underset{\text{for }k\in\left\{  11,12\right\}
}{\left(  k,k\right)  \hspace{1.5em} \left(  k-7,k-7\right)  }$}  & $v(n)\geq1,v(a_{6})=n+3$ & $\left(  2,2\right)  $\\\cmidrule{7-8}
&  &  &    \multicolumn{3}{c}{$n\geq4\
\begin{array}
[c]{c}%
\left(  n+8,n+8\right)  \\
\left(  n+10,n+10\right)
\end{array}
\
\begin{array}
[c]{c}%
\left(  4,4\right)  \\
\left(  6,6\right)
\end{array}
$} & $v(n)\geq1,v(a_{6})\geq n+4$ & $\left(  4,4\right)  $\\\cmidrule{2-8}

& $d\equiv3\ \operatorname{mod}4$ & $n=1$ & $\rm{I}_{0}^{\ast}$ & $(8,8)$ & $(3,4)$ & $v(a_{1})\geq2$ & $\left(  c,1\right)  $\\\cmidrule{7-8}
&  &  &  &  &  & $v(a_{1})=1$ & $\left(  c,2\right)  $\\\cmidrule{3-8}

&  & $n=2,v(a_{1})=1$ & $\rm{III}^{\ast}$ & $(10,10)$ & $(4,3)$ &  & $\left(  c,2\right)  $\\\cmidrule{3-8}

&  & $v(n)\geq1,v(a_{1})\geq2$ & $\rm{I}_{n}^{\ast}$ & $(n+10,n+10)$ & $(6,6)$ & $v(a_{3}^{2}+2a_{6})=n+4$ & $\left(  c,2\right)  $\\\cmidrule{7-8}
&  &  &  & \multicolumn{2}{l}{or $(13,13)$ \hspace{2em} $(7,7)$ if $n=2$} &$v(a_{3}^{2}+2a_{6})\geq n+5$ & $\left(  c,4\right)  $\\\cmidrule{3-8}

&  & $n=3,v(a_{1})=1$ & $\rm{II}^{\ast}$ & $(11,11)$ & $(4,3)$ &  & $\left(  c,1\right)  $\\\cmidrule{3-8}

&  & $n \ge 3,v(n)=0, v(a_{1})\geq2$ & $\rm{I}_{n}^{\ast}$ &\multicolumn{2}{l}{$n=3$ $(12,12)$ \hspace{3.3em} $(5,5)$} & $v(P_{R,2}^0)= n+4$ & $\left(  c,2\right)  $\\\cmidrule{7-8}
&  &  &  &  \multicolumn{2}{l}{$n\ge 5$ $(n+10,n+10)$ \hspace{0.5em} $(6,6)$}  & $v(P_{R,2}^0)\ge n+5$ & $\left(  c,4\right)  $\\\cmidrule{3-8}

&  & $n=4,v(a_{1})=1$ & $\rm{I}_{0}$ & $(12,0)$ & $(4,0)$ &  & $\left(  c,1\right)  $\\\cmidrule{3-8}

&  & $n\geq5,v(a_{1})=1$ & $\rm{I}_{n-4}$ & $(n+8,n-4)$ & $(4,1)$ & $v(d-1+a_{2}d)=2$ & $\left(c,n^{\prime}\right)  $\\\cmidrule{7-8}
&  &  &  &  &  & $v(d-1+a_{2}d)\geq3$ & $\left(  c,n-4\right)  $\\\midrule

$\rm{IV}^{\ast}$ & $d\equiv1\ \operatorname{mod}4$ &  & $\rm{IV}^{\ast}$ & $\left(8,8\right)  $ & $\left(  2,2\right)  $ & $v(4d-4 +a_6)=4 $ & $\left(  c,1\right)  $\\\cmidrule{7-8}
&  &  &  &  &  & $v(4d-4 +a_6)\ge 5$ & $\left(c,3\right)  $\\\cmidrule{2-8}

& $d\equiv3\ \operatorname{mod}4$ &  & $\rm{I}_{0}^{\ast}$ & $\left(  8,8\right)  $ & $\left(  2,4\right)  $ & $v(a_{1})=1$ & $\left(c,1\right)  $\\\cmidrule{7-8}
&  &  &  &  &  & $v(a_{1})\geq2$ & $\left(  c,2\right)  $\\\midrule

$\rm{III}^{\ast}$ & \multicolumn{2}{l}{\hspace{0.3em} $v(a_{1})\geq2$ or $d\equiv
1\ \operatorname{mod}4$} & $\rm{III}^{\ast}$ & $\left(  k,k\right)  $ & $\left(
k-7,k-7\right)  $ & for $k\in\left\{ 12,14,15\right\}  $ & $\left(
2,2\right)  $\\\cmidrule{2-8}
& $v(a_{1})=1$ & $d\equiv3\ \operatorname{mod}4$ & $\rm{I}_{2}^{\ast}$ &
$\left(  10,10\right)  $ & $\left(  3,4\right)  $ & $v(a_{3}^{2}%
+2a_{6})=6$ & $\left(  2,2\right)  $\\\cmidrule{7-8}
&  &  &  &  &  & $v(a_{3}^{2}+2a_{6})\geq7$ & $\left(  2,4\right)  $\\\midrule

$\rm{II}^{\ast}$ & $d\equiv1\ \operatorname{mod}4$ &  & $\rm{II}^{\ast}$ & $\left(
k,k\right)  $ & $\left(  k-8,k-8\right)  $ & for $k\in\left\{
11,12,14\right\}  $ & $\left(  1,1\right)  $\\\cmidrule{2-8}
& $d\equiv3\ \operatorname{mod}4$ & $v(a_{1})=1,v(a_{3})=3$ &
$\rm{I}_{3}^{\ast}$ & $\left(  11,11\right)  $ & $\left(  3,4\right)  $ &
$v(16d-16+a_{6}d)=6$ & $\left(  1,2\right)
$\\\cmidrule{7-8}
&  &  &  &  &  & $v(16d-16+a_{6}d)\ge 7$ &
$\left(  1,4\right)  $\\\cmidrule{3-8}
&  & $v(a_{1})=1,v(a_{3})\geq4$ & $\rm{I}_{3}^{\ast}$ & $\left(
11,11\right)  $ & $\left(  3,4\right)  $ & $v(P_{R,1}^{0})=6$ & $\left(  1,2\right)  $\\\cmidrule{7-8}
&  &  &  &  &  & $v(P_{R,1}^{0})\ge 7$ &
$\left(  1,4\right)  $\\\cmidrule{3-8}
&  & $v(a_{1})\geq2,v(a_{3})=3$ & $\rm{I}_{0}$ & $\left(  12,0\right)  $ &
$\left(  4,0\right)  $ &  & $\left(  1,1\right)  $\\\cmidrule{3-8}
&  & $v(a_{1})\geq2,v(a_{3})\geq4$ & $\rm{II}^{\ast}$ & $\left(
14,14\right)  $ & $\left(  6,6\right)  $ &  & $\left(  1,1\right)$
\end{longtable}
\endgroup}

{ \begingroup \tiny
\renewcommand{\arraystretch}{1.3}
 \begin{longtable}{cccccclc}
 	\caption{Local data for $E/\Q_2$ and $E^d/\Q_2$ with $v(d)=1$ }
  \label{tab:localdata-deven}
  \\\hline
$R$ & \multicolumn{2}{c}{Conditions} &
$R^{d}$ & $\left(  \delta,\delta^{d}\right)  $ &
$\left(  f,f^{d}\right)  $ & Additional conditions
& $\left(c,c^{d}\right)  $\\
	\hline
	\endhead
	\hline 
	\multicolumn{4}{r}{\emph{continued on next page}}
	\endfoot
	\hline 
	\endlastfoot
$\rm{I}_{0}$ & $v(a_{1})=0$ & $v(a_{6})=0$ & $\rm{I}_{8}^{\ast}$ & $\left(  0,18\right)$ & $\left(  0,6\right)  $ & $v(P_{R,1})=4$ & $(1,2)$\\\cmidrule{7-8}
&  &  &  &  &  & $v(P_{R,1})\geq5$ & $(1,4)$\\\cmidrule{3-8}
&  & $v(a_{6})\geq1$ & $\rm{I}_{8}^{\ast}$ & $\left(  0,18\right)  $ & $\left(
0,6\right)  $ & $v(P_{R,2})=4$ & $(1,2)$\\\cmidrule{7-8}
&  &  &  &  &  & $v(P_{R,2})\geq5$ & $(1,4)$\\\cmidrule{2-8}
& $v(a_{1})\geq1$ &  & $\rm{II}$ & $(0,6)$ & $(0,6)$ &  & $(1,1)$\\\midrule


$\rm{I}_{n>0}$ &  &  & $\rm{I}_{n+8}^{\ast}$ & $\left(  n,n+18\right)  $ & $\left(
1,6\right)  $ & $v(n)=0:v(P_{R,1})=n+4$ & $\left(  c,2\right)  $\\
&  &  &  &  &  & or $v(P_{R,3})=n+4$ & \\\cmidrule{7-7}
 &  &  &  &  &  & $v(n)\geq1:v(P_{R,2})=n+2$ & \\
&  &  &  &  &  & or $v(P_{R,4})=n+2$ & \\\cmidrule{7-8}
 &  &  &  &  &  & $v(n)=0:v(P_{R,1})\geq n+5$ & $\left(  c,4\right)  $\\
&  &  &  &  &  & or $v(P_{R,3})\geq
n+5$ & \\\cmidrule{7-7}
 &  &  &  &  &  & $v(n)\geq1:v(P_{R,2})\geq n+3$ &\\
&  &  &  &  &  & or $v(P_{R,4}%
)\geq n+3$ & \\\midrule


$\rm{II}$ & $v(a_{3})=1$ &  & $\rm{I}_{0}^{\ast}$ & $\left(  4,10\right)  $ & $\left(4,6\right)  $ & $v(a_{1})=1$ & $(1,1)$\\
&  &  &  &  &  & $v(a_{1})\geq2$ & $(1,2)$\\\cmidrule{2-8}
& $v(a_{3})\geq2,$ & $v(a_{4})=1$ & $\rm{I}_{2}^{\ast}$ & $\left(  7,13\right)  $ & $\left(  7,7\right)  $ & $v(a_{3}^{2}+4a_{6}-d^{3})=4$ & $\left(  1,2\right)$\\\cmidrule{7-8}
& $v(a_{1})=1$ &  &  &  &  & $v(a_{3}^{2}+4a_{6}-d^{3})\geq5$ & $\left(1,4\right)  $\\\cmidrule{3-8}
&  & $v(a_{4})\geq2$ & $\rm{I}_{3}^{\ast}$ & $\left(  6,12\right)  $ & $\left(6,5\right)  $ & $v(P_{R,1})=5$ or $v(P_{R,2})=5$ & $\left(  1,2\right)  $\\\cmidrule{7-8}
&  &  &  &  &  & $v(P_{R,1})\geq6$ or $v(P_{R,2})\geq6$ & $\left(  1,4\right)$\\\cmidrule{2-8}
& $v(a_{3})\geq2,$ & $v(a_{4})=1$ & $\rm{III}^{\ast}$ & $\left(  6,12\right)  $ &$\left(  6,5\right)  $ &  & $\left(  1,2\right)  $\\\cmidrule{3-8}
& $v(a_{1})\geq2$ & $v(a_{4})\geq2,v(P_{R,1})=4$ & $\rm{II}^{\ast}$ & $\left(  6,12\right)  $ & $\left(  6,4\right)  $ &  & $\left(  1,1\right)  $\\\cmidrule{3-8}
&  & $v(a_{4})\geq2,v(P_{R,1})\geq5$ & $\rm{I}_{0}$ & $\left( 6,0\right)  $ & $\left(  6,0\right)  $ &  & $\left(  1,1\right)  $\\\midrule


$\rm{III}$ & $v(a_{3})=1$ &  & $\rm{I}_{0}^{\ast}$ & $\left(  4,10\right)  $ & $\left(
3,6\right)  $ & $v(a_{1})=1$ & $\left(  2,1\right)  $\\\cmidrule{7-8}
&  &  &  &  &  & $v(a_{1})\geq2$ & $\left(  2,2\right)  $\\\cmidrule{2-8}
& $v(a_{3})\geq2$ & $v(a_{1})=1$ & $\rm{I}_{2}^{\ast}$ & $\left(  6,12\right)  $ & $\left(  5,6\right)  $ & $v(a_{3}^{2}+4a_{6})=4$ & $\left(  2,2\right)  $\\\cmidrule{7-8}
&  &  &  &  &  & $v(a_{3}^{2}+4a_{6})\geq5$ & $\left(  2,4\right)  $\\\cmidrule{3-8}
&  & $v(a_{1})\geq2$ & $\rm{III}^{\ast}$ & $\left(  k,k+6\right)  $ & $\left(k-1,k-1\right)  $ & \multicolumn{1}{l}{for $k\in\{8,9\}$} & $\left(2,2\right)  $\\\midrule


$\rm{IV}$ &  &  & $\rm{I}_{0}^{\ast}$ & $\left(  4,10\right)  $ & $\left(  2,6\right)  $ & $v(a_{1})=1$ & $\left(  c,1\right)  $\\\cmidrule{7-8}
&  &  &  &  &  & $v(a_{1})\geq2$ & $\left(  c,2\right)  $\\\midrule


$\rm{I}_{0}^{\ast}$ & $v(a_{1})=1$ & $v(a_{3})=2$ & $\rm{I}_{4}^{\ast}$ & $\left(8,14\right)  $ & $\left(  4,6\right)  $ & $v(P_{R,1})=6$ & $\left(c,2\right)  $\\\cmidrule{7-8}
&  &  &  &  &  & $v(P_{R,1})\geq7$ & $\left(  c,4\right)  $\\\cmidrule{3-8}
&  & $v(a_{3})\geq3$ & $\rm{I}_{5}^{\ast}$ & $\left(  9,15\right)  $ & $\left(5,6\right)  $ & $v(P_{R,2})=7$ or $v(P_{R,3})=7$ & $\left(  c,2\right)  $\\\cmidrule{7-8}
&  &  &  &  &  & $v(P_{R,2})\geq8$ or $v(P_{R,3})\geq8$ & $\left(  c,4\right)$\\\cmidrule{2-8}
& $v(a_{1})\geq2$ & $v(a_{3})=2$ & $\rm{II}^{\ast}$ & $\left(  8,14\right)  $ & $\left(  4,6\right)  $ &  & $\left(  c,1\right)  $\\\cmidrule{2-8}
& $v(a_{1})\geq2,$ & $v(P_{R,4})=7$ & $\rm{II}$ & $\left(  10,4\right)  $ & $\left(  6,4\right)  $ &  & $\left(  c,1\right)  $\\\cmidrule{3-8}
& $v(a_{3})\geq3$ & $v(P_{R,4})\geq8,v(P_{R,5})=3$ & $\rm{III}$ & $\left(10,4\right)  $ & $\left(  6,3\right)  $ &  & $\left(  c,2\right)  $\\\cmidrule{3-8}
&  & $v(P_{R,4})\geq8,v(P_{R,5})\geq4$ & $\rm{IV}$ & $\left(  10,4\right)  $ & $\left(  6,2\right)  $ & $v(a_{3}^{2}d+4a_{6}d-64)=8$ & $\left(  c,1\right)$\\\cmidrule{7-8}
&  &  &  &  &  & $v(a_{3}^{2}d+4a_{6}d-64)\geq9$ & $\left(  c,3\right)  $\\\midrule


$\rm{I}_{n>0}^{\ast}$ & $n=1$ & $v(a_{1})=1$ & $\rm{I}_{4}^{\ast}$ & $\left(8,14\right)  $ & $\left(  3,6\right)  $ & $v(P_{R,1})=8$ & $\left(c,2\right)  $\\\cmidrule{7-8}
&  &  &  &  &  & $v(P_{R,1})\geq9$ & $\left(  c,4\right)  $\\\cmidrule{3-8}
&  & $v(a_{1})\geq2$ & $\rm{II}^{\ast}$ & $\left(  8,14\right)  $ & $\left(3,6\right)  $ &  & $\left(  c,1\right)  $\\\cmidrule{2-8}
& $n=2$ & $v(a_{1})\geq2,v(a_3)=3$ & $\rm{II}$ & $\left(  13,7\right)  $ & $\left(  7,7\right)  $ &  & $\left(  c,1\right)  $\\\cmidrule{3-8}
&  & $v(a_{1})\geq2,v(a_3)\geq4$ & $\rm{III}$ & $\left(  12,6\right)  $ & $\left(  6,5\right)  $ &  & $\left(  c,2\right)  $\\\cmidrule{2-8}
& $n=3$ & $v(a_{1})\geq2$ & $\rm{II}$ & $\left(  12,6\right)  $ & $\left(5,6\right)  $ &  & $\left(  c,1\right)  $\\\cmidrule{2-8}
& $n=4,$ & $v(P_{R,3})=4$ & $\rm{I}_{0}^{\ast}$ & $\left(  14,8\right)  $ & $\left(  6,4\right)  $ & $v(a_{1}^{2}d+4a_{2}d+48d-16)=5$ & $\left(
c,1\right)  $\\\cmidrule{7-8}
& $v(a_{1})\geq2$ &  &  &  &  & $v(a_{1}^{2}d+4a_{2}d+48d-16)\geq6$ & $\left(c,2\right)  $\\\cmidrule{3-8}
&  & $v(P_{R,3})\geq5,v(P_{R,4})=5$ & $\rm{I}_{1}^{\ast}$ & $\left(  14,8\right)  $ & $\left(  6,3\right)  $ & $v(P_{R,5})=9$ & $\left(  c,2\right)  $\\\cmidrule{7-8}
&  &  &  &  &  & $v(P_{R,5})\geq10$ & $\left(  c,4\right)  $\\\cmidrule{3-8}
&  & $v(P_{R,3})\geq5,v(P_{R,4})\geq6$ & $\rm{IV}^{\ast}$ & $\left(  14,8\right)  $
& $\left(  6,2\right)  $ & $v(P_{R,5})=9$ & $\left(  c,1\right)  $\\\cmidrule{7-8}
&  &  &  &  &  & $v(P_{R,5})\geq10$ & $\left(  c,3\right)  $\\\cmidrule{2-8}
& $n=5$ & $v(a_{1})\geq2$ & $\rm{I}_{0}^{\ast}$ & $\left(  15,9\right)  $ & $\left(  6,5\right)  $ & $v(a_{1}^{2}+4a_{2}-4d)=4$ & $\left(  c,1\right)  $\\\cmidrule{7-8}
&  &  &  &  &  & $v(a_{1}^{2}+4a_{2}-4d)\geq5$ & $\left(  c,2\right)  $\\\cmidrule{2-8}
& $n=6$ & $v(a_{1})\geq2,v(P_{R,6})\geq5$ & $\rm{III}^{\ast}$ & $\left(16,10\right)  $ & $\left(  6,3\right)  $ &  & $\left(  c,2\right)  $\\\cmidrule{2-8}
& $n=7$ & $v(a_{1})\geq2,v(P_{R,6})\geq5$ & $\rm{II}^{\ast}$ & $\left(17,11\right)  $ & $\left(  6,3\right)  $ &  & $\left(  c,1\right)  $\\\cmidrule{2-8}
& $n=8$ & $v(a_{1})\geq2,v(P_{R,6})\geq5$ & $\rm{I}_{0}$ & $\left(  18,0\right)  $
& $\left(  6,0\right)  $ &  & $\left(  c,1\right)  $\\\cmidrule{2-8}
& $n\geq2$ & $v(a_{1})=1$ & $\rm{I}_{n+4}^{\ast}$ & \multicolumn{2}{l}{$\left(n+8,n+14\right)  \ \ \ \left(  4,6\right)  $} & $v(n)=0:\ v(P_{R,8})=n+8$ & $\left(  c,2\right)  $\\
&  &  &  &  &  & or $v(P_{R,9})=n+8$ & \\\cmidrule{7-7}
&  &  &  & && $v(n)\geq1:\ v(P_{R,1})=n+8$ & \\
&  &  &  &&& or $v(P_{R,7})=n+8$ & \\\cmidrule{7-8}

& & &  & & & $v(n)=0:\ v(P_{R,8})\geq n+9$ & $\left(  c,4\right)  $\\
&  &  &  &  &  & or $v(P_{R,9})\geq n+9$ & \\\cmidrule{7-7}
&  &  &  & && $v(n)\geq1:\ v(P_{R,1})\geq n+9$ & \\
&  &  &  &&& or $v(P_{R,7})\geq n+9$ & \\\cmidrule{2-8}
& $n\geq6$ & $v(a_{1})\geq2,v(P_{R,6})=4$ & $\rm{I}_{n-4}^{\ast}$ & \multicolumn{2}{l}{$\left(  n+10,n+4\right)  \ \ \ \ \left(  6,4\right)  $} & $v(n)=0:$ $v(P_{R,10})=n+9$ & $\left(  c,2\right)  $\\
&  &  &  &  &  & $v(n)\geq1:\ v(P_{R,2})=n+4$ & \\\cmidrule{7-8}
&  &  &  &  &  & $v(n)=0:$ $v(P_{R,10})\geq n+10$ & $\left(  c,4\right)  $\\
&  &  &  &  &  & $v(n)\geq1:\ v(P_{R,2})\geq n+5$ & \\\cmidrule{2-8}
& $n\geq9$ & $v(a_{1})\geq2,v(P_{R,6})\geq5$ & $\rm{I}_{n-8}$ & \multicolumn{2}{l}{$\left(  n+10,n-8\right)  \ \ \ \left(  6,1\right)  $} & $v(a_1^2 d +4a_2d-16)=6$ & $\left(  c,n^{\prime}\right)  $\\\cmidrule{7-8}
&  &  &  &  &  & $v(a_1^2 d +4a_2d-16)\geq7$ & $\left(  c,n-8\right)  $\\\midrule


$\rm{IV}^{\ast}$ & $v(a_{1})=1$ &  & $\rm{I}_{4}^{\ast}$ & $\left(  8,14\right)  $ & $\left(  2,6\right)  $ & $v(P_{R,1})=6$ & $\left(  c,2\right)  $\\\cmidrule{7-8}
&  &  &  &  &  & $v(P_{R,1})\geq7$ & $\left(  c,4\right)  $\\\cmidrule{2-8}
& $v(a_{2})\geq2$ &  & $\rm{II}^{\ast}$ & $\left(  8,14\right)  $ & $\left(2,6\right)  $ &  & $\left(  c,1\right)  $\\\midrule


$\rm{III}^{\ast}$ & $v(a_{1})=1$ &  & $\rm{I}_{6}^{\ast}$ & $\left(  10,16\right)  $ & $\left(  3,6\right)  $ & $v(P_{R,1})=10$ or $v(P_{R,2})=10$ & $\left(2,2\right)  $\\\cmidrule{7-8}
&  &  &  &  &  & $v(P_{R,1})\geq11$ or $v(P_{R,2})\geq11$ & $\left(2,4\right)  $\\\cmidrule{2-8}
& $v(a_{1})\geq2$ & $v(a_{3})=3$ & $\rm{II}$ & $\left(  12,6\right)  $ & $\left(5,6\right)  $ &  & $\left(  2,1\right)  $\\\cmidrule{3-8}
&  & $v(a_{3})\geq4$ & $\rm{III}$ & $\left(  14,8\right)  $ & $\left(  15,9\right)$ &  & $\left(  2,2\right)  $\\\midrule


$\rm{II}^{\ast}$ & $v(a_{1})=1$ &  & $\rm{I}_{7}^{\ast}$ & $\left(  11,17\right)  $ & $\left(  3,6\right)  $ & $v(P_{R,1})=11$ or $v(P_{R,2})=11$ & $\left(1,2\right)  $\\\cmidrule{7-8}
&  &  &  &  &  & $v(P_{R,1})\geq12$ or $v(P_{R,2})\geq12$ & $\left(1,4\right)  $\\\cmidrule{2-8}
& $v(a_{1})\geq2$ & $v(a_{3})=3$ & $\rm{II}$ & $\left(  12,6\right)  $ & $\left(4,6\right)  $ &  & $\left(  1,1\right)  $\\\cmidrule{2-8}
& $v(a_{1})\geq2,$ & $v(P_{R,3})=12$ & $\rm{I}_{0}^{\ast}$ & $\left(  14,8\right)$ & $\left(  6,4\right)  $ & $v(P_{R,4})=4$ & $\left(  1,1\right)  $\\\cmidrule{7-8}
& $v(a_{3})\ge 4$ &  &  &  &  & $v(P_{R,4})\geq5$ & $\left(  1,2\right)  $\\\cmidrule{3-8}
&  & $v(P_{R,3})\geq13,v(P_{R,4})=4$ & $\rm{I}_{1}^{\ast}$ & $\left(  14,8\right)
$ & $\left(  6,3\right)  $ & $v(P_{R,3})=13$ & $\left(  c,2\right)  $\\\cmidrule{7-8}
&  &  &  &  &  & $v(P_{R,3})\geq14$ & $\left(  c,4\right)  $\\\cmidrule{3-8}
&  & $v(P_{R,3})\geq13,v(P_{R,4})\geq5$ & $\rm{IV}^{\ast}$ & $\left(  14,8\right)
$ & $\left(  6,2\right)  $ & $v(P_{R,3})=13$ & $\left(  c,2\right)  $\\\cmidrule{7-8}
&  &  &  &  &  & $v(P_{R,3})\geq14$ & $\left(  c,4\right)$
\end{longtable}
\endgroup}

\subsection{The case when \texorpdfstring{$v(d)=0$}{v(d)=0}}\label{sec5_1}

\begin{proof}[Proof of Theorem~\ref{thmQ2combined} for when $v(d)=0$]
Suppose an elliptic curve over $\Q_2$ is given by a \KNT model~$E$, and denote its Weierstrass coefficients by $a_{i}$. Then, we can take the quadratic twist $E^{d}$ of $E$ by $d$ to given by the Weierstrass model (\ref{Edmodel}). For a fixed $R=\operatorname*{typ}(E)$, let $F_{R,j}^{0}$ be the elliptic curve obtained from $E^{d}$ via the $\mathbb{Q}_{2}$-isomorphism $[u_{j},r_{j},s_{j},w_{j}]$, where $u_{j},r_{j},s_{j},w_{j}$ are as given in Table \ref{FRj}. In addition, for each $R$, we set $F_{R,0}^0$ to be the elliptic curve attained from $E^d$ via the $\mathbb{Q}_{2}$-isomorphism $[u_{0},r_{0},s_{0},w_{0}]$. For each of the conditions listed in Table~\ref{tab:localdata-dodd}, we associate a model $F_{R,j}^{0}$ as listed in Table~\ref{doddmodels}. In the SageMath \cite{sagemath} accompaniment to this proof \cite[Theorem5\_1.ipynb]{gittwists}, we verify that $\mathcal{V}(F_{R,j}^{0})$ is as given in the table. This allows us to conclude that $R^{d}=\operatorname*{typ}(E^{d})$ is as claimed by Proposition~\ref{lem:pnormal} since $E^{d}$ is $\mathbb{Q}_{2}$-isomorphic to $F_{R,j}^{0}$. This demonstrates that the pair $\left(  R,R^{d}\right)  $ is as given in Table~\ref{tab:localdata-dodd}.

{\begingroup \footnotesize
\renewcommand{\arraystretch}{1.5}
 \begin{longtable}{ccccccc}
 	\caption{Conditions to determine $\mathcal{V}(F_{R,j}^0)$}\\
	\hline
$R$ & \multicolumn{3}{c}{Conditions} & $j$ & $\mathcal{V}(F_{R,j}%
^{0})$ & $\operatorname*{typ}(F_{R,j}^{0})$   \\
	\hline
	\endfirsthead
	\hline
$R$ & \multicolumn{3}{c}{Conditions} & $j$ & $\mathcal{V}(F_{R,j}%
^{0})$ & $\operatorname*{typ}(F_{R,j}^{0})$  \\
	\hline
	\endhead
	\hline 
	\multicolumn{4}{r}{\emph{continued on next page}}
	\endfoot
	\hline 
	\endlastfoot
	
$\rm{I}_{0}$ & $d\equiv1\ \operatorname{mod}4$ &  \multicolumn{2}{l}{$v(a_1)=0$ and exactly of $v(a_{4}),v(a_{6})\geq1$}  & $1$
& $\left(  =0,0,\infty,0,0\right)  $ & $\rm{I}_{0}$ \\\cmidrule{3-7}
&  & $v(a_{1})\geq1$ &  & $2$ & $\left(  1,1,=0,0,0\right)  $ & $\rm{I}_{0}$\\\cmidrule{2-7}
& $d\equiv3\ \operatorname{mod}4$ & $v(a_{1})=0$ & $v(a_{4}%
)=0,v(a_{6})\geq1$ & $3$ & $\left(  =1,=1,\infty,=4,7\right)  $ &
$\rm{I}_{4}^{\ast}$ \\\cmidrule{4-7}
&  &  & $v(a_{4})\geq1,v(a_{6})=0$ & $4$ & $\left(
=1,=1,=4,=4,7\right)  $ & $\rm{I}_{4}^{\ast}$\\\cmidrule{3-7}
&  & $v(a_{1})\ge 1$ &  & $5$ & $\left(  \infty,2,=3,4,=5\right)  $ &
$\rm{II}^{\ast}$\\\hline

$\rm{I}_{n>0}$  & $d\equiv1\ \operatorname{mod}4$ & $v(n)=0$  &  & $0$ & $\left(=0,0,=\frac{n+1}{2},\frac{n+1}{2},=n\right)  $ & $\rm{I}_{n}$\\\cmidrule{3-7}
& & $v(n)\geq1$ &  & $0$ & $\left(=0,0,\frac{n+2}{2},=\frac{n}{2},n+1\right)  $ & $\rm{I}_{n}$\\\cmidrule{2-7}

&  $d\equiv3\ \operatorname{mod}4$ & $v(n) =0$  & & $1$ & $\left(  =1,=1,=\frac{n+7}{2},\frac{n+9}{2},n+7\right)  $ & $\rm{I}_{n+4}^{\ast}$\\\cmidrule{3-7}
&  & $v(n)\ge 1$ &  & $1$ & $\left(  =1,=1,\frac{n+8}{2},=\frac{n+8}{2},n+7\right)  $ & $\rm{I}_{n+4}^{\ast}$\\\hline

$\rm{II}$ & $v(a_{3})\geq2$ &  &  & $0$ & $\left(  1,1,2,1,=1\right)  $ &
$\rm{II}$ \\\cmidrule{2-7}
& $v(a_{3})=1$ & $d\equiv1\ \operatorname{mod}4$ &  & $0$ & $\left(
1,1,=1,1,=1\right)  $ & $\rm{II}$\\\cmidrule{2-7}
&  & $d\equiv3\ \operatorname{mod}4$ & $v(a_{4})=1$ & $0$ & $\left(
1,1,=1,=1,2\right)  $ & $\rm{III}$\\\cmidrule{4-7}
&  &  & $v(a_{4})\geq2$ & $0$ & $\left(  1,1,=1,2,2\right)  $ & $\rm{IV}$\\\hline

$\rm{III}$ & $d\equiv1\ \operatorname{mod}4$ &  &  & $0$ & $\left(
1,1,1,=1,2\right)  $ & $\rm{III}$\\\cmidrule{2-7}
& $d\equiv3\ \operatorname{mod}4$ & $v(a_{3})=1$ &  & $0$ & $\left(
1,1,=1,=1,=1\right)  $ & $\rm{II}$\\\cmidrule{3-7}
&  & $v(a_{3})\geq2$ &  & $0$ & $\left(  1,1,2,=1,2\right)  $ & $\rm{III}$\\\hline

$\rm{IV}$ & $d\equiv3\ \operatorname{mod}4$ &  &  & $0$ & $\left(
1,1,=1,2,=1\right)  $ & $\rm{II}$\\\cmidrule{2-7}
& $d\equiv1\ \operatorname{mod}4$ &  &  & $0$ & $\left(  1,1,=1,2,2\right)  $
& $\rm{IV}$\\\hline

$\rm{I}_{0}^{\ast}$ & $d\equiv1\ \operatorname{mod}4$ &  &  & $0$ & $\left(
1,1,2,3,=3\right)  $ & $\rm{I}_{0}^{\ast}$\\\cmidrule{2-7}
& $d\equiv3\ \operatorname{mod}4$ & $v(a_{3})\geq3$ &  & $0$ & $\left(
1,1,3,3,=3\right)  $ & $\rm{I}_{0}^{\ast}$\\\cmidrule{3-7}
&  & $v(a_{3})=2$ & $v(a_{1}+a_{2})=1$ & $0$ & $\left(
1,=1,=2,3,4\right)  $ & $\rm{I}_{1}^{\ast}$\\\cmidrule{4-7}
&  &  & $v(a_{1}+a_{2})\geq2$ & $0$ & $\left(  1,2,=2,3,4\right)  $ &
$\rm{IV}^{\ast}$\\\hline

$\rm{I}_{n>0}^{\ast}$ & $d\equiv1\ \operatorname{mod}4$ & $v(n)=0$ &  & $0$ & $\left(  1,=1,=\frac{n+3}{2},\frac{n+5}{2},n+3\right)  $ & $\rm{I}_{n}^{\ast}$\\\cmidrule{3-7}
&  & $v(n)\geq1$ &  & $0$ & $\left(  1,=1,\frac{n+4}{2},=\frac{n+4}{2},n+3\right)  $ & $\rm{I}_{n}^{\ast}$\\\cmidrule{2-7}
& $d\equiv3\ \operatorname{mod}4$ & $n=1$ &  & $0$ & $\left(1,1,=2,3,=3\right)  $ & $\rm{I}_{0}^{\ast}$\\\cmidrule{3-7}
&  & $n=2$ & $v(a_{1})=1$ & $0$ & $\left(  =1,2,3,=3,6\right)  $ & $\rm{III}^{\ast}$\\\cmidrule{3-7}
&  & $v(n)\geq1$ & $v(a_{1})\geq2$ & $0$ & $\left(  2,=1,\frac{n+4}{2},\frac{n+4}{2},n+3\right)  $ & $\rm{I}_{n}^{\ast}$\\\cmidrule{3-7}
&  & $n=3$ & $v(a_{1})=1$ & $0$ & $\left(  =1,2,=3,4,=5\right)  $ & $\rm{II}^{\ast}$ \\\cmidrule{3-7}
&  & $n\geq3$ odd & $v(a_{1})\geq2$ & $1$ & $\left( \infty,=1,=\frac{n+3}{2},\frac{n+5}{2},n+3\right)  $ & $\rm{I}_{n}^{\ast}$\\\cmidrule{3-7}
&  & $n=4$ & $v(a_{1})=1$ & $2$ & $\left(  =0,0,\infty,0,0|v(a_{4}^{d}+a_{6}^{d})=0\right)  $ & $\rm{I}_{0}$\\\cmidrule{3-7}
&  & $n\geq5$ odd & $v(a_{1})=1$ & $3$ & $\left(  =0,0,=\frac{n-3}{2},\frac{n-3}{2},=n-4\right)  $ & $\rm{I}_{n-4}$\\\cmidrule{3-7}
&  & $n\geq6$ even & $v(a_{1})=1$ & $3$ & $\left(  =0,0,\frac{n-2}{2},=\frac{n-4}{2},n-3\right)  $ & $\rm{I}_{n-4}$\\\hline

$\rm{IV}^{\ast}$ & $d\equiv1\ \operatorname{mod}4$ &  &  & $0$ & $\left(
1,2,=2,3,4\right)  $ & $\rm{IV}^{\ast}$\\\cmidrule{2-7}
& $d\equiv3\ \operatorname{mod}4$ &  &  & $0$ & $\left(  1,1,=2,3,=3\right)  $
& $\rm{I}_{0}^{\ast}$\\\hline

$\rm{III}^{\ast}$ & \multicolumn{2}{l}{$v(a_{1})\geq2$ or $d\equiv
1\ \operatorname{mod}4$} &  & $0$ & $\left(  2,2,3,=3,5\right)  $ &
$\rm{III}^{\ast}$\\\cmidrule{2-7}
& $v(a_{1})=1$ & $d\equiv3\ \operatorname{mod}4$ &  & $0$ & $\left(
=1,=1,3,=3,5\right)  $ & $\rm{I}_{2}^{\ast}$\\\hline

$\rm{II}^{\ast}$ & $d\equiv1\ \operatorname{mod}4$ &  &  & $0$ & $\left(
1,2,3,4,=5\right)  $ & $\rm{II}^{\ast}$\\\cmidrule{2-7}
& $d\equiv3\ \operatorname{mod}4$ & $v(a_{1})=1$ & $v(a_{3})=3$ & $0$
& $\left(  =1,=1,=3,4,6\right)  $ & $\rm{I}_{3}^{\ast}$\\\cmidrule{4-7}
&  &  & $v(a_{3})\geq4$ & $1$ & $\left(  =1,=1,=3,4,6\right)  $ &
$\rm{I}_{3}^{\ast}$\\\cmidrule{3-7}
&  & $v(a_{1})\geq2$ & $v(a_{3})=3$ & $2$ & $\left(
\infty,2,=0,0,0\right)  $ & $\rm{I}_{0}$\\\cmidrule{4-7}
&  &  & $v(a_{3})\geq4$ & $0$ & $\left(  2,2,4,4,=5\right)  $ & $\rm{II}^*$

\label{doddmodels}
\end{longtable}
\endgroup}

It remains to show that the pairs $\left(  \delta,\delta^{d}\right)  ,$ $\left(  f,f^{d}\right)  $, and $\left(  c,c^{d}\right)  $ are as given in Table \ref{tab:localdata-dodd}. By Lemma~\ref{Lemmadisccond}, we can uniquely determine $\delta$ and $f$ from $\mathcal{V}(E)$. Consequently, a case-by-case analysis of the conditions listed for each $R$ shows that $\delta$ and $f$ are as claimed in Table \ref{tab:localdata-dodd}. Since $d$ is odd, $\delta \equiv\delta^{d}$ $\operatorname{mod}12$. Next, since $R^{d}$ has been determined, Table \ref{tab:discriminantconductor} gives the possible pairs $\left(  \delta^{d},f^{d}\right)  $ for a fixed N\'{e}ron-Kodaira type. By inspection, the listed $\delta^{d}$ for a fixed N\'{e}ron-Kodaira type in the table are in unique congruence classes modulo $12$. Consequently, the knowledge of $R^{d}$ and the fact that $\delta\equiv\delta^{d}$ allows us to determine the exact value of $\delta^{d}$. This also gives the value of $f^{d}$. This case-by-case analysis shows that the pairs $\left(\delta,\delta^{d}\right)  $ and $\left(  f,f^{d}\right)  $ are as claimed in Table~\ref{tab:localdata-dodd}.

Since Proposition \ref{lem:pnormal} determines $c$, it remains to show that $c^{d}$ is as claimed. Observe that $c^d$ is uniquely determined when $R^d \in \left\{  \rm{I}_{0},II,III,II^{\ast},III^{\ast}\right\} $. Below, we consider the cases corresponding to $R^{d}\not \in \left\{  \rm{I}_{0},II,III,II^{\ast},III^{\ast}\right\}  $. Specifically, we split into cases corresponding to those $R$'s that lead to one of these $R^{d}$. We then consider the assumptions on the $a_{i}$'s which result in $R^{d}$ and the corresponding Weierstrass model $F_{R,j}^{0}$ determined by these assumptions as given in Table~\ref{doddmodels}. We denote the Weierstrass coefficients of $F_{R,j}^{0}$ by $a_{i}^{d}$. We note that this proof also makes reference to polynomials $P_{R,i}^0$ given in Table~\ref{tab:PRj}. Since $v(d)=0$ throughout this proof, we will write $F_{R,j}=F_{R,j}^0$ and $P_{R,i}=P_{R,i}^0$. 

\textbf{Case 1.} Let $R=\rm{I}_{0}$ and suppose that $d\equiv3\ \operatorname{mod}4$ with $v(a_{1})=0$. Then one of the following holds: $\left(  i\right)  $ $v(a_{4})=0$ and $v(a_{6})\geq1$ or $\left(  ii\right)  $ $v(a_{4})\geq1$ and $v(a_{6})=0$. From Table~\ref{doddmodels}, we have that $F_{R,3}$ (resp. $F_{R,4}$) is a \KNT model if $\left(  i\right)  $ (resp. $\left(ii\right)  $) holds. Moreover, $\operatorname*{typ}(E^{d})=\rm{I}_{4}^{\ast}$. By Corollary~\ref{TamagawaQ2}, $c^{d}=2$ (resp. $4$) if $v(a_{6}^{d})=7$ (resp. $v(a_{6}^{d})\geq8$). We now consider the subcases $\left(  i\right)  $ and $\left(  ii\right)  $ separately.

\qquad\textbf{Subcase 1a.} Suppose $v(a_{4})=0$ and $v(a_{6})\geq1$. Then $F_{R,3}$ is given by a \KNT model with $a_{6}^{d}=16d^{3}\left(  a_{3}^{2}+4a_{6}\right)  $. Now observe that
\[
\frac{a_{6}^{d}}{64d^{3}}=\left(  \frac{a_{3}}{2}\right)  ^{2}+a_{6}\equiv
a_{6}\ \operatorname{mod}4.
\]
It follows that $c^{d}=2$ (resp. $4$) if $a_{6}\equiv2\ \operatorname{mod}4$ (resp. $0\ \operatorname{mod}4$).

\qquad\textbf{Subcase 1b.} Suppose $v(a_{4})\geq1$ and $v(a_{6})=0$. Then $F_{R,4}$ is given by a \KNT model with $a_{6}^{d}=16\left(  a_{3}^{2}d^{3}+4a_{6}d^{3}-4\right)  $. Since
\[
\frac{a_{6}^{d}}{64}=\left(  \frac{a_{3}}{2}\right)  ^{2}d^{3}+a_{6}d^{3}-1\equiv3a_{6}-1\ \operatorname{mod}4=\left\{
\begin{array}
[c]{cl}
2\ \operatorname{mod}4 & \text{if }a_{6}\equiv1\ \operatorname{mod}4,\\
0\ \operatorname{mod}4 & \text{if }a_{6}\equiv3\ \operatorname{mod}4,
\end{array}
\right.
\]
it follows that $c^{d}=2$ (resp. $4$) if $a_{6}\equiv1\ \operatorname{mod}4$ (resp. $3\ \operatorname{mod}4$).

\textbf{Case 2.} Let $R=\operatorname*{typ}(E)=\rm{I}_{n>0}$. Then
\[
\mathcal{V}(E)=\left\{
\begin{array}
[c]{cl}
\left(  =0,0,=\frac{n+1}{2},\frac{n+1}{2},=n\right)   & \text{if }n\text{ is
odd,}\\
\left(  =0,0,\frac{n+2}{2},=\frac{n}{2},n+1\right)   & \text{if }n\text{ is
even.}
\end{array}
\right.
\]
From Table \ref{doddmodels}, we have that $F_{R,0}$ (resp. $F_{R,1}$) is a
\KNT model and $\operatorname{typ}(E^d)=\rm{I}_n$ (resp. $\rm{I}_{n+4}^*)$ if $d\equiv1\ \operatorname{mod}4$ (resp. $3\ \operatorname{mod}%
4$). We now consider these subcases separately.

\qquad\textbf{Subcase 2a.} Suppose $d\equiv1\ \operatorname{mod}4$. Then
$F_{R,0}$ is a \KNT model, and by Corollary~\ref{TamagawaQ2}, $c^{d}$ is determined by whether $v(a_{2}^{d})$ is zero or positive. Now observe that
\[
4a_{2}^{d}=a_{1}^{2}(d-1)+4a_{2}d\equiv d-1+4a_{2}\ \operatorname{mod}8.
\]
It follows that%
\[
c^{d}=\left\{
\begin{array}
[c]{cl}%
2-(n\ \operatorname{mod}2) & \text{if }v(d-1+4a_2)=2,\\
n & \text{if }v(d-1+4a_2)\ge 3.
\end{array}
\right.
\]

\qquad\textbf{Subcase 2b.} Suppose $d\equiv3\ \operatorname{mod}4$. Then
$F_{R,1}$ is a \KNT model, with Weierstrass coefficient $a_{6}^{d}=16\left(  a_{3}^{2}
(d^{3}-1)+4a_{6}d^{3}\right)  $. We now consider the cases when $n$ is even or
odd. If $n$ is odd, we have that%
\begin{align*}
\frac{a_{6}^{d}}{2^{n+5}} &  =\frac{\left(  a_{3}^{2}(d^{3}-1)+4a_{6}%
d^{3}\right)  }{2^{n+1}}\\
&  =\left(  \frac{a_{3}}{2^{\frac{n+1}{2}}}\right)  ^{2}\left(  d^{3}%
-1\right)  +2\left(  \frac{a_{6}}{2^{n}}\right)  d^{3}\\
&  \equiv d-1+2\left(  \frac{a_{6}}{2^{n}}\right)  d\ \operatorname{mod}8.
\end{align*}
Since $v(a_{6}^{d})=n+7$ is equivalent to $\frac{a_{6}^{d}}{2^{n+5}}%
\equiv4\ \operatorname{mod}8$, we conclude by Corollary \ref{TamagawaQ2} that%
\[
c^{d}=\left\{
\begin{array}
[c]{cl}%
2 & \text{if }v(2^{n-1}(d-1)+a_{6}d)=n+1,\\
4 & \text{if }v(2^{n-1}(d-1)+a_{6}d)\geq n+2.
\end{array}
\right.
\]
Now suppose that $n$ is even, and observe that%
\begin{align*}
\frac{a_{6}^{d}}{2^{n+6}}  & =\frac{\left(  a_{3}^{2}(d^{3}-1)+4a_{6}%
d^{3}\right)  }{2^{n+2}}\\
& =\left(  \frac{a_{3}}{2^{\frac{n+2}{2}}}\right)  ^{2}\left(  d^{3}-1\right)
+2\left(  \frac{a_{6}}{2^{n+1}}\right)  d^{3}\\
& \equiv\frac{a_{3}^{2}}{2^{n+1}}+2\left(  \frac{a_{6}}{2^{n+1}}\right)
\ \operatorname{mod}4.
\end{align*}
It follows from Corollary \ref{TamagawaQ2} that%
\[
c^{d}=\left\{
\begin{array}
[c]{cl}%
2 & \text{if }v(a_{3}^{2}+2a_{6})=n+2,\\
4 & \text{if }v(a_{3}^{2}+2a_{6})\geq n+3.
\end{array}
\right.
\]

\textbf{Case 3.} Let $R=\rm{II}$ and suppose that $\mathcal{V}(E)=\left(  1,1,=1,2,=1\right)  $. From Table~\ref{doddmodels}, we have that $F_{R,0}$ is a \KNT model for $E^{d}$ and $R^{d}=\rm{IV}$. By Corollary~\ref{TamagawaQ2}, $c^{d}$ is determined from the Weierstrass coefficient $a_{6}^{d}$ of $F_{R,0}$. In particular, $c^{d}=1$ (resp. $3$) if $v(a_{6}^{d})=2$ (resp. $\geq3$). The condition for $c_2^d$ as given in Table~\ref{tab:localdata-dodd} now follows since
\begin{equation}\label{eq:forII)star-to-IV}
a_{6}^d  = \frac{a_3^2(d^3-1)+4a_6d^3}{4} = \left(\frac{a_3}{2}\right)^2(d^3-1)+a_6d^3\equiv (d-1)+a_6d \equiv d+a_6d-1 \mod 8.
\end{equation}

\textbf{Case 4.} Let $R=\rm{IV}$ and suppose that $\mathcal{V}(E)=\left(  1,1,=1,2,2\right)  $ with $d\equiv1\ \operatorname{mod}4$. From Table \ref{doddmodels}, we have that $F_{R,0}$ is a \KNT model for $E^{d}$ and $R^{d}=\rm{IV}$. By Corollary~\ref{TamagawaQ2}, $c^{d}=1$ (resp. $3$) if $v(a_{6}^{d})=2$ (resp. $\geq3$). Since $a_6^d$ also satisfies the equality and congruence given in \eqref{eq:forII)star-to-IV}, we conclude that $c^d$ is as claimed.

\textbf{Case 5.} Let $R=\rm{I}_{0}^{\ast}$ and suppose that $\mathcal{V}(E)=\left(1,1,2,3,=3\right)  $. Below, we consider the four cases listed in Table~\ref{doddmodels}. We note that in each case, $F_{R,0}$ is a \KNT model for $E^{d}$. The Weierstrass coefficient $a_{2}^{d}$ satisfies
\begin{equation}
a_{2}^{d}=\left(  \frac{a_{1}}{2}\right)  ^{2}\left(  d-1\right)+a_{2}d\equiv\left\{
\begin{array}
[c]{cl}
a_{2}\ \operatorname{mod}4 & \text{if }d\equiv1\ \operatorname{mod}4,\\
a_{1}+a_{2}\ \operatorname{mod}4 & \text{if }d\equiv3\ \operatorname{mod}4.
\end{array}
\right.  \label{I0stardoddvalsa2d}
\end{equation}
We now proceed by cases.

\qquad\textbf{Case 5a.} Suppose $d\equiv1\ \operatorname{mod}4$ or $d\equiv3\ \operatorname{mod}4$ with $v(a_{3})\geq3$. From Table~\ref{doddmodels}, we have that $F_{R,0}$ is a \KNT model for $E^{d}$ and $R^{d}=\rm{I}_{0}^{\ast}$. By Corollary~\ref{TamagawaQ2}, $c^{d}=1$ (resp. $2$) if $v(a_{2}^{d})=1$ (resp. $\geq2$). If $d\equiv1\ \operatorname{mod}4$, then the result follows since (\ref{I0stardoddvalsa2d}) implies that $v(a_{2}^{d})\geq2$ if and only if $v(a_{2})\geq2$. Thus, $c^{d}=1$ (resp. $2$) if $v(a_{2})=1$ (resp. $\geq2$).

Similarly, if $d\equiv3\ \operatorname{mod}4$ with $v(a_{3})\geq3$, then (\ref{I0stardoddvalsa2d}) shows that $v(a_{2}^{d})\geq2$ if and only if $v(a_{1}+a_{2})\geq2$. Consequently,
\[
c^{d}=\left\{
\begin{array}
[c]{cll}
1 & \text{if }  v(a_{1}+a_2)=1,\\
2 & \text{if\ } \ v(a_{1}+a_2) \ge 2.
\end{array}
\right.
\]

\qquad\textbf{Case 5b.} Suppose that $d\equiv3\ \operatorname{mod}4$ and $v(a_{3})=2$ with $v(a_{1}+a_{2})=1$ (resp. $\geq2$). From Table~\ref{doddmodels}, we have that $F_{R,0}$ is a \KNT model for $E^{d}$ and $R^{d}=\rm{I}_{1}^{\ast}$ (resp. $\rm{IV}^{\ast}$). If $v(a_{1}+a_{2})=1$, then Corollary~\ref{TamagawaQ2} implies
\[
c^{d}=\left\{
\begin{array}
[c]{cl}
2 & \text{if }v(a_{6}^{d})=4,\\
4 & \text{if }v(a_{6}^{d})\geq5.
\end{array}
\right.
\]
Similarly, if $v(a_{1}+a_{2})\geq2$, then Corollary~\ref{TamagawaQ2} implies
\[
c^{d}=\left\{
\begin{array}
[c]{cl}%
1 & \text{if }v(a_{6}^{d})=4,\\
3 & \text{if }v(a_{6}^{d})\geq5.
\end{array}
\right.
\]
Next, we recall that in both cases, $v(a_{3})=2$, $v(a_{6})=3$, and $d\equiv3\ \operatorname{mod}4$. The quantity $a_{6}^{d}$ only depends on $a_{3},a_{6}$, and $d$. Thus, it suffices to determine conditions for when $v(a_{6}^{d})$ is $4$ or greater than $4$. To this end, observe that
\[
a_{6}^{d}=4\left(  \left(  \frac{a_{3}}{4}\right)  ^{2}\left(  d^{3}-1\right) +\frac{a_{6}d}{4}\right)  .
\]
Since
\[
\left(  \frac{a_{3}}{4}\right)  ^{2}\left(  d^{3}-1\right)  +\frac{a_{6}d}%
{4}\equiv\left(  d-1\right)  +\frac{a_{6}d}{4}\ \operatorname{mod}8,
\]
we conclude that
\[
v(a_{6}^{d})=\left\{
\begin{array}
[c]{cl}
4 & \text{if }v(4(d-1)+a_{6}d)=4,\\
\geq5 & \text{if }v(4(d-1)+a_{6}d)\geq5.
\end{array}
\right.
\]
Thus, the local Tamagawa number $c^{d}$ is as claimed.

\textbf{Case 6.} Let $R=\rm{I}^*_{n>0}$. Since $E$ is given by a \KNT model, we have that
\[
\mathcal{V}(E)=\left\{
\begin{array}
[c]{cl}%
\left(  1,=1,=\frac{n+3}{2},\frac{n+5}{2},n+3\right)  & \text{if }n\text{ is
odd,}\\
\left(  1,=1,\frac{n+4}{2},=\frac{n+4}{2},n+3\right)  & \text{if }n\text{ is
even.}%
\end{array}
\right.
\]

We now proceed by cases, following the order in Table~\ref{doddmodels}.

\qquad\textbf{Subcase 6a.} Suppose that $d\equiv1\ \operatorname{mod}4$ with
$v(n)=0$. By Table~\ref{doddmodels}, $F_{R,0}$ is a \KNT model for $E^{d}$
and $\operatorname*{typ}(E^{d})=\rm{I}_{n}^{\ast}$. Now observe that
\begin{align*}
a_{6}^{d} &  =\frac{1}{4}\left(  a_{3}^{2}\left(  d^{3}-1\right)  +4a_{6}%
d^{3}\right)  \\
&  =2^{n+1}\left(  \left(  \frac{a_{3}}{2^{\frac{n+3}{2}}}\right)  ^{2}\left(
d^{3}-1\right)  +\frac{a_{6}}{2^{n+1}}d^{3}\right)  .
\end{align*}
Since $v\!\left(  \frac{a_{6}}{2^{n+1}}\right)  \geq2$, we have that%
\[
\left(  \frac{a_{3}}{2^{\frac{n+3}{2}}}\right)  ^{2}\left(  d^{3}-1\right)
+\frac{a_{6}}{2^{n+1}}d^{3}\equiv d-1+\frac{a_{6}}{2^{n+1}}%
\ \operatorname{mod}8.
\]
By Corollary \ref{TamagawaQ2}, we conclude that%
\[
c^{d}=\left\{
\begin{array}
[c]{cl}%
2 & \text{if }v(2^{n+1}(d-1)+a_{6})=n+3,\\
4 & \text{if }v(2^{n+1}(d-1)+a_{6})\geq n+4.
\end{array}
\right.
\]

\qquad\textbf{Subcase 6b.} Suppose that $d\equiv1\ \operatorname{mod}4$ with
$v(n)\geq1$. By Table \ref{doddmodels}, $F_{R,0}$ is a \KNT model for
$E^{d}$ and $\operatorname*{typ}(E^{d})=\rm{I}_{n}^{\ast}$. Now observe that%
\begin{align}
a_{6}^{d} &  =\frac{1}{4}\left(  a_{3}^{2}\left(  d^{3}-1\right)  +4a_{6}%
d^{3}\right)  \nonumber\\
&  =2^{n+2}\left(  \left(  \frac{a_{3}}{2^{\frac{n+4}{2}}}\right)  ^{2}\left(
d^{3}-1\right)  +\frac{a_{6}}{2^{n+2}}d^{3}\right)  .\label{Instar6b}
\end{align}
Since%
\[
\left(  \frac{a_{3}}{2^{\frac{n+4}{2}}}\right)  ^{2}\left(  d^{3}-1\right)
+\frac{a_{6}}{2^{n+2}}d^{3}\equiv\frac{a_{6}}{2^{n+2}}\ \operatorname{mod}4,
\]
Corollary \ref{TamagawaQ2} implies tha%
\[
c^{d}=\left\{
\begin{array}
[c]{cl}%
2 & \text{if }v(a_{6})=n+3,\\
4 & \text{if }v(a_{6})\geq n+4.
\end{array}
\right.
\]

\qquad\textbf{Subcase 6c.} Suppose that $d\equiv3\ \operatorname{mod}4$ with
$n=1$. In particular, $\mathcal{V}(E)=\left(  1,=1,=2,3,4\right)  $. By Table
\ref{doddmodels}, $F_{R,0}$ is a \KNT model for $E^{d}$ and
$\operatorname*{typ}(E^{d})=\rm{I}_{0}^{\ast}$. Now observe that%
\begin{align*}
a_{2}^{d}  & =\left(  \frac{a_{1}}{2}\right)  ^{2}\left(  d-1\right)
+a_{2}d\equiv2\left(  \frac{a_{1}}{2}\right)  ^{2}+2\ \operatorname{mod}4\\
& \Longrightarrow\qquad v(a_{2}^{d})=\left\{
\begin{array}
[c]{cl}%
1 & \text{if }v(a_{1})\geq2,\\
\geq2 & \text{if }v(a_{1})=1.
\end{array}
\right.
\end{align*}
This shows that $c^d$ is as claimed by Corollary~\ref{TamagawaQ2}.

\qquad\textbf{Subcase 6d.} Suppose that $d\equiv3\ \operatorname{mod}%
4,\ v(n)\geq1,$ and $v(a_{1})\geq2$. By Table \ref{doddmodels}, $F_{R,0}$ is a
\KNT model for $E^{d}$ and $\operatorname*{typ}(E^{d})=\rm{I}_{n}^{\ast}$. From
(\ref{Instar6b}), we obtain
\[
\frac{a_{6}^{d}}{2^{n+2}}=\left(  \frac{a_{3}}{2^{\frac{n+4}{2}}}\right)
^{2}\left(  d^{3}-1\right)  +\frac{a_{6}}{2^{n+2}}d^{3}\equiv\frac{a_{3}^{2}%
}{2^{n+3}}+\frac{a_{6}}{2^{n+2}}\ \operatorname{mod}4.
\]
We conclude from Corollary \ref{TamagawaQ2} that%
\[
c^{d}=\left\{
\begin{array}
[c]{cl}%
2 & \text{if }v(a_{3}^{2}+2a_{6})=n+4,\\
4 & \text{if }v(a_{3}^{2}+2a_{6})\geq n+5.
\end{array}
\right.
\]

\qquad\textbf{Subcase 6e.} Suppose $d\equiv3\ \operatorname{mod}4,$
$v(a_{1})\geq2$, and $n\geq3$ is odd. Thus, $\mathcal{V}(E)=\left(
2,=1,=\frac{n+3}{2},\frac{n+5}{2},n+3\right)  $. By Table~\ref{doddmodels}, $F_{R,1}$ is a \KNT model for $E^{d}$ and $\operatorname*{typ}(E^{d})=\rm{I}_{n}^{\ast}$. To compute $c^{d}$, we must determine when equality holds in $v(a_{6}^{d})\geq n+3$ by Corollary~\ref{TamagawaQ2}. To this end, we have that
\begin{align*}
16a_{6}^{d}  & =a_{1}a_{3}^{2}d(a_{1}+4d)+4a_{3}^{2}(d^{3}-1+a_{2}%
d)+8a_{3}a_{4}d^{2}+16a_{6}d^{3}+2a_{3}^{3}\\
& =2^{n+5}\left(  \frac{a_{1}a_{3}^{2}d}{2^{n+5}}(a_{1}+4d)+\frac{a_{3}^{2}%
}{2^{n+3}}(d^{3}-1+a_{2}d)+\frac{a_{3}a_{4}d^{2}}{2^{n+2}}+\frac{a_{6}d^{3}%
}{2^{n+1}}+\frac{a_{3}^{3}}{2^{n+4}}\right)  .
\end{align*}
Considering the cases $v(a_{1})=2$ and $v(a_{1})\geq3$, it is easily checked that $v\!\left(  \frac{a_{1}a_{3}^{2}d}{2^{n+5}}(a_{1}+4d)\right)\geq3$. Thus,
\begin{align*}
\frac{a_{6}^{d}}{2^{n+1}}  & \equiv d^{3}-1+a_{2}d+\frac{a_{3}a_{4}d^{2}%
}{2^{n+2}}+\frac{a_{6}d^{3}}{2^{n+1}}+\left(  \frac{a_{3}}{2^{\frac{n+3}{2}}}\right)  ^{2}\frac{a_{3}}{2}\ \operatorname{mod}8\\
& \equiv d-1+a_{2}d+\frac{a_{3}a_{4}}{2^{n+2}}+\frac{a_{6}}{2^{n+1}}%
+\frac{a_{3}}{2}\ \operatorname{mod}8.
\end{align*}
It follows from Corollary~\ref{TamagawaQ2} that
\[
c^{d}=\left\{
\begin{array}
[c]{cl}%
2 & \text{if }v(2^{n+2}(d-1+a_{2}d)+a_{3}a_{4}+2a_{6}+2^{n+1}a_{3})=n+4,\\
4 & \text{if }v(2^{n+2}(d-1+a_{2}d)+a_{3}a_{4}+2a_{6}+2^{n+1}a_{3})\geq n+5.
\end{array}
\right.
\]

\textbf{Subcase 6f.} Suppose $d\equiv3\ \operatorname{mod}4,\ n\geq5,$ and $v(a_{1})=1$. By Table~\ref{doddmodels}, $F_{R,3}$ is a \KNT model for $E^{d}$ and $\operatorname*{typ}(E^{d})=\rm{I}_{n-4}$. To compute $c^{d}$, it suffices to determine when $a_{2}^{d}$ is even by Corollary~\ref{TamagawaQ2}.
Since
\begin{align*}
4a_{2}^{d}  & =\left(  \frac{a_{1}}{2}\right)  ^{2}\left(  d-1\right)
+a_{2}d\\
& \equiv d-1+a_{2}d\ \operatorname{mod}8,
\end{align*}
we conclude by Corollary \ref{TamagawaQ2} that
\[
c^{d}=\left\{
\begin{array}
[c]{cl}
2-\left(  n\ \operatorname{mod}2\right)   & \text{if }v(d-1+a_{2}d)=2,\\
n-4 & \text{if }v(d-1+a_{2}d)\geq3.
\end{array}
\right.
\]

\textbf{Case 7.} Let $R=\rm{IV}^{\ast}$. Then $\mathcal{V}(E)=\left(
1,2,=2,3,4\right)  $ and by Table~\ref{doddmodels}, $F_{R,0}$ is a \KNT
model for $E^{d}$ and
\[
\operatorname*{typ}(E^{d})=\left\{
\begin{array}
[c]{cl}%
\rm{IV}^{\ast} & \text{if }d\equiv1\ \operatorname{mod}4,\\
\rm{I}_{0}^{\ast} & \text{if }d\equiv3\ \operatorname{mod}4.
\end{array}
\right.
\]
We now consider these two cases separately.

\qquad\textbf{Subcase 7a.} Suppose $d\equiv1\ \operatorname{mod}4$ so that
$\operatorname*{typ}(E^{d})=\rm{IV}^{\ast}$. Now observe that
\begin{align*}
\frac{a_{6}^{d}}{4}  & =\left(  \frac{a_{3}}{4}\right)  ^{2}\left(
d^{3}-1\right)  +\frac{a_{6}}{4}d^{3}\\
& \equiv d-1+\frac{a_{6}}{4}\ \operatorname{mod}8.
\end{align*}
We conclude from Corollary~\ref{TamagawaQ2} that
\[
c^{d}=\left\{
\begin{array}
[c]{cl}
1 & \text{if }v(4d-4+a_{6})=4,\\
3 & \text{if }v(4d-4+a_{6})\geq5.
\end{array}
\right.
\]

\qquad\textbf{Subcase 7b.} Suppose $d\equiv3\ \operatorname{mod}4$ so that
$\operatorname*{typ}(E^{d})=\rm{I}_{0}^{\ast}$. Since
\[
a_{2}^{d}=\left(  \frac{a_{1}}{2}\right)  ^{2}\left(  d-1\right)
+a_{2}d\equiv\frac{a_{1}^{2}}{4}\ \operatorname{mod}4,
\]
it follows from Corollary~\ref{TamagawaQ2} that
\[
c^{d}=\left\{
\begin{array}
[c]{cl}
1 & \text{if }v(a_{1})=1,\\
2 & \text{if }v(a_{1})\geq2.
\end{array}
\right.
\]

\textbf{Case 8.} Let $R=\rm{III}^{\ast}$, suppose that $v(a_{1})=1$ with
$d\equiv3\ \operatorname{mod}4$. Since $E$ is given by a \KNT model, we have that $\mathcal{V}(E)=\left(  =1,2,3,=3,5\right)  $. From Table~\ref{doddmodels}, $F_{R,0}$ is a \KNT model for $E^{d}$ and $\operatorname*{typ}(E^{d})=\rm{I}_{2}^{\ast}$. Now observe that
\begin{align*}
\frac{a_{6}^{d}}{16}  & =\left(  \frac{a_{3}}{8}\right)  ^{2}\left(
d^{3}-1\right)  +\frac{a_{6}}{16}d^{3}\\
& \equiv2\left(  \frac{a_{3}}{8}\right)  ^{2}+\frac{a_{6}}{16}\ \operatorname{mod}4.
\end{align*}
We can now conclude from Corollary~\ref{TamagawaQ2} that
\[
c^{d}=\left\{
\begin{array}
[c]{cl}
2 & \text{if }v(a_{3}^{2}+2a_{6})=6,\\
4 & \text{if }v(a_{3}^{2}+2a_{6})\geq7.
\end{array}
\right.
\]

\textbf{Case 9.} Let $R=\rm{II}^{\ast}$ and suppose that $v(a_{1})=1$ with
$v(a_{3})\geq3$. In particular, $\mathcal{V}(E)=\left(  =1,2,3,4,=5\right)  $.
By Table \ref{doddmodels}, $F_{R,0}$ (resp. $F_{R,1}$) is a \KNT model for $E^{d}$ if $v(a_{3})=3$ (resp. $\geq4$). In both cases, $\operatorname*{typ}(E^{d})=\rm{I}_{3}^{\ast}$. By Corollary \ref{TamagawaQ2}, it suffices to determine
when equality occurs in $v(a_{6}^{d})\geq6$. We now consider these two cases separately.

\qquad\textbf{Subcase 9a.} Suppose $v(a_{3})=3$, and consider the \KNT model
$F_{R,0}$. Then
\begin{align*}
\frac{a_{6}^{d}}{16} &  =\left(  \frac{a_{3}}{8}\right)  ^{2}\left(
d^{3}-1\right)  +\frac{a_{6}}{16}d^{3}\\
&  \equiv d-1+\frac{a_{6}}{16}d\ \operatorname{mod}8.
\end{align*}
Consequently,
\[
c^{d}=\left\{
\begin{array}
[c]{cl}
2 & \text{if }v(16d-16+a_{6}d)=6,\\
4 & \text{if }v(16d-16+a_{6}d)\geq7.
\end{array}
\right.
\]

\qquad\textbf{Subcase 9b.} Suppose $v(a_{3})\geq4$, and consider the \KNT
model $F_{R,1}$. Then%
\begin{align*}
\frac{a_{6}^{d}}{16}  & =\left(  \frac{a_{1}}{2}\right)  ^{10}a_{1}
^{2}+\left(  \frac{a_{1}}{2}\right)  ^{5}\frac{a_{3}}{4}d^{2}+\left(
\frac{a_{1}}{2}\right)  ^{8}a_{2}d+\left(  \frac{a_{1}}{2}\right)  ^{4}
\frac{a_{4}}{4}d^{2}+\left(  \frac{a_{3}}{8}\right)  ^{2}d^{3}+\left(
\frac{a_{1}}{2}\right)  ^{10}d-\left(  \frac{a_{6}}{2^{5}}\right)  ^{2}
+\frac{a_{6}}{16}d^{3}\\
& \equiv4+\left(  \frac{a_{1}}{2}\right)  \frac{a_{3}}{4}+a_{2}d+\frac{a_{4}
}{4}+\left(  \frac{a_{3}}{8}\right)  ^{2}d+d-1+\frac{a_{6}}{16}
d\ \operatorname{mod}8\\
& \equiv3+a_{2}d+\frac{a_{4}}{4}+d+\frac{a_{6}}{16}d\ \operatorname{mod}8.
\end{align*}
We conclude from Corollary \ref{TamagawaQ2} that
\[
c^{d}=\left\{
\begin{array}
[c]{cl}
2 & \text{if }v(48+16a_{2}d+4a_{4}+16d+a_{6}d)=6,\\
4 & \text{if }v(48+16a_{2}d+4a_{4}+16d+a_{6}d)\geq7.
\end{array}
\right.
\]

\end{proof}

\subsection{The case when \texorpdfstring{$v(d)=1$}{v(d)=1}}\label{sec5_2}

\begin{proof}[Proof of Theorem~\ref{thmQ2combined} for when $v(d)=1$] 
Let $E$ denote a \KNT model for an elliptic curve defined over $\Q_2$, and denote its Weierstrass coefficients by $a_{i}$. Next, let $E^{d}$ denote the elliptic curve with Weierstrass model (\ref{Edmodel}). In particular, $E^{d}$ is a model for the quadratic twist of $E$ by $d$. For a fixed $R=\operatorname*{typ}(E)$, let $F_{R,j}^{1}$ denote the elliptic curve obtained from $E^{d}$ via the $\mathbb{Q}_{2}$-isomorphism $[u_{j},r_{j},s_{j},w_{j}]$, where $u_{j},r_{j},s_{j},w_{j}$ are as given in Table \ref{FRj}. 
In addition, for each $R$, we set $F_{R,0}^1$ to be the elliptic curve attained from $E^d$ via the $\mathbb{Q}_{2}$-isomorphism $[u_{0},r_{0},s_{0},w_{0}]$.
For each of the conditions listed in Table~\ref{tab:localdata-deven}, we associate a model $F_{R,j}^{1}$ as listed in Table~\ref{tab:KNtypeEdeven}. In the SageMath \cite{sagemath} accompaniment to this proof \cite[Theorem5\_2.ipynb]{gittwists}, we verify that $\mathcal{V}(F_{R,j}^{1})$ is as given in the table. Consequently, by Proposition~\ref{lem:pnormal}, we conclude that $R^{d}=\operatorname*{typ}(E^{d})$ is as claimed since $E^{d}$ is $\mathbb{Q}_{2}$-isomorphic to $F_{R,j}^{1}$. This shows that the pair $(R,R^{d})$ is as given in Table \ref{tab:localdata-deven}.

{\begingroup \footnotesize
\renewcommand{\arraystretch}{1.5}
 \begin{longtable}{ccccccc}
 	\caption{Conditions to determine $\mathcal{V}(F_{R,j}^{1})$}
    \label{tab:KNtypeEdeven}\\
	\hline
$R$ & \multicolumn{3}{c}{Conditions} & $j$ & $\mathcal{V}(F_{R,j}%
^{1})$ & $\operatorname*{typ}(F_{R,j}^{1})$   \\
	\hline
	\endfirsthead
	\hline
$R$ & \multicolumn{3}{c}{Conditions} & $j$ & $\mathcal{V}(F_{R,j}^{1})$ & $\operatorname*{typ}(F_{R,j}^{1})$  \\
	\hline
	\endhead
	\hline 
	\multicolumn{4}{r}{\emph{continued on next page}}
	\endfoot
	\hline 
	\endlastfoot
$\rm{I}_{0}$ & $v(a_{1})=0$ & $v(a_{6})=0$ &  & $1$ & $\left(  \infty,=1,6,=6,11\right)  $ & $\rm{I}_{8}^{\ast}$\\\cmidrule{3-7}
&  & $v(a_{6})\geq1$ &  & $2$ & $\left(  \infty,=1,6,=6,11\right)  $ &$\rm{I}_{8}^{\ast}$\\\cmidrule{2-7}
& $v(a_{1})\geq1$ &  &  & $0$ & $\left(  \infty,1,\infty,2,=1\right)  $ &
$\rm{II}$\\\hline


 $\rm{I}_{n>0}$ & $v(n)=0$ & $v(P_{R,1}^{1})=n+3$ &  & $1$ & $\left(  \infty
,=1,=\frac{n+11}{2},\frac{n+13}{2},n+11\right)  $ & $\rm{I}_{n+8}^{\ast}$\\\cmidrule{3-7}
&  & $v(P_{R,1}^{1})\geq n+4$ &  & $2$ & $\left(  \infty,=1,=\frac{n+11}%
{2},\frac{n+13}{2},n+11\right)  $ & $\rm{I}_{n+8}^{\ast}$\\\cmidrule{2-7}
& $v(n)\geq1$ & $v(P_{R,2}^{1})=n+1$ &  & $3$ & $\left(  \infty,=1,\infty,=\frac{n+12}{2},n+11\right)  $ & $\rm{I}_{n+8}^{\ast}$\\\cmidrule{3-7}
&  & $v(P_{R,2}^{1})\geq n+2$ &  & $4$ & $\left(  \infty,=1,\infty,=\frac{n+12}{2},n+11\right)  $ & $\rm{I}_{n+8}^{\ast}$\\\hline


$\rm{II}$ & $v(a_{3})=1$ &  &  & $0$ & $\left(  \infty,1,\infty,3,=3\right)  $ & $\rm{I}_{0}^{\ast}$\\\cmidrule{2-7}
& $v(a_{3})\geq2$ & $v(a_{1})=1$ & $v(a_{4})=1$ & $1$ & $\left(
\infty,=1,=3,=3,5\right)  $ & $\rm{I}_{2}^{\ast}$\\\cmidrule{4-7}
&  &  & $v(a_{4})\geq2,(P_{R,1})=4$ & $2$ & $\left(  \infty,=1,=3,4,6\right) $ & $\rm{I}_{3}^{\ast}$\\\cmidrule{4-7}
&  &  & $v(a_{4})\geq2,(P_{R,1})\geq5$ & $3$ & $\left(  2,=1,=3,4,6\right)  $ & $\rm{I}_{3}^{\ast}$\\\cmidrule{3-7}
&  & $v(a_{1})\geq2$ & $v(a_{4})=1$ & $1$ & $\left(  \infty,2,=3,=3,5\right) $ & $\rm{III}^{\ast}$\\\cmidrule{3-7}
&   \multicolumn{3}{r}{ $v(a_{1})\geq2, v(a_{4})\geq2,v(a_{3}^{2}+4a_{6}-4d)=4$} & $1$ & $\left(\infty,2,=3,4,=5\right)  $ & $\rm{II}^{\ast}$\\\cmidrule{3-7}
&   \multicolumn{3}{r}{ $v(a_{1})\geq2, v(a_{4})\geq2,v(a_{3}^{2}+4a_{6}-4d)\geq5$} & $4$ & $\left(\infty,1,=0,0,0\right)  $ & $\rm{I}_{0}$\\\hline


$\rm{III}$ & $v(a_{3})=1$ &  &  & $0$ & $\left(  \infty,1,\infty,3,=3\right)  $ &
$\rm{I}_{0}^{\ast}$\\\cmidrule{2-7}
& $v(a_{3})\geq2$ & $v(a_{1})=1$ &  & $0$ & $\left(  \infty,=1,\infty
,=3,5\right)  $ & $\rm{I}_{2}^{\ast}$\\\cmidrule{3-7}
&  & $v(a_{1})\geq2$ &  & $0$ & $\left(  \infty,2,\infty,=3,5\right)  $ &
$\rm{III}^{\ast}$\\\hline


$\rm{IV}$ &  &  &  & $0$ & $\left(  \infty,1,\infty,3,=3\right)  $ & $\rm{I}_{0}^{\ast}$\\\hline


$\rm{I}_{0}^{\ast}$ & $v(a_{1})=1$ & $v(a_{3})=2$ &  & $1$ & $\left(\infty,=1,4,=4,7\right)  $ & $\rm{I}_{4}^{\ast}$\\\cmidrule{3-7}
&  &  & $v(P_{R,2})=6$ & $2$ & $\left(  \infty,=1,=4,5,8\right)  $ &$\rm{I}_{5}^{\ast}$\\\cmidrule{4-7}
&  & $v(a_{3})\geq3$ & $v(P_{R,2})\geq7$ & $3$ & $\left(  \infty,=1,=4,5,8\right)  $ & $\rm{I}_{5}^{\ast}$\\\cmidrule{2-7}
& $v(a_{1})\geq2$ & $v(a_{3})=2$ &  & $0$ & $\left(  \infty,2,\infty,5,=5\right)  $ & $\rm{II}^{\ast}$\\\cmidrule{2-7}
& $v(a_{1})\geq2,$ & \multicolumn{2}{c}{$v(a_{3}^{2}d+4a_{6}d-64)=7$} & $4$ &$\left(  1,1,=1,1,=1\right)  $ & $\rm{II}$\\\cmidrule{3-7}
& $v(a_{3})\geq3$ & \multicolumn{2}{c}{$v(a_{3}^{2}d+4a_{6}d-64)\geq 8,v(a_{4}-4a_{2})=3$} & $4$ & $\left(  1,1,=1,=1,2\right)  $ & $\rm{III}$\\\cmidrule{3-7}
&  & \multicolumn{2}{c}{$v(a_{3}^{2}d+4a_{6}d-64)\geq8,v(a_{4}-4a_{2})\geq4$} & $4$ & $\left(  1,1,=1,2,2\right)  $ & $\rm{IV} $\\\hline


 $\rm{I}_{n>0}^{\ast}$ & $n=1$ & $v(a_{1})=1$ &  & $1$ & $\left(  \infty,=1,\infty,=4,7\right)  $ & $\rm{I}_{4}^{\ast}$\\\cmidrule{3-7}
&  & $v(a_{1})\geq2$ &  & $1$ & $\left(  \infty,3,\infty,=4,=5\right)  $ & $\rm{II}^{\ast}$\\\cmidrule{2-7}
& $n=2$ & $v(a_{1})\geq2$ & $v(a_3)=3$ & $2$ & $\left(=1,1,=2,=1,=1\right)  $ & $\rm{II}$\\\cmidrule{4-7}
&  &  & $v(a_3)\geq4$ & $2$ & $\left(  =1,1,3,=1,2\right)  $ & $\rm{III}$\\\cmidrule{2-7}
& $n=3$ & $v(a_{1})\geq2$ &  & $2$ & $\left(  =1,1,=2,2,=1\right)  $ & $\rm{II}$\\\cmidrule{2-7}
& $n=4,$ & $v(P_{R,3})=4$ &  & $3$ & $\left(  =1,1,=2,3,=3\right)  $ & $\rm{I}_{0}^{\ast}$\\\cmidrule{3-7}
& $v(a_{1})\geq2$ & $v(P_{R,3})\geq5$ & $v(P_{R,4})=5$ & $3$ & $\left(=1,=1,=2,3,4\right)  $ & $\rm{I}_{1}^{\ast}$\\\cmidrule{4-7}
&  &  & $v(P_{R,4})\geq6$ & $3$ & $\left(  =1,2,=2,3,4\right)  $ & $\rm{IV}^{\ast}$\\\cmidrule{2-7}
& $n=5$ & $v(a_{1})\geq2$ &  & $2$ & $\left(  =1,1,=3,3,=3\right)  $ & $\rm{I}_{0}^{\ast}$\\\cmidrule{2-7}
& $n=6$ & $v(a_{1})\geq2$ & $v(a_1^2 +4a_2-4d)\geq5$ & $2$ & $\left(=1,2,4,=3,5\right)  $ & $\rm{III}^{\ast}$\\\cmidrule{2-7}
& $n=7$ & $v(a_{1})\geq2$ & $v(a_1^2 +4a_2-4d)\geq5$ & $2$ & $\left(=1,2,=4,4,=5\right)  $ & $\rm{II}^{\ast}$\\\cmidrule{2-7}
& $n=8$ & $v(a_{1})\geq2$ & $v(a_1^2 +4a_2-4d)\geq5$ & $4$ & $\left(  =0,0,\infty,=0,1\right)  $ & $\rm{I}_{0}$\\\cmidrule{2-7}
& $n\geq2,$ & $v(n)\geq1$ & $v(P_{R,1})=n+7$ & $5$ & $\left(  \infty,=1,=\frac{n+8}{2},=\frac{n+8}{2},n+7\right)  $ & $\rm{I}_{n+4}^{\ast}$\\\cmidrule{4-7}
& $v(a_{1})=1$ &  & $v(P_{R,1})\geq n+8$ & $1$ & $\left(  \infty,=1,\infty,=\frac{n+8}{2},n+7\right)  $ & $\rm{I}_{n+4}^{\ast}$\\\cmidrule{3-7}
&  & $v(n)=0$ & $v(P_{R,8})=n+7$ & $6$ & $\left(  \infty,=1,=\frac{n+7}{2},\frac{n+9}{2},n+7\right)  $ & $\rm{I}_{n+4}^{\ast}$\\\cmidrule{4-7}
&  &  & $v(P_{R,8})\geq n+8$ & $7$ & $\left(  \infty,=1,=\frac{n+7}{2},\frac{n+9}{2},n+7\right)  $ & $\rm{I}_{n+4}^{\ast}$\\\cmidrule{2-7}
& $n\geq6,$ & $v(P_{R,6})=4$ & $v(n)\geq1$ & $2$ & $\left(  =1,=1,\frac{n+2}{2},=\frac{n}{2},n-1\right)  $ & $\rm{I}_{n-4}^{\ast}$\\\cmidrule{4-7}
& $v(a_{1})\geq2$ &  & $v(n)=0$ & $8$ & $\left(  =1,=1,=\frac{n-1}{2},\frac{n+1}{2},n-1\right)  $ & $\rm{I}_{n-4}^{\ast}$\\\cmidrule{2-7}
& $n\geq9,$ & $v(P_{R,6})\geq5$ & $v(n)=0$ & $9$ & $\left(  =0,0,=\frac{n-7}{2},\frac{n-7}{2},=n-8\right)  $ & $\rm{I}_{n-8}$\\\cmidrule{4-7}
& $v(a_{1})\geq2$ &  & $v(n)\geq1$ & $9$ & $\left(  =0,0,\frac{n-6}{2},=\frac{n-8}{2},n-6\right)  $ & $\rm{I}_{n-8}$\\\hline


$\rm{IV}^{\ast}$ & $v(a_{1})=1$ &  &  & $1$ & $\left(  \infty,=1,=4,=4,7\right)  $
& $\rm{I}_{4}^{\ast}$\\\cmidrule{2-7}
& $v(a_{2})\geq2$ &  &  & $0$ & $\left(  \infty,3,\infty,5,=5\right)  $ &
$\rm{II}^{\ast}$\\\hline


$\rm{III}^{\ast}$ & $v(a_{1})=1$ & $v(P_{R,1})=9$ &  & $1$ & $\left(\infty,=1,=5,=5,9\right)  $ & $\rm{I}_{6}^{\ast}$\\\cmidrule{3-7}
&  & $v(P_{R,1})\geq10$ &  & $2$ & $\left(  \infty,=1,\infty,=5,9\right)  $ & $\rm{I}_{6}^{\ast}$\\\cmidrule{2-7}
& $v(a_{1})\geq2$ & $v(a_{3})=3$ &  & $3$ & $\left(  \infty,1,\infty,=1,=1\right) $ & $\rm{II}$\\\cmidrule{3-7}
&  & $v(a_{3})\geq4$ &  & $3$ & $\left(  \infty,1,\infty,=1,2\right)  $ &
$\rm{III}$\\\hline


$\rm{II}^{\ast}$ & $v(a_{1})=1$ & $v(P_{R,1})=10$ &  & $1$ & $\left(\infty,=1,=5,6,10\right)  $ & $\rm{I}_{7}^{\ast}$\\\cmidrule{3-7}
&  & $v(P_{R,1})\geq11$ &  & $2$ & $\left(  \infty,=1,=5,6,10\right)  $ & $\rm{I}_{7}^{\ast}$\\\cmidrule{2-7}
& $v(a_{1})\geq2$ & $v(a_{3})=3$ &  & $3$ & $\left(  \infty,1,\infty,2,=1\right)  $ & $\rm{II}$\\\cmidrule{2-7}
& $v(a_{1})\geq2,$ & $v(P_{R,3})=12$ &  & $4$ & $\left(  \infty,1,=2,3,=3\right)  $ & $\rm{I}_{0}^{\ast}$\\\cmidrule{3-7}
& $v(a_{3})\geq4$ & $v(P_{R,3})\geq13$ & $v(P_{R,4})=4$ & $4$ & $\left(\infty,=1,=2,3,4\right)  $ & $\rm{I}_{1}^{\ast}$\\\cmidrule{4-7}
&  &  & $v(P_{R,4})\geq5$ & $4$ & $\left(  \infty,2,=2,3,4\right)  $ & $\rm{IV}^{\ast}$

 \end{longtable}
\endgroup}

It remains to show that the pairs $\left(  \delta,\delta^{d}\right)  $, $\left(  f,f^{d}\right)  $, and $\left(  c,c^{d}\right)  $ are as given in Table~\ref{tab:localdata-deven}. Since $E$ is given by a \KNT model, a case-by-case analysis with Lemma \ref{Lemmadisccond} shows that $\delta$ and $f$ are as claimed. Since $R^{d}$ has been determined, we obtain $\delta^{d}$ and $f^{d}$ from Table \ref{tab:discriminantconductor}. This case-by-case analysis establishes that the pairs $\left(  \delta,\delta^{d}\right)  $ and $\left(  f,f^{d}\right)  $ are as claimed in Table \ref{tab:localdata-deven}.

We now turn our attention to the computation of the pair $\left(c,c^{d}\right)  $. In fact, it suffices to show that $c^{d}$ is as claimed since Proposition~\ref{lem:pnormal} determines $c$. To this end, we note that $c^{d}$ is uniquely determined if $R^{d}\in\left\{  \rm{I}_{0},\rm{II},\rm{III},\rm{II}^{\ast},\rm{III}^{\ast}\right\}  $. Below, we proceed by considering the cases for a fixed $R$ that correspond to a $R^{d}\not \in \left\{  \rm{I}_{0},\rm{II},\rm{III},\rm{II}^{\ast},\rm{III}^{\ast}\right\}  $. Specifically, for each fixed $R$, we consider the conditions given in Table~\ref{tab:KNtypeEdeven} and the corresponding \KNT model $F_{R,j}^{1}$ and determine $c^{d}$ by means of Corollary~\ref{TamagawaQ2}. In what follows, we denote the Weierstrass coefficients of $F_{R,j}^{1}$ by
$a_{i}^{d}$. We note that this proof also makes reference to polynomials $P_{R,i}^1$, given in Table~\ref{tab:PRj}. Since $v(d)=1$ throughout this proof, we will write $F_{R,j}=F_{R,j}^1$ and $P_{R,i}=P_{R,i}^1$. 

\textbf{Case 1.} Let $R=\rm{I}_{0}$. By Table \ref{tab:KNtypeEdeven}, the local
Tamagawa number $c^{d}$ is not uniquely determined if $\mathcal{V}(E)$ is
either $(=0,0,2,1,=0)$ or $(=0,0,2,=0,1)$. We consider these two cases separately.

\qquad\textbf{Subcase 1a.} Suppose $\mathcal{V}(E)=\left(  =0,0,2,1,=0\right)
$ and consider the Weierstrass model $F_{R,1}$. By Table
\ref{tab:KNtypeEdeven}, $\operatorname*{typ}(E^{d})=\rm{I}_{8}^{\ast}$. Since
$v(a_{1}a_{3}+\frac{1}{4}a_{3}^{2})\geq3$, we deduce that
\begin{align*}
\frac{a_{6}^{d}}{64d^{3}}  & =9a_{1}^{6}+4a_{1}^{4}a_{2}+a_{1}^{3}a_{3}+\frac{a_{3}^{2}}{4}+2a_{1}^{2}a_{4}-\frac{d}{2}\cdot\frac{1+a_{6}^{2}+2a_{6}}{2}+a_{6}\\
& \equiv1+4a_{2}+2a_{4}-\frac{d}{2}\cdot\frac{1+a_{6}^{2}+2a_{6}}{2}+a_{6}\ \operatorname{mod}8.
\end{align*}
It follows from Corollary \ref{TamagawaQ2} that
\[
c^{d}=\left\{
\begin{array}
[c]{cl}
2 & \text{if }v\!\left(  4+16a_{2}+8a_{4}+4a_{6}-d-da_{6}^{2}-2a_{6}d\right)
=4,\\
4 & \text{if }v\!\left(  4+16a_{2}+8a_{4}+4a_{6}-d-da_{6}^{2}-2a_{6}d\right)
\geq5.
\end{array}
\right.
\]

\qquad\textbf{Subcase 1b.} Suppose $\mathcal{V}(E)=\left(  =0,0,2,=0,1\right)
$ and consider the Weierstrass model $F_{R,2}$. By Table
\ref{tab:KNtypeEdeven}, $\operatorname*{typ}(E^{d})=\rm{I}_{8}^{\ast}$. Now observe that
\begin{align*}
\frac{a_{6}^{d}}{64d^{3}}  & =68a_{1}^{6}+16a_{1}^{4}a_{2}+2a_{1}^{3}a_{3}+4a_{1}^{2}a_{4}+\frac{a_{3}^{2}-a_{6}^{2}d}{4}+a_{6}\\
& \equiv\frac{a_{3}^{2}-a_{6}^{2}d}{4}+a_{6}\ \operatorname{mod}8.
\end{align*}
It follows from Corollary \ref{TamagawaQ2} that
\[
c^{d}=\left\{
\begin{array}
[c]{cl}
2 & \text{if }v\!\left(  a_{3}^{2}-a_{6}^{2}d+4a_{6}\right)  =4,\\
4 & \text{if }v\!\left(  a_{3}^{2}-a_{6}^{2}d+4a_{6}\right)  \geq5.
\end{array}
\right.
\]

\textbf{Case 2.} Let $R=\operatorname*{typ}(E)=\mathrm{{I}_{n>0}}$. Then
\[
\mathcal{V}(E)=\left\{
\begin{array}
[c]{cl}%
\left(  =0,0,=\frac{n+1}{2},\frac{n+1}{2},=n\right)   & \text{if }n\text{ is
odd,}\\
\left(  =0,0,\frac{n+2}{2},=\frac{n}{2},n+1\right)   & \text{if }n\text{ is
even.}%
\end{array}
\right.
\]
By Table \ref{tab:KNtypeEdeven}, there are four possible \KNT models
$F_{R,j}$ to consider, which depend on conditions for $P_{R,i}$ for $i=1,2$.
Note that in each case, $\operatorname{typ}(E^{d})=\operatorname{typ}%
(F_{R,j})=\mathrm{{I}_{n+8}^{\ast}}$. In \cite[Theorem5\_2.ipynb]{gittwists},
it is verified that $v(P_{R,1})\geq n+3,$ $v(P_{R,2})\geq n+1$, and
\begin{equation}
a_{6}^{d}=\left\{
\begin{array}
[c]{cl}%
16d^{3}\left(  P_{R,1}-4a_{1}a_{3}^{2}-8a_{3}^{3}-8a_{3}a_{4}\right)   &
\text{if }v(n)=0\text{ and }v(P_{R,1})=n+3,\\
16d^{3}P_{R,1} & \text{if }v(n)=0\text{ and }v(P_{R,1})\geq n+4,\\
64d^{3}a_{1}^{-3}\left(  P_{R,2}-a_{1}^{3}a_{4}^{2}d\right)   & \text{if
}v(n)\geq1\text{ and }v(P_{R,2})=n+1,\\
64d^{3}a_{1}^{-3}P_{R,2} & \text{if }v(n)\geq1\text{ and }v(P_{R,2})\geq n+2.
\end{array}
\right.  \label{Eqna6dIn}%
\end{equation}
We also have the following equalities:%
\begin{align*}
P_{R,3}  & =P_{R,1}-4a_{1}a_{3}^{2}-8a_{3}^{3}-8a_{3}a_{4},\\
P_{R,4}  & =P_{R,2}-a_{1}^{3}a_{4}^{2}d.
\end{align*}
If $v(n)=0$, then $v(4a_{1}a_{3}^{2}+8a_{3}^{3}+8a_{3}a_{4})=n+3$. Thus,
$v(P_{R,1})=n+3$ if and only if $v(P_{R,3})\geq n+4$. Similarly, if
$v(n)\geq1$, then $v(a_{1}^{3}a_{4}^{2}d)=n+1$. Hence, $v(P_{R,2})=n+1$ if and
only if $v(P_{R,4})\geq n+2$. From these facts, we obtain the claimed local
Tamagawa number by Corollary \ref{TamagawaQ2} since%
\[
c^{d}=\left\{
\begin{array}
[c]{cl}%
2 & \text{if }v(a_{6}^{d})=n+11,\\
4 & \text{if }v(a_{6}^{d})\geq n+12.
\end{array}
\right.
\]
It is now easily checked using (\ref{Eqna6dIn}) and the assumptions on
$\mathcal{V}(E)$ that $c^{d}$ is as claimed in Table~\ref{tab:localdata-deven}.

\textbf{Case 3.} Let $R=\rm{II}$, so that $\mathcal{V}(E)=\left(
1,1,1,1,=1\right)  $. We now proceed by cases, following the order in Table~\ref{tab:KNtypeEdeven}, to verify the local Tamagawa number $c^{d}$ for those N\'{e}ron-Kodaira types for which $c^{d}$ is not uniquely determined.

\qquad\textbf{Subcase 3a.} Suppose $v(a_{3})=1$. Then, $F_{R,0}$ is a \KNT
model for $E^{d}$ and by Table \ref{tab:KNtypeEdeven}, $\operatorname*{typ}(F_{R,0})=\rm{I}_{0}^{\ast}$. Since
\[
\frac{4a_{2}^{d}}{d}=a_{1}^{2}+4a_{2}\equiv a_{1}^{2}\ \operatorname{mod}8,
\]
we conclude by Corollary \ref{TamagawaQ2} that
\[
c^{d}=\left\{
\begin{array}
[c]{cl}%
1 & \text{if }v(a_{1})=1,\\
2 & \text{if }v(a_{2})\geq2.
\end{array}
\right.
\]

\qquad\textbf{Subcase 3b.} Suppose $v(a_{3})\geq2$ and $v(a_{1})=v(a_{4})=1$.
Then, $F_{R,1}$ is a \KNT model for $E^{d}$ and by Table
\ref{tab:KNtypeEdeven}, $\operatorname*{typ}(F_{R,1})=\rm{I}_{2}^{\ast}$. Corollary
\ref{TamagawaQ2} now implies that $c^{d}$ is as claimed since
\[
\frac{4a_{6}^{d}}{d^{3}}=a_{3}^{2}+4a_{6}-d^{3}.
\]

\qquad\textbf{Subcase 3c.} Suppose $v(a_{3})\geq2,\ v(a_{1})=1,$ and
$v(a_{4})\geq2$. By Table \ref{tab:KNtypeEdeven}, $F_{R,2}$ (resp. $F_{R,3}$)
are \KNT models for $E^{d}$ if $v(P_{R,1})=4$ (resp. $\geq5$). In both cases, $\operatorname{typ}(E^d)=\rm{I}_3^*$. In
\cite[Theorem5\_2.ipynb]{gittwists}, it is verified that
\[
\frac{4a_{6}^{d}}{d^{3}}=\left\{
\begin{array}
[c]{cl}%
P_{R,2} & \text{if }v(P_{R,1})=4,\\
P_{R,1} & \text{if }v(P_{R,1})\geq5.
\end{array}
\right.
\]
It now follows from Corollary \ref{TamagawaQ2} that $c^{d}$ is as claimed. We
note that the conditions for $c^{d}$ are mutually exclusive since the equality
\[
P_{R,2}=P_{R,1}+4a_{1}^{2}+4a_{1}a_{3}+16a_{2}+8a_{4}+32
\]
implies that $v(P_{R,1})=4$ if and only if $v(P_{R,2})\geq5$. 

\textbf{Case 4.} Let $R=\rm{III}$, so that $\mathcal{V}(E)=\left(
1,1,1,=1,2\right)  $. We now proceed by cases, following the order in Table
\ref{tab:KNtypeEdeven}, to verify the local Tamagawa number $c^{d}$ for those
N\'{e}ron-Kodaira types for which $c^{d}$ is not uniquely determined.

\qquad\textbf{Subcase 4a.} Suppose $v(a_{3})=1$ so that $\mathcal{V}%
(E)=\left(  1,1,=1,=1,2\right)  $. Then, by Table \ref{tab:KNtypeEdeven} ,
$F_{R,0}$ is a \KNT model for $E^{d}$ and $\operatorname*{typ}(E^{d}%
)=\rm{I}_{0}^{\ast}$. By Corollary \ref{TamagawaQ2}, it is verified that $c^{d}$ is
as claimed since $\frac{4a_{2}^{d}}{d}=a_{1}^{2}+4a_{2}\equiv a_{1}%
^{2}\ \operatorname{mod}8$.

\qquad\textbf{Subcase 4b.} Suppose $v(a_{3})\geq2$ and $v(a_{1})=1$. Thus,
$\mathcal{V}(E)=\left(  =1,1,2,=1,2\right)  $. By Table \ref{tab:KNtypeEdeven}
, $F_{R,0}$ is a \KNT model for $E^{d}$ and $\operatorname*{typ}(E^{d})=\rm{I}_{2}^{\ast}$. Since%
\[
\frac{4a_{6}^{d}}{d^{3}}=a_{3}^{2}+4a_{6},
\]
it follows from Corollary \ref{TamagawaQ2} that $c^{d}$ is as claimed.

\textbf{Case 5.} Let $R=\rm{IV}$ so that $\mathcal{V}(E)=\left(
1,1,=1,2,2\right)  $. Then, by Table \ref{tab:KNtypeEdeven}, $F_{R,0}$ is a
\KNT model for $E^{d}$ and $\operatorname*{typ}(E^{d})=\rm{I}_{0}^{\ast}$. By
Corollary \ref{TamagawaQ2}, it is verified that $c^{d}$ is as claimed since
$\frac{4a_{2}^{d}}{d}=a_{1}^{2}+4a_{2}\equiv a_{1}^{2}\ \operatorname{mod}8$.

\textbf{Case 6.} Let $R=\rm{I}_{0}^{\ast}$ so that $\mathcal{V}(E)=\left(1,1,2,3,=3\right)  $. Next, we consider the cases appearing in Table~\ref{tab:KNtypeEdeven} for which the local Tamagawa number $c^{d}$ is not uniquely determined.

\qquad\textbf{Subcase 6a.} Suppose $v(a_{1})=1$ and $v(a_{3})=2$. Then,
$F_{R,1}$ is a \KNT model for $E^{d}$ and $\operatorname{typ}(E^d)=\rm{I}_4^*$. In \cite[Theorem5\_2.ipynb]{gittwists}, it is verified that $\frac{a_{6}^{d}}{4d^{3}}=P_{R,1}$ and that $v(P_{R,1})\geq6$. It follows that $c^{d}$ is as claimed by Corollary~\ref{TamagawaQ2}.

\qquad\textbf{Subcase 6b.} Suppose $v(a_{1})=1$ and $v(a_{3})\geq3$. Thus, $\mathcal{V}(E)=\left(  =1,1,3,3,=3\right)  $. By Table~\ref{tab:KNtypeEdeven},
$F_{R,2}$ (resp. $F_{R,3}$) are \KNT models for $E^{d}$ if $v(P_{R,2})=6$ (resp. $\geq7$). In both cases, $\operatorname*{typ}(E^{d})=\rm{I}_{5}^{\ast}$. In
\cite[Theorem5\_2.ipynb]{gittwists}, it is verified that
\begin{equation}
\frac{4a_{6}^{d}}{d^{3}}=\left\{
\begin{array}
[c]{cl}%
P_{R,3}+16a_{1}^{2}d^{2}+8a_{1}a_{3}d+64a_{2}d^{2}+256d^{3}+16a_{4}d &
\text{if }v(P_{R,2})=6,\\
P_{R,2} & \text{if }v(P_{R,2})\geq7.
\end{array}
\right.  \label{Io0stardev1}
\end{equation}
\qquad Since $v(4a_6)=v(16d)=5$, we obtain that $v(P_{R,3})=v(a_{3}^{2}+4a_{6}-16d)\geq6$. Similarly, as $v(4a_{1}^{2}d^{2})=6$, $v(4a_{1}%
a_{3}d+16a_{2}d^{2}+32d^{3}+8a_{4}d)\geq7$, and
\begin{equation}
P_{R,2}=P_{R,3}+4a_{1}^{2}d^{2}+4a_{1}a_{3}d+16a_{2}d^{2}+32d^{3}%
+8a_{4}d,\label{Io0stardev2}
\end{equation}
we deduce that $v(P_{R,2})\geq6$. In fact, (\ref{Io0stardev2}) implies that
$v(P_{R,2})=6$ if and only if $v(P_{R,3})\geq7$. Next, we observe that%
\[
v(16a_{1}^{2}d^{2}+8a_{1}a_{3}d+64a_{2}d^{2}+256d^{3}+16a_{4}d)\geq8.
\]
As a consequence of the above and (\ref{Io0stardev1}), we conclude that%
\[
v(a_{6}^{d})=\left\{
\begin{array}
[c]{cl}%
8 & \text{if }v(P_{R,2})=7\text{ or }v(P_{R,3})\geq7,\\
\geq9 & \text{if }v(P_{R,2})\geq8\text{ or }v(P_{R,3})\geq8.
\end{array}
\right.
\]
The local Tamagawa number $c^{d}$ now follows from Corollary \ref{TamagawaQ2}.

\qquad\textbf{Subcase 6c.} Suppose $v(a_{1})\geq2$, $v(a_{3})\geq3$,
$v(a_{3}^{2}d+4a_{6}d-64)\geq8$, and $v(a_{4}-4a_{2})\geq4$. By Table
\ref{tab:KNtypeEdeven}, $F_{R,4}$ is a \KNT model for $E^{d}$ and
$\operatorname*{typ}(E^{d})=\rm{IV}$. In \cite[Theorem5\_2.ipynb]{gittwists}, it is
verified that%
\[
\frac{256a_{6}^{d}}{d^{2}}=a_{3}^{2}d+4a_{6}d-64.
\]
It now follows that $c^{d}$ is as claimed by Corollary \ref{TamagawaQ2}.

\textbf{Case 7.} Let $R=\rm{I}^*_{n>0}$. Since $E$ is given by a \KNT model, we have that
\[
\mathcal{V}(E)=\left\{
\begin{array}
[c]{cl}%
\left(  1,=1,=\frac{n+3}{2},\frac{n+5}{2},n+3\right)  & \text{if }n\text{ is
odd,}\\
\left(  1,=1,\frac{n+4}{2},=\frac{n+4}{2},n+3\right)  & \text{if }n\text{ is
even.}%
\end{array}
\right.
\]

We now proceed by cases, following the order in Table~\ref{tab:KNtypeEdeven} for those cases for which the local Tamagawa number $c^d$ is not uniquely determined.

\qquad\textbf{Subcase 7a.} Suppose $n=1$ and $v(a_{1})=1$. By Table~\ref{tab:KNtypeEdeven}, $F_{R,1}$ is a \KNT model for $E^{d}$ and $\operatorname*{typ}(E^{d})=I_{4}^{\ast}$. In \cite[Theorem5\_2.ipynb]{gittwists}, it is verified that $v(P_{R,1,})\geq8$ and
\[
\frac{16a_{6}^{d}}{d^{3}}=P_{R,1,}.
\]
Consequently, $c^{d}$ is as claimed by Corollary \ref{TamagawaQ2}.

\qquad\textbf{Subcase 7b.} Suppose $n=2$ and $v(a_{1})\geq2$. By Table
\ref{tab:KNtypeEdeven},
\begin{equation}
\operatorname*{typ}(E^{d})=\left\{
\begin{array}
[c]{cl}%
\rm{I}_{0}^{\ast} & \text{if }v(P_{R,3})=4,\\
\rm{I}_{1}^{\ast} & \text{if }v(P_{R,3})\geq5,v(P_{R,4})=5,\\
\rm{IV}^{\ast} & \text{if }v(P_{R,3})\geq5,v(P_{R,4})\geq6,
\end{array}
\right.  \label{I4stareq1}
\end{equation}
and $F_{R,3}$ is a \KNT model for $E^{d}$ in each of the three cases
appearing in (\ref{I4stareq1}). In \cite[Theorem5\_2.ipynb]{gittwists}, it is
verified that $16a_{2}^{d}=P_{R,4}$ and $256a_{6}^{d}d^{-3}=P_{R,5}$. The
local Tamagawa number $c^{d}$ now follows from Corollary \ref{TamagawaQ2}.

\qquad\textbf{Subcase 7c.} Suppose $n=5$ and $v(a_{1})\geq2$. By Table
\ref{tab:KNtypeEdeven}, $F_{R,2}$ is a \KNT model for $E^{d}$ and
$\operatorname*{typ}(E^{d})=\rm{I}_{0}^{\ast}$. The local Tamagawa number $c^{d}$
now follows from Corollary \ref{TamagawaQ2} since $16a_{2}^{d}d^{-1}=a_{1}^{2}+4a_{2}-4d$.

\qquad\textbf{Subcase 7d.} Suppose $n\geq2$ and that $v(a_{1})=1$. By Table~\ref{tab:KNtypeEdeven}, there are four possible \KNT models $F_{R,j}$ to
consider, which depend on conditions for $P_{R,i}$ for $i=1,8$. In each case, $\operatorname*{typ}(E^{d})=I_{n+4}^{\ast}$. In \cite[Theorem5\_2.ipynb]%
{gittwists}, it is verified that $v(P_{R,1}),v(P_{R,8})\geq n+7$ and
\[
\frac{16a_{6}^{d}}{d^{3}}=\left\{
\begin{array}
[c]{cl}
P_{R,8}-8a_{1}a_{3}^{2}-4a_{3}^{3}-16a_{3}a_{4} & \text{if }v(n)=0\text{ and
}v(P_{R,8})=n+7,\\
P_{R,8} & \text{if }v(n)=0\text{ and }v(P_{R,8})\geq n+8,\\
P_{R,1}-4a_{4}^{2}d & \text{if }v(n)\geq1\text{ and }v(P_{R,1})=n+7,\\
P_{R,1} & \text{if }v(n)\geq1\text{ and }v(P_{R,1})\geq n+8.
\end{array}
\right.
\]
We also have the following equalities:
\begin{align*}
P_{R,9}  & =P_{R,8}-8a_{1}a_{3}^{2}-4a_{3}^{3}-16a_{3}a_{4},\\
P_{R,7}  & =P_{R,1}-4a_{4}^{2}d.
\end{align*}
If $v(n)=0$, then $v(8a_{1}a_{3}^{2}+4a_{3}^{3}+16a_{3}a_{4})=n+7$. Thus,
$v(P_{R,8})=n+7$ if and only if $v(P_{R,9})\geq n+8$. Similarly, if
$v(n)\geq1$, then $v(4a_{4}^{2}d)=n+7$. Hence, $v(P_{R,1})=n+7$ if and only if
$v(P_{R,7})\geq n+8$. From these facts, we obtain the claimed local Tamagawa
number by Corollary \ref{TamagawaQ2} since%
\[
c^{d}=\left\{
\begin{array}
[c]{cl}%
2 & \text{if }v(a_{6}^{d})=n+7,\\
4 & \text{if }v(a_{6}^{d})\geq n+8.
\end{array}
\right.
\]

\qquad\textbf{Subcase 7e.} Suppose $n\geq6,$ $v(a_{1})\geq2$, and $v(a_{1}%
^{2}+4a_{2}-4d)=4$. We note that $P_{R,6}=a_{1}^{2}+4a_{2}-4d$. By Table~\ref{tab:KNtypeEdeven}, there are two \KNT models $F_{R,j}$ to consider, depending on whether $v(n)=0$ or $v(n)\geq1$. In each case, $\operatorname*{typ}(E^{d})=\rm{I}_{n-4}^{\ast}$. In \cite[Theorem5\_2.ipynb]{gittwists}, it is verified that
\[
a_{6}^{d}=\left\{
\begin{array}
[c]{cl}%
\frac{1}{1024}P_{R,10} & \text{if }v(n)=0,\\
\frac{-d^{3}}{256}P_{R,2} & \text{if }v(v)\geq1.
\end{array}
\right.
\]
Loc. cit. also verifies that $v(P_{R,10})\geq n+9$ (resp. $v(P_{R,2})\geq
n+4$) if $v(n)=0$ (resp. $\geq1$). It is now verified that $c^{d}$ is as
claimed from Corollary \ref{TamagawaQ2} since
\[
c^{d}=\left\{
\begin{array}
[c]{cl}%
2 & \text{if }v(a_{6}^{d})=n-1,\\
4 & \text{if }v(a_{6}^{d})\geq n.
\end{array}
\right.
\]

\qquad\textbf{Subcase 7f}. Suppose $n\geq9,$ $v(a_{1})\geq2$, and $v(a_{1}^{2}+4a_{2}-4d)\geq5$. By Table~\ref{tab:KNtypeEdeven}, $F_{R,9}$ is a \KNT
model for $E^{d}$ and $\operatorname*{typ}(E^{d})=\rm{I}_{n-8}$. In \cite[Theorem5\_2.ipynb]{gittwists}, it is verified that
\[
64a_{2}^{d}=a_{2}d+4a_{2}d+6a_{3}-16
\]
and that $v(64a_{2}^{d})\geq1$. By assumption, $v(6a_{3})\geq7$. Consequently, by Corollary \ref{TamagawaQ2} we conclude that
\[
c^{d}=\left\{
\begin{array}
[c]{cl}%
2-(n\ \operatorname{mod}2) & \text{if }v(a_{2}d+4a_{2}d-16)=6,\\
n-8 & \text{if }v(a_{2}d+4a_{2}d-16)\geq7.
\end{array}
\right.
\]

\textbf{Case 8.} Let $R=\rm{IV}^{\ast}$ so that $\mathcal{V}(E)=\left(
1,2,=2,3,4\right)
$. Next, we consider the case appearing in Table~\ref{tab:KNtypeEdeven} for which the local Tamagawa number $c^{d}$ is not uniquely determined.

\qquad\textbf{Subcase 8a.} Suppose $v(a_{1})=1$. Then,
$F_{R,1}$ is a \KNT model for $E^{d}$ and $\operatorname{typ}(E^d)=\rm{I}_4^*$. In \cite[Theorem5\_2.ipynb]{gittwists}, it is verified that $\frac{a_{6}^{d}}{4d^{3}}=P_{R,1}$ and that $v(P_{R,1})\geq6$. It follows that $c^{d}$ is as claimed by Corollary~\ref{TamagawaQ2}.

\textbf{Case 9.}
Let $R=\rm{III}^{\ast}$ so that $\mathcal{V}(E)=\left(
1,2,=2,3,4\right)
$. Next, we consider the case appearing in Table~\ref{tab:KNtypeEdeven} for which the local Tamagawa number $c^{d}$ is not uniquely determined.

\qquad\textbf{Subcase 9a.} Suppose $v(a_{1})=1$. By Table \ref{tab:KNtypeEdeven}, $F_{R,1}$ (resp. $F_{R,2}$)
are \KNT models for $E^{d}$ if $v(P_{R,1})=9$ (resp. $\geq10$).
In both cases, $\operatorname{typ}(E^d)=\rm{I}_6^*$. In
\cite[Theorem5\_2.ipynb]{gittwists}, it is verified that
\[
\frac{4a_{6}^{d}}{d^{3}}=\left\{
\begin{array}
[c]{cl}%
P_{R,2} & \text{if }v(P_{R,1})=9,\\
P_{R,1} & \text{if }v(P_{R,1})\geq10.
\end{array}
\right.
\]
It now follows from Corollary \ref{TamagawaQ2} that $c^{d}$ is as claimed. We
note that the conditions for $c^{d}$ are mutually exclusive since the equality
\[
P_{R,2}=P_{R,1}-256d
\]
implies that $v(P_{R,1})=9$ if and only if $v(P_{R,2})\geq10$.

\textbf{Case 10.}
Let $R=\rm{II}^{\ast}$ so that $\mathcal{V}(E)=\left(
 1,2,3,4,=5\right)
$. Next, we consider the case appearing in Table~\ref{tab:KNtypeEdeven} for which the local Tamagawa number $c^{d}$ is not uniquely determined.

\qquad\textbf{Subcase 10a.} Suppose $v(a_{1})=1$. By Table \ref{tab:KNtypeEdeven}, $F_{R,1}$ (resp. $F_{R,2}$)
are \KNT models for $E^{d}$ if $v(P_{R,1})=10$ (resp. $\geq11$).
In both cases, $\operatorname{typ}(E^d)=\rm{I}_7^*$. In
\cite[Theorem5\_2.ipynb]{gittwists}, it is verified that
\[
\frac{4a_{6}^{d}}{d^{3}}=\left\{
\begin{array}
[c]{cl}%
P_{R,2} & \text{if }v(P_{R,1})=10,\\
P_{R,1} & \text{if }v(P_{R,1})\geq11.
\end{array}
\right.
\]
It now follows from Corollary \ref{TamagawaQ2} that $c^{d}$ is as claimed. We
note that the conditions for $c^{d}$ are mutually exclusive since the equality
\[
P_{R,2}=P_{R,1}+ 16d(a_{1}^{2} a_{3} + 4 a_{1}^{2} d + 2 a_{1} a_{3} + 4 a_{2} a_{3} + 3 a_{3}^{2} + 16 a_{2} d + 24 a_{3} d + 64 d^{2} + 4 a_{4})
\]
implies that $v(P_{R,1})=10$ if and only if $v(P_{R,2})\geq11$.

\qquad\textbf{Subcase 10b.} Suppose $v(a_{1})\geq 2$ and $v(a_{3})\geq 4$.  By Table
\ref{tab:KNtypeEdeven},
\begin{equation}
\operatorname*{typ}(E^{d})=\left\{
\begin{array}
[c]{cl}%
\rm{I}_{0}^{\ast} & \text{if }v(P_{R,3})=12,\\
\rm{I}_{1}^{\ast} & \text{if }v(P_{R,3})\geq13,v(P_{R,4})=4,\\
\rm{IV}^{\ast} & \text{if }v(P_{R,3})\geq13,v(P_{R,4})\geq5,
\end{array}
\right.  \label{IIstareq1}
\end{equation}
and $F_{R,4}$ is a \KNT model for $E^{d}$ in each of the three cases
appearing in (\ref{IIstareq1}). In \cite[Theorem5\_2.ipynb]{gittwists}, it is
verified that $16a_{2}^{d}=P_{R,4}$ and $256a_{6}^{d}d^{-3}=P_{R,5}$. The
local Tamagawa number $c^{d}$ now follows from Corollary \ref{TamagawaQ2}.
\end{proof}

\vspace{0.2in}
\noindent 
\textbf{Acknowledgements.} This paper began as a research collaboration through the workshop Rethinking Number Theory: 2021, which was organized by Heidi Goodson, Allechar Serrano L\'{o}pez, Christelle Vincent, and Mckenzie West. The authors extend their sincere thanks to the workshop organizers, without which this paper would not have been written. We also thank Joseph H. Silverman for suggesting the name \textit{strongly-minimal model}.
During the carrying out of this work, A. Barrios was supported through a research grant from the University of St. Thomas and an AMS Simons Research Enhancement Grant. H. Wiersema was supported by the Herchel Smith Postdoctoral Fellowship Fund, and the Engineering and Physical Sciences Research Council (EPSRC) grant EP/W001683/1. We also thank the anonymous referees for reading our manuscript very carefully and providing us with valuable comments and suggestions.

\newpage
\section*{Appendix}
\label{AppendixTables}

{\setlength{\tabcolsep}{12pt}
\renewcommand{\arraystretch}{1.15} 
\begin{longtable}{ccccccc}
	\caption{The isomorphism $\left[  u_{j},r_{j},s_{j},w_{j}\right]  $ from $E^{d}$ onto
$F_{R,j}^{v(d)}$. Note that for each $\typ(E)=R$, the model $F_{R,0}^{v(d)}$ is obtained from $E^d$ via the isomorphism $\left[  u_{0},r_{0},s_{0},w_{0}\right]  $.}\\
	\midrule
$v(d)$ & $R$ & $j$ & $u_{j}$ & $r_{j}$ & $s_{j}$ &
$w_{j}$\\
	\midrule
	\endfirsthead
	\caption[]{\emph{continued}}\\
	\midrule
$v(d)$ & $R$ & $j$ & $u_{j}$ & $r_{j}$ & $s_{j}$ &
$w_{j}$\\
	\midrule
	\endhead
	\midrule
	\multicolumn{2}{r}{\emph{continued on next page}}
	\endfoot
	\midrule
	\endlastfoot
0 & $R$ & $0$ & $2$ & $0$ & $a_{1}$ & $4a_{3}$\\\cmidrule{2-7}
& $\rm{I}_{0}$ & $1$ & $2$ & $0$ & $a_{1}$ & $0$\\\cmidrule{3-7}
&  & $2$ & $2$ & $0$ & $a_{1}$ & $4$\\\cmidrule{3-7}
&  & $3$ & $1$ & $0$ & $a_{1}$ & $0$\\\cmidrule{3-7}
&  & $4$ & $1$ & $0$ & $a_{1}$ & $8$\\\cmidrule{3-7}
&  & $5$ & $1$ & $0$ & $0$ & $4$\\\cmidrule{2-7}
 & $\rm{I}_{n>0}$ & $1$ & $1$ & $0$ & $a_{1}$ & $4a_{3}$\\\cmidrule{2-7}
 
 & $\rm{I}^*_{n>0}$ & $1$ & $2$ & $2a_3$ & $0$ & $4a_{3}$\\\cmidrule{3-7}
 & & $2$ & $4$ & $0$ & $a_1$ & $0$\\\cmidrule{3-7}
 & & $3$ & $4$ & $0$ & $a_1$ & $4a_3$\\\cmidrule{2-7}

& $\rm{II}^{\ast}$ & $1$ & $2$ & $a_1^4$ & $a_{1}$ & $a_6$\\\cmidrule{3-7}
&  & $2$ & $4$ & $a_1^2 d + 4a_2 d$ & $0$ & $4a_{3}$\\\hline

$1$ & $R$ & $0$ & $2$ & $0$ & $0$ & $0$\\\cmidrule{2-7}
& $\rm{I}_{0}$ & $1$ & $1$ & $8a_{1}^{2}d$ & $0$ & $4(a_{6}+1)d^{2}$\\\cmidrule{3-7}
&  & $2$ & $1$ & $16a_{1}^{2}d$ & $0$ & $4a_{6}d^{2}$\\\cmidrule{2-7}


& $\rm{I}_{n>0}$ & $1$ & $1$ & $-4a_{3}d$ & $0$ & $4a_{3}d^{2}$\\\cmidrule{3-7}
&  & $2$ & $1$ & $4a_{3}d$ & $0$ & $4a_{3}d^{2}$\\\cmidrule{3-7}
&  & $3$ & $1$ & $-4a_{1}^{-1}a_{3}d$ & $0$ & $8a_{4}d^{2}$\\\cmidrule{3-7}
&  & $4$ & $1$ & $-4a_{1}^{-1}a_{3}d$ & $0$ & $0$\\\cmidrule{2-7}


& $\rm{II}$ & $1$ & $2$ & $0$ & $0$ & $4d^{3}$\\\cmidrule{3-7}
&  & $2$ & $2$ & $8d$ & $0$ & $8d^{2}$\\\cmidrule{3-7}
&  & $3$ & $2$ & $0$ & $a_{1}-d$ & $8d^{2}$\\\cmidrule{3-7}
&  & $4$ & $4$ & $a_{2}^{3}d$ & $0$ & $4d^{3}$\\\cmidrule{2-7}


& $\rm{I}_{0}^{\ast}$ & $1$ & $2$ & $4d^{2}$ & $0$ & $8a_{2}d^{2}$\\\cmidrule{3-7}
&  & $2$ & $2$ & $16d^{2}$ & $0$ & $16d^{2}$\\\cmidrule{3-7}
&  & $3$ & $2$ & $8d^{2}$ & $0$ & $16d^{2}$\\\cmidrule{3-7}
&  & $4$ & $4$ & $0$ & $2a_{2}$ & $32d$\\\cmidrule{2-7}


& $\rm{I}_{n>0}^{\ast}$ & $1$ & $2$ & $2a_{3}d$ & $0$ & $0$\\\cmidrule{3-7}
&  & $2$ & $4$ & $0$ & $2d$ & $4a_{3}d^{2}$\\\cmidrule{3-7}
&  & $3$ & $4$ & $16d$ & $4$ & $32d^{2}$\\\cmidrule{3-7}
&  & $4$ & $8$ & $0$ & $2d$ & $0$\\\cmidrule{3-7}
&  & $5$ & $2$ & $2a_{3}d$ & $0$ & $4a_{4}d^{2}$\\\cmidrule{3-7}
&  & $6$ & $2$ & $-2a_{3}d$ & $0$ & $4a_{3}d^{2}$\\\cmidrule{3-7}
&  & $7$ & $2$ & $2a_{3}d$ & $0$ & $4a_{3}d^{2}$\\\cmidrule{3-7}
&  & $8$ & $4$ & $2a_{3}$ & $4$ & $8a_{3}$\\\cmidrule{3-7}
&  & $9$ & $8$ & $2a_{3}$ & $4$ & $8a_{3}$\\\cmidrule{2-7}

& $\rm{IV}^{\ast}$ & $1$ & $2$ & $4d^{2}$ & $0$ & $16d^{2}$\\\cmidrule{2-7}

& $\rm{III}^{\ast}$ & $1$ & $2$ & $2a_{3}d$ & $0$ & $32d^{2}$\\\cmidrule{3-7}
&  & $2$ & $2$ & $2a_{3}d$ & $0$ & $0$\\\cmidrule{3-7}
&  & $3$ & $4$ & $0$ & $0$ & $0$\\\cmidrule{2-7}

& $\rm{II}^{\ast}$ & $1$ & $2$ & $2a_{3}d+16d^{2}$ & $0$ & $32d^{2}$\\\cmidrule{3-7}
&  & $2$ & $2$ & $2a_{3}d$ & $0$ & $32d^{2}$\\\cmidrule{3-7}
&  & $3$ & $4$ & $0$ & $0$ & $0$\\\cmidrule{3-7}
&  & $4$ & $4$ & $a_{4}d$ & $0$ & $32d^{2}$
\label{FRj}	
\end{longtable}}

{\setlength{\tabcolsep}{6pt}
\renewcommand{\arraystretch}{1.15} 
\begin{longtable}{cccC{4.7in}}
	\caption{The polynomials $P_{R,j}^{v(d)}$ to determine Tamagawa
number of $E^{d}$}\\
	\midrule
$v(d)$ & $R$ & $j$ & $P_{R,j}^{v(d)}$\\
	\midrule
	\endfirsthead
	\caption[]{\emph{continued}}\\
	\midrule
$v(d)$ & $R$ & $j$ & $P_{R,j}^{v(d)}$\\
	\midrule
	\endhead
	\midrule
	\multicolumn{3}{r}{\emph{continued on next page}}
	\endfoot
	\midrule
	\endlastfoot
$0$  & $\rm{I}_{n>0}$ & $1$ & $2^{n-1}(d-1)+a_{6}d$\\
&  & $2$ & $a_{3}^{2}+2a_{6}$\\\cmidrule{2-4}
& $\rm{I}_{0}^{\ast}$ & $1$ & $4(d-1)+a_{6}d$\\\cmidrule{2-4}

 & $\rm{I}_{n>0}^*$ & $1$ & $2^{n+1}(d-1)+a_{6}$\\
 &  & $2$ & $2^{n+2}(d-1+a_{2}d)+a_{3}a_{4}+2a_{6}+2^{n+1}a_{3}$ \\\cmidrule{2-4}

& $\rm{II}^{\ast}$ & $1$ & $48+16a_{2}d+4a_{4}+16d+a_{6}d$\\\hline
$1$ & $\rm{I}_{0}$ & $1$ & $4+16a_{2}+8a_{4}+4a_{6}-d-da_{6}^{2}-2a_{6}d$\\\cmidrule{3-4}
&  & $2$ & $a_{3}^{2}-a_{6}^{2}d+4a_{6}$\\\cmidrule{2-4}


& $\rm{I}_{n>0}$ & $1$ & $a_{1}^{2}a_{3}^{2}+2a_{1}a_{3}^{2}+4a_{2}a_{3}^{2}+4a_{3}^{3}-a_{3}^{2}d+a_{3}^{2}+4a_{3}a_{4}+4a_{6}$\\\cmidrule{3-4}
&  & $2$ & $a_{1}a_{2}a_{3}^{2}-a_{1}^{2}a_{3}a_{4}+a_{1}^{3}a_{6}-a_{3}^{3}$\\\cmidrule{3-4}

&  & $3$ & $a_{1}^{2} a_{3}^{2} - 2 a_{1} a_{3}^{2} + 4 a_{2} a_{3}^{2} - 4 a_{3}^{3} - a_{3}^{2} d + a_{3}^{2} - 4 a_{3} a_{4} + 4 a_{6}$ \\\cmidrule{3-4}

&  & $4$ & $ a_{1} a_{2} a_{3}^{2} -a_{1}^{3} a_{4}^{2} d - a_{1}^{2} a_{3} a_{4} + a_{1}^{3} a_{6} - a_{3}^{3}$ \\\cmidrule{2-4}


& $\rm{II}$ & $1$ & $a_{3}^{2}+4a_{6}-4d$\\\cmidrule{3-4}
&  & $2$ & $4a_{1}^{2}+4a_{1}a_{3}+a_{3}^{2}+16a_{2}+8a_{4}+4a_{6}-4d+32$\\\cmidrule{2-4}


& $\rm{I}_{0}^{\ast}$ & $1$ & $a_{1}^{2}d^{2}-4a_{2}^{2}d+2a_{1}a_{3}d+4a_{2}d^{2}+4d^{3}+a_{3}^{2}+4a_{4}d+4a_{6}$\\\cmidrule{3-4}
&  & $2$ & $4a_{1}^{2}d^{2}+4a_{1}a_{3}d+16a_{2}d^{2}+32d^{3}+a_{3}^{2}+8a_{4}d+4a_{6}-16d$\\\cmidrule{3-4}
&  & $3$ & $a_{3}^{2}+4a_{6}-16d$\\\cmidrule{3-4}
&  & $4$ & $a_{3}^{2}d+4a_{6}d-64$\\\cmidrule{3-4}
&  & $5$ & $a_{4}-4a_{2}$\\\cmidrule{2-4}


& $\rm{I}_{n>0}^{\ast}$ & $1$ & $a_{1}^{2}a_{3}^{2}+4a_{1}a_{3}^{2}+4a_{2}a_{3}
^{2}+2a_{3}^{3}+4a_{3}^{2}+8a_{3}a_{4}+16a_{6}$\\\cmidrule{3-4}
&  & $2$ & $a_{3}^{2}d-a_{3}^{2}-4a_{6}$\\\cmidrule{3-4}
&  & $3$ & $4a_{1}+4a_{2}+a_{3}-4d$\\\cmidrule{3-4}
&  & $4$ & $a_{1}^{2}d+4a_{2}d+48d-16$\\\cmidrule{3-4}
&  & $5$ & $16a_{1}^{2}+8a_{1}a_{3}+a_{3}^{2}+64a_{2}+16a_{4}+4a_{6}-64d+256$\\\cmidrule{3-4}
&  & $6$ & $a_{1}^{2}+4a_{2}-4d$\\\cmidrule{3-4}
&  & $7$ & $a_{1}^{2}a_{3}^{2}+4a_{1}a_{3}^{2}+4a_{2}a_{3}^{2}+2a_{3}
^{3}+4a_{3}^{2}+8a_{3}a_{4}+16a_{6}-4a_{4}^{2}d$\\\cmidrule{3-4}
&  & $8$ & $a_{1}^{2}a_{3}^{2}+4a_{1}a_{3}^{2}+4a_{2}a_{3}^{2}+2a_{3}
^{3}-4a_{3}^{2}d+4a_{3}^{2}+8a_{3}a_{4}+16a_{6}$\\\cmidrule{3-4}
&  & $9$ & $a_{1}^{2}a_{3}^{2}-4a_{1}a_{3}^{2}+4a_{2}a_{3}^{2}-2a_{3}^{3}-4a_{3}^{2}d+4a_{3}^{2}-8a_{3}a_{4}+16a_{6}$\\\cmidrule{3-4}
&  & $10$ & $a_{1}^{2}a_{3}^{2}d+4a_{1}a_{3}^{2}d^{2}+4a_{3}^{2}d^{3}+4a_{2}a_{3}^{2}d+8a_{3}a_{4}d^{2}+16a_{6}d^{3}+2a_{3}^{3}-16a_{3}^{2}$\\\cmidrule{2-4}

& $\rm{IV}^{\ast}$ & $1$ & $a_{1}^{2} d^{2} + 2 a_{1} a_{3} d + 4 a_{2} d^{2} + 4 d^{3} + a_{3}^{2} + 4 a_{4} d + 4 a_{6} - 16 d$\\\cmidrule{2-4}
& $\rm{III}^{\ast}$ & $1$ & $a_{1}^{2} a_{3}^{2} + 4 a_{1} a_{3}^{2} + 4 a_{2} a_{3}^{2} + 2 a_{3}^{3} + 4 a_{3}^{2} + 8 a_{3} a_{4} + 16 a_{6}$\\\cmidrule{3-4}
&  & $2$ & $a_{1}^{2} a_{3}^{2} + 4 a_{1} a_{3}^{2} + 4 a_{2} a_{3}^{2} + 2 a_{3}^{3} + 4 a_{3}^{2} + 8 a_{3} a_{4} + 16 a_{6}-256 d$\\\cmidrule{2-4}
& 
$\rm{II}^{\ast}$ & $1$ & $a_1^2a_3^2 + 4a_1a_3^2 + 4a_2a_3^2 + 2a_3^3 + 4a_3^2 + 8a_3a_4 + 16a_6 - 256d$ \\\cmidrule{3-4}
&  & $2$ & 
$P_{\rm{II}^{\ast},1}^1+16d(a_{1}^{2} a_{3} + 4 a_{1}^{2} d + 2 a_{1} a_{3} + 4 a_{2} a_{3} + 3 a_{3}^{2} + 16 a_{2} d + 24 a_{3} d + 64 d^{2} + 4 a_{4})$ 
\\\cmidrule{3-4}
&  & $3$ & $a_{1}^{2} a_{4}^{2} + 8 a_{1} a_{3} a_{4} + 4 a_{2} a_{4}^{2} + a_{4}^{3} + 16 a_{3}^{2} + 16 a_{4}^{2} + 64 a_{6} - 1024 d$\\\cmidrule{3-4}
&  & $4$ & $a_{1}^{2} + 4 a_{2} + 3 a_{4}$
\label{tab:PRj}	
\end{longtable}}

\bibliographystyle{amsalpha}
\bibliography{LocalDataTwists}

\newcommand{\etalchar}[1]{$^{#1}$}
\providecommand{\bysame}{\leavevmode\hbox to3em{\hrulefill}\thinspace}
\providecommand{\MR}{\relax\ifhmode\unskip\space\fi MR }
\providecommand{\MRhref}[2]{%
  \href{http://www.ams.org/mathscinet-getitem?mr=#1}{#2}
}
\providecommand{\href}[2]{#2}
\begin{thebibliography}{BRS{\etalchar{+}}25}

\bibitem[BRS{\etalchar{+}}25]{gittwists}
Alexander~J. Barrios, Manami Roy, Nandita Sahajpal, Darwin Tallana, Bella
  Tobin, and Hanneke Wiersema, \emph{Computations for twists},
  \url{https://github.com/manamiroy/Twists}, 2025.

\bibitem[BSD65]{BSD}
B.~J. Birch and H.~P.~F. Swinnerton-Dyer, \emph{Notes on elliptic curves.
  {II}}, J. Reine Angew. Math. \textbf{218} (1965), 79--108. \MR{179168}

\bibitem[Com94]{Comalada}
Salvador Comalada, \emph{Twists and reduction of an elliptic curve}, J. Number
  Theory \textbf{49} (1994), no.~1, 45--62. \MR{1295951}

\bibitem[CS23]{CromonaSadek2020}
John~E. Cremona and Mohammad Sadek, \emph{Local and global densities for
  {W}eierstrass models of elliptic curves}, Math. Res. Lett. \textbf{30}
  (2023), no.~2, 413--461. \MR{4649636}

\bibitem[Dev24]{sagemath}
The~Sage Developers, \emph{Sagemath, the {S}age {M}athematics {S}oftware
  {S}ystem (version 9.6)}, 2024, {\tt http://www.sagemath.org}.

\bibitem[N\'64]{Neron1964}
Andr\'{e} N\'{e}ron, \emph{Mod\`eles minimaux des vari\'{e}t\'{e}s
  ab\'{e}liennes sur les corps locaux et globaux}, Inst. Hautes \'{E}tudes Sci.
  Publ. Math. (1964), no.~21, 128. \MR{179172}

\bibitem[Ogg67]{Ogg}
A.~P. Ogg, \emph{Elliptic curves and wild ramification}, Amer. J. Math.
  \textbf{89} (1967), 1--21. \MR{207694}

\bibitem[Pal12]{Pal}
Vivek Pal, \emph{Periods of quadratic twists of elliptic curves}, Proc. Amer.
  Math. Soc. \textbf{140} (2012), no.~5, 1513--1525, With an appendix by Amod
  Agashe. \MR{2869136}

\bibitem[Sil84]{SilMin}
Joseph~H. Silverman, \emph{Weierstrass equations and the minimal discriminant
  of an elliptic curve}, Mathematika \textbf{31} (1984), no.~2, 245--251.
  \MR{804199}

\bibitem[Sil94]{Silverman1994}
\bysame, \emph{Advanced topics in the arithmetic of elliptic curves}, Graduate
  Texts in Mathematics, vol. 151, Springer-Verlag, New York, 1994. \MR{1312368}

\bibitem[Sil09]{Silverman2009}
\bysame, \emph{The arithmetic of elliptic curves}, second ed., Graduate Texts
  in Mathematics, vol. 106, Springer, Dordrecht, 2009. \MR{2514094}

\bibitem[Tat75]{Tate1975}
J.~Tate, \emph{Algorithm for determining the type of a singular fiber in an
  elliptic pencil}, Modular functions of one variable, {IV} ({P}roc.
  {I}nternat. {S}ummer {S}chool, {U}niv. {A}ntwerp, {A}ntwerp, 1972), 1975,
  pp.~33--52. Lecture Notes in Math., Vol. 476. \MR{0393039}

\end{thebibliography}
\end{document}